\documentclass{article}

\usepackage{theorem}
\usepackage{enumitem,multicol}
\usepackage{amsmath,mathtools, amssymb,amsfonts,graphicx,float}
\usepackage[hang,small,bf]{caption}
\usepackage[subrefformat=parens]{subcaption}
\def\BBox{\rule{2mm}{3mm}}
\def\QED{\hfill$\BBox$}
\newenvironment{proof}
{\begin{rm}\par\smallskip\noindent{\bf Proof.}\quad}{\QED\end{rm}}

\theorembodyfont{\rmfamily}
\newtheorem{thm}{Theorem}[section]
\newtheorem{lem}[thm]{Lemma}        %
\newtheorem{remark}[thm]{\bfseries Remark}    %

\newtheorem{prop}[thm]{\bfseries Proposition} %
\newtheorem{cor}[thm]{\bfseries Corollary}

\newtheorem{defn}[thm]{\bfseries Definition}

\numberwithin{equation}{section}

\title{Complete combinatorial characterization \\ of greedy-drawable trees}

\author{Hiroyuki Miyata \\ Faculty of Informatics, \\ Gunma University \\ \texttt{hmiyata@gunma-u.ac.jp} \and Reiya Nosaka \\ Faculty of Informatics, \\ Gunma University \\ \texttt{t201d062@gunma-u.ac.jp}}

\date{}

\begin{document}
\maketitle            
\thispagestyle{empty} 

\begin{abstract}
A (Euclidean) greedy drawing of a graph is a drawing in which, for every pair of vertices $s,t$ ($s \neq t$),
there is a neighboring vertex of $s$ that is closer to $t$ than $s$ in the Euclidean distance.   
Greedy drawings are crucial in the context of message routing in networks, and graph classes that admit greedy drawings have been actively investigated.
N\"{o}llenburg and Prutkin (Discrete Comput.\ Geom.,\ 58(3), pp.~543--579, 2017) characterized greedy-drawable trees in terms of an inequality system containing a non-linear equation. 
Using the characterization, they designed a linear-time recognition algorithm for greedy-drawable trees with maximum degree $\leq 4$.
However, the possibility of a combinatorial characterization of  greedy-drawable trees with maximum degree 5 was left open.
In this paper, we present a combinatorial characterization of greedy-drawable trees with maximum degree $5$, which leads to 
a complete combinatorial characterization of greedy-drawable trees.
Furthermore, we present a characterization of greedy-drawable pseudo-trees.
\end{abstract}

\section{Introduction}\label{Sec: Intro}
\emph{Geographic routing} (or \emph{geometric routing}) is a type of routing that uses the geographic coordinates of the nodes as addresses for the purpose of routing.
One of the simplest routing algorithms is \emph{greedy routing}, in which each node simply forwards a message to its neighbor that is closest to the destination.
In (pure) greedy routing, however, a message delivery fails if the message is forwarded to a node with no neighbor that is closer to the destination.
A graph drawing in which greedy routing is guaranteed to work is called a \emph{greedy drawing} (or \emph{greedy embedding})~\cite{PR05}.
That is,  a (Euclidean) greedy drawing of a graph $G$ is a drawing of $G$ in which for every pair of vertices $s,t$ ($s \neq t$) of $G$,
there is a neighboring vertex $u$ of $s$ with $d(s,t) > d(u,t)$, where $d$ is the Euclidean distance. 
This notion  is motivated by work of Rao et al.~\cite{RRPSS03}, who proposed the idea of applying greedy routing based on \emph{virtual coordinates} (i.e., a graph drawing) rather than geographic coordinates,
which makes greedy routing  applicable even in the absence of locational information.
Through extensive experiments, the authors showed that their approach makes greedy routing much more reliable. 

In~\cite{PR05}, Papadimitriou and Ratajczak presented a remarkable conjecture that every 3-connected planar graph admits a greedy drawing.
This conjecture gives a theoretical guarantee to the approach proposed by Rao et al.~\cite{RRPSS03}, and triggered intense research on
the class of graphs that admit greedy drawings (greedy-drawable graphs).
The conjecture of Papadimitriou and Ratajczak was first proved for triangulations by Dhandapani~\cite{D10}.
The conjecture itself was proved independently by Leighton and Moitra~\cite{LM10} and Angelini et al.~\cite{AFG10}.
Subsequently,  Da Lozzo et al.~\cite{DAF20} proved a stronger version of the conjecture, asserting the existence of a planar greedy drawing.
As greedy-drawability is a monotonic graph property (that is, the property is preserved under edge addition), the class of greedy-drawable trees has also gained much attention.
In~\cite{LM10}, Leighton and Moitra investigated a sufficient condition in which a binary tree does not admit a greedy drawing.
N\"{o}llenburg and Prutkin\cite{NP17} characterized greedy-drawable trees in terms of an inequality system containing a non-linear equation.
Investigating the feasibility of the inequality system, they derived an explicit description of the greedy-drawable trees with maximum degree $3$ and a linear-time recognition algorithm for greedy-drawable trees with maximum degree $\leq 4$.
In the case of maximum degree~$5$, however, testing the feasibility of the inequality system is sometimes difficult, and
neither an explicit description nor a recognition algorithm is known for greedy-drawable trees.
Indeed, the authors of the paper~\cite{NP17} presented a tree with maximum degree~$5$ (p.~574, Fig.~21 (d)) whose greedy-drawability was unclear.
Meanwhile, Kleinberg~\cite{K10} proved that every tree has a greedy drawing in the hyperbolic plane.
\\
\\
\textbf{Our contributions.}
In this paper, we present a complete combinatorial characterization of  (Euclidean) greedy-drawable trees.
First, we carefully analyze and clarify the work of N\"{o}llenburg and Prutkin~\cite{NP17}, and derive an explicit description of greedy-drawable trees with maximum degree $\leq 4$ (Proposition~\ref{prop:deg4}).
Next, we derive an explicit description of greedy-drawable trees with maximum degree $5$ (Theorem~\ref{thm:main}) by developing a new technique for determining the feasibility of the inequality system proposed in~\cite{NP17}.
As the maximum degree of  a greedy-drawable tree must be less than $6$~\cite{PR05}, our results represent a complete combinatorial characterization of greedy-drawable trees.
Furthermore, we study the greedy-drawability of  \emph{pseudo-trees}, i.e., graphs obtained by adding one edge to a tree.
We provide a complete characterization of greedy-drawable pseudo-trees (Theorem~\ref{thm:pseudo_trees}).
Using our characterizations of greedy-drawable trees and pseudo-trees, one can verify the greedy-drawability of a tree or pseudo-tree in linear time (Corollaries~\ref{cor:tree} and \ref{cor:pseudo_tree}).
\\
\\
\textbf{Related Work.}
In a greedy drawing of a graph, one can find a path between any two vertices by iteratively selecting a neighboring vertex closer to the destination.
In this sense, greedy drawings can be viewed as graph drawings in which one can easily find a path between any two vertices.
Such a property is desirable in many applications, and several other types of graph drawings have been proposed for helping path-finding tasks.
Alamdari et al.~\cite{ACLP13} introduced \emph{self-approaching drawings}, which satisfy a condition stronger than greedy drawings that
every two vertices $s$ and $t$ are joined by an $st$-path such that $d(b,c) < d(a,c)$ holds for any points (not necessarily vertices) $a,b,c$ along the path,
where $d$ is the Euclidean distance.
Self-approaching drawings have several advantages over greedy drawings.
For example, the \emph{stretch factor}, i.e., the maximum, over vertices $s$ and $t$, ratio between the length of a shortest $st$-path and the Euclidean distance $d(s,t)$, is always at most $5.3332$~\cite{IKL95} for self-approaching drawings, 
whereas the ratio can be arbitrarily large for greedy drawings. 
The authors of~\cite{ACLP13} characterized the trees that have self-approaching drawings.
In terms of our characterization of greedy-drawable trees, those trees correspond to subdivisions of type-$D_{k,l,n}$ trees with maximum degree $3$ (see Section~\ref{sec:open_angle}), which form a very special case of greedy-drawable trees.
\emph{Increasing-chord drawings}~\cite{ACLP13} are an even more restricted class of drawings, in which every two vertices $s$ and $t$ are connected by a path that is an $st$-self-approaching and $ts$-self-approaching path.
The stretch factor of an increasing-chord drawing is at most 2.094~\cite{R04}.
N\"{o}llenburg et al.~\cite{NPR16} proved that triangulations admit increasing-chord drawings in the Euclidean plane and 3-connected planar graphs admit such drawings in the hyperbolic plane.
Angelini et al.~\cite{ADKMRSW15} introduced \emph{monotone drawings}, in which there is always a path between two vertices that is monotone with respect to some direction.
The authors of~\cite{ADKMRSW15} proved that every biconnected graph and outer planar graph admits a monotone drawing.
They also introduced the notion of strongly monotone drawings, which assumes the existence of an $st$-path between every pair of vertices $s$ and $t$ that is monotone with respect to the direction $\vec{st}$.
Kindermann et al.~\cite{KSSW14} proved that every tree admits a strongly monotone drawing.
Felsner et al.~\cite{FIKKMS16} showed that every 3-connected planar graph and 2-tree admits a strongly monotone drawing. 
Dehkhori et al.~\cite{DFG15} introduced \emph{angle-monotone drawings} (this name is given in \cite{BBCKLV16}), which require a path between every pair of vertices that is $x$- and $y$-monotone after some rotation. 

\section{Preliminaries}
\label{sec:preliminaries}
Let $G$ be a graph.
A drawing $\Gamma$ of $G$ is called a \emph{straight-line drawing}
if every node is represented as a point in the plane and
every edge as the line segment between its endpoints.
We hereinafter assume that all drawings are \emph{plane straight-line drawings}, that is, straight-line drawings with no edge crossing.
The drawing $\Gamma$ is said to be \emph{greedy} if, for every pair of vertices $s,t(s \neq t)$, 
there exists a neighboring vertex $u$ of $s$ with $d(u, t) < d(s, t)$, where $d$ is a distance function.
That is, the drawing $\Gamma$ is greedy if, for every pair of vertices $s,t(s \neq t)$, there exists a path $v_0 (\coloneqq s),v_1,\dots,v_m(\coloneqq t)$ such that $d(v_i,t) > d(v_{i+1},t)$ for $i=1,\dots,m-1$.
Throughout this paper, we consider the case in which $d$ is the Euclidean distance.   
If $G$ has a greedy drawing, we say that $G$ is \emph{greedy-drawable}. 

Let $T$ be a tree.
For an edge $uv$ of $T$, we consider the subtree  $T^u_{uv}$ of  $T$ that contains $u$ obtained by deleting $uv$.
We let $\mathrm{axis}(uv)$ be the perpendicular bisector of $uv$, and consider the open half-plane $h^u_{uv}$  bounded by $\mathrm{axis}(uv)$ that contains $u$. 
Since the half-plane $h^u_{uv}$ corresponds to the set of points that are closer to $u$ than to $v$, the following proposition holds.
\begin{prop}(\cite[Lemma~3]{ABF12},\cite[Lemma~2.3]{NP17})\\
 A drawing $\Gamma$ of a tree $T$ is greedy if and only if every subtree $T^u_{uv}$ is contained in $h^u_{uv}$ in $\Gamma$.
\label{prop:halfspace}
\end{prop}
Based on this proposition, N\"{o}llenburg and Prutkin~\cite{NP17} developed the following strategy to determine the greedy-drawability of a tree $T$.
Let $r$ be a vertex of $T$, and $v_0, \dots, v_{d-1}$ be the neighbors of $r$.
We first decompose $T$ into $d$ subtrees  $T^{v_i}_{rv_i} +rv_i \ (i = 0, \dots, d-1)$.
Then, we regard each tree $T^{v_i}_{rv_i} +rv_i$ as  a rooted tree with root $r$, and denote it as $T_i = (V_i,E_i)$. 
We let $\mathrm{polygon}(T_i) \coloneqq \bigcap \{ h_{uw}^w \mid uw \in E_i, uw \neq rv_i, d_{T}(w,r) < d_T(u,r)\}$, where $d_T$ is the graph distance, and let $\widehat{T}_r$ be the star induced by $r$ and its neighbors.
Suppose we already have a greedy drawing of  each $T_i$ and want to construct a greedy drawing of $T$ by combining them.
Then, the task is to construct a drawing in which $\widehat{T}_r$ is drawn greedily and each $T_j$ ($j \neq i$) is contained in $\mathrm{polygon}(T_i)$.
If $\mathrm{polygon}(T_i)$ is unbounded with respect to the direction $\overrightarrow{v_ir}$, then this task can be simplified;
if this condition holds, one can transform the drawing of $T_i$ into a drawing in which $T^{v_i}_{rv_i}$ is drawn infinitesimally small, where $\mathrm{polygon}(T_i)$ is (arbitrarily close to) the open cone spanned by the two unbounded edges (see Figure~\ref{fig:polytope} for an intuition).
This observation leads to the notion of opening angles, which plays a fundamental role in classifying greedy-drawable trees.
\begin{defn}(\cite[Definitions 2.8, 2.10, and 2.12]{NP17}, slightly modified)\\
Let $\Gamma$ be a drawing of $T$.
\begin{itemize}
\item
If $\mathrm{polygon}(T_i)$ is unbounded, we say that $T_i$ is drawn with an \emph{open angle} in $\Gamma$.
Let $a_1b_1$ and $a_2b_2$ be the edges of $T$, where $d_T(a_i,r) < d_T(b_i,r)$, such that $\mathrm{axis}(a_1b_1)$ and $\mathrm{axis}(a_2b_2)$ are the supporting lines of the two unbounded edges of $\mathrm{polygon}(T_i)$.
We define $\angle{T_i} \coloneqq h_{a_1b_1}^{a_1} \cap h_{a_2b_2}^{a_2}$ and refer to it as the \emph{opening angle} of $T_i$ in $\Gamma$. 
We write $|\angle{T_i}| = \alpha$, where $\alpha$ is the angle between the two rays of $\angle{T_i}$.
We define $|\angle{T_i}|_* \coloneqq \sup  (|\angle{T_i}|)$, where the supremum is taken over all greedy drawings of $T_i$.
\item 
If $\mathrm{polygon}(T_i)$ is bounded, we say that $T_i$ is drawn with a \emph{closed angle} in $\Gamma$ and write $|\angle{T_i}| < 0$. 
If $|\angle{T_i}| < 0$ for every greedy drawing of $T_i$, we write $|\angle{T_i}|_* < 0$.
\end{itemize}
\label{def:open_angle}
\end{defn}
\begin{figure}[htb]
\begin{minipage}{0.5\linewidth}
\centering
\includegraphics[scale=0.15, bb = 0 0  846 948, clip]{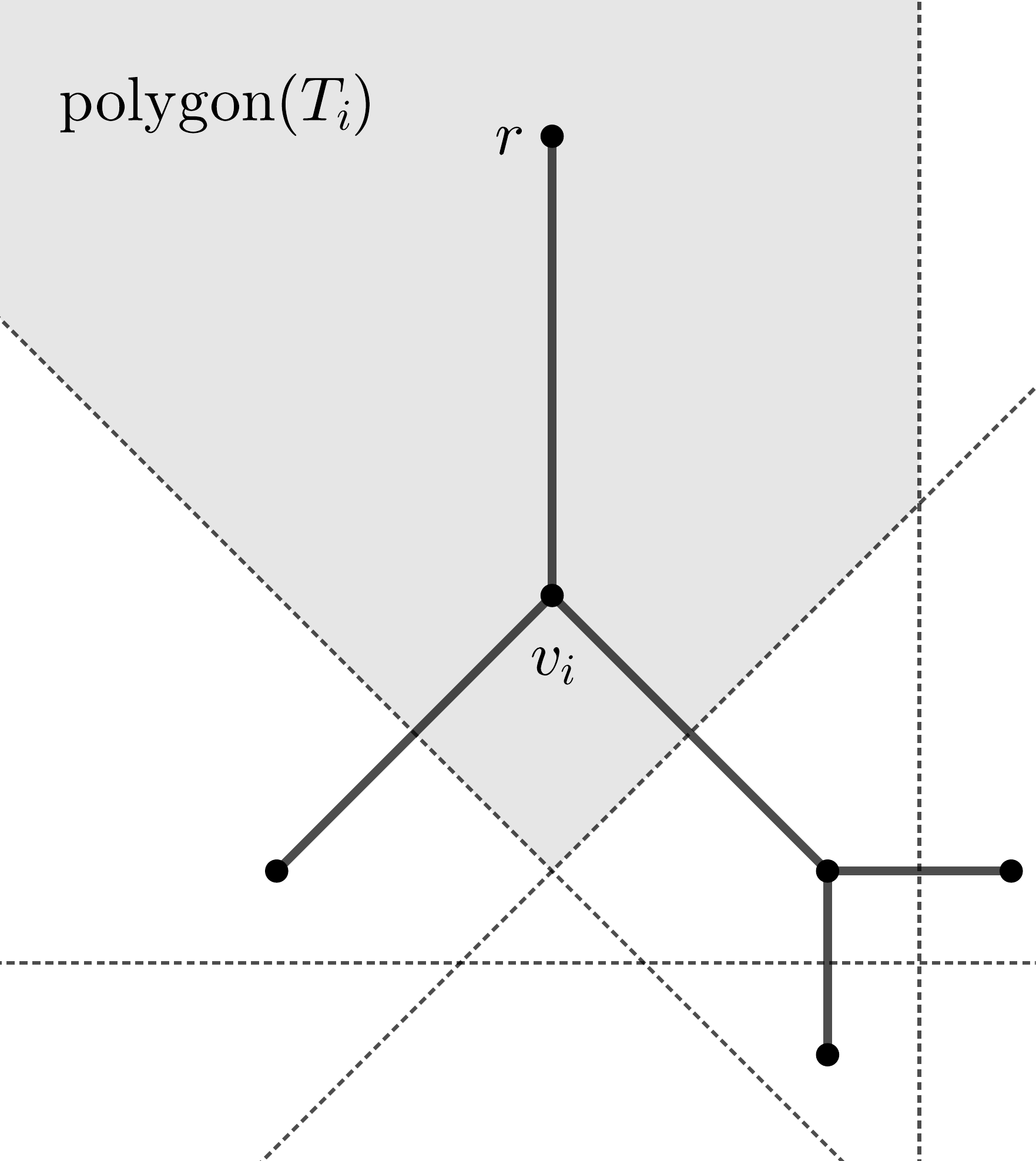} 
\end{minipage}
\begin{minipage}{0.5\linewidth}
\centering
\includegraphics[scale=0.15,bb = 0 0 846 948, clip]{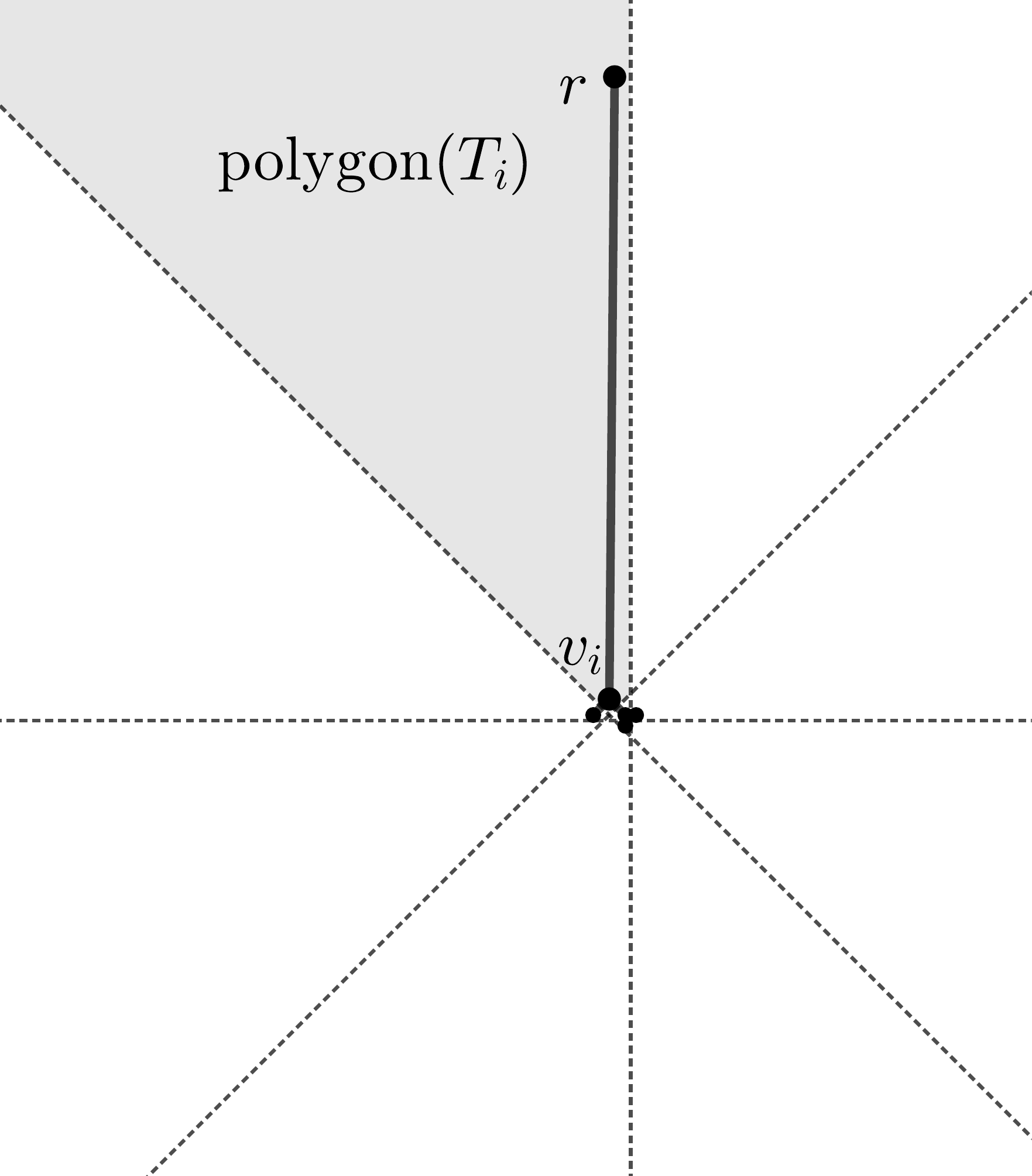} 
\end{minipage}
    \caption{Drawings of $T_i$ (left: original drawing, right: shrunk drawing)} 
\label{fig:polytope}
\end{figure}
Slightly perturbating a greedy drawing, one can always assume that  $|\angle{T_i}| > 0$ or $|\angle{T_i}| < 0$.
In what follows, we do not pay so much attention to the case  $|\angle{T_i}| = 0$.
N\"{o}llenburg and Prutkin~\cite{NP17} proved that the drawing of $T^{v_i}_{rv_i}$ can always be shrunk while preserving greediness of $T$ if $|\angle{T_i}| > 0$. 
\begin{lem}(shrinking lemma~\cite[Lemma~2.18]{NP17})\\
Let $T=(V,E)$ be a tree and $T' = T^v_{rv} +rv$, $rv \in E$, be a subtree of $T$.
If $T$ has a greedy drawing $\Gamma$ such that $|\angle{T'}| > 0$, then there is a greedy drawing $\Gamma'$  such that $T_{rv}^v$ is drawn infinitesimally small, and the drawing of $T_{rv}^r$ and the angle $|\angle{T'}|$ are the same as those in $\Gamma$
(and $\angle{T'}$ contains the original angle in $\Gamma$).
\label{lem:shrinking}
\end{lem}
\begin{remark}
For some technical reasons, the original definition of an opening angle is slightly more complicated.
Indeed, Lemma~2.14 in \cite{NP17}, which claims that the apex of the opening angle $\angle{T_i}$ of a subtree $T_i$ is always contained in the opening angle $\angle{T_j}$ of another subtree $T_j$ in every greedy drawing of $T$, does not always hold under the present definition. 
This fact is important because the shrunk drawing in Lemma~\ref{lem:shrinking} is constructed in \cite{NP17} by placing an infinitesimally small drawing of $T_i$ on the apex of $\angle{T_i}$.
However, once we accept Lemma~\ref{lem:shrinking}, and consider such a drawing, the original and present definition coincide (with arbitrary precision).
Throughout the remainder of this paper, we always consider such drawings, making the two definitions functionally equivalent.
\end{remark}
Suppose that  each $T_i$ can be drawn with an open angle, and let $\varphi_i \coloneqq |\angle{T_i}|_*$. 
Based on Lemma~\ref{lem:shrinking}, N\"{o}llenburg and Prutkin~\cite[Theorem~4.4]{NP17}  proved that $T$ has a greedy drawing if and only if there is a permutation $\tau$ on $\{ 0,\dots, d-1 \}$
such  that the following inequality system has a solution.
For $i=0,\dots,d-1$,
\begin{align}
\begin{split}
&0 < \alpha_i, \beta_i, \gamma_i < 180, \\
&\alpha_i + \beta_i + \gamma_i = 180, \ \alpha_0 + \dots + \alpha_{d-1} = 360, \\
&\sin (\beta_0)\cdot \dots \cdot \sin (\beta_{d-1}) = \sin (\gamma_0)\cdot \dots \cdot \sin (\gamma_{d-1}), \\
&\beta_i < \alpha_i, \ \gamma_i < \alpha_i, \\
&\beta_i + \gamma_{i+1} < \varphi_{\tau(i)} \ (i \bmod d).
\end{split}
\label{ineq_main}
\end{align}
The conditions in the first three lines describe the possible angles $\alpha_i,\beta_i,\gamma_i$ of the wheel graph $W_d$, which appear in the work of Di Battista and Vismara~\cite{DV96}.
Together with the conditions in the fourth line, they describe all possible angles in greedy drawings of  $\widehat{T_r}$.
Finally, the combined system describes  all possible angles of greedy drawings of $T$ in which the subtrees $T_0,\dots,T_{d-1}$ are drawn infinitesimally small with the order $T_{\tau (0)},\dots, T_{\tau (d-1)}$ around $r$.
See Figure~\ref{fig:gluing} for an intuition.
We will see that, for any greedy-drawable tree $T$, a vertex $r$ can be chosen so that $|\angle{T_0}|_*>0,\dots,|\angle{T_{d-1}}|_*>0$ (Lemma~\ref{lem:all_open_angles}).
Thus, the classification of greedy-drawable trees can be performed by the following two steps~\cite{NP17}:
\begin{enumerate}
\item Determine all rooted trees that can be drawn with open angles, and compute the supremum of opening angles of each tree. This step is discussed in Section~\ref{sec:open_angle}.
\item Determine all possible combinations of rooted trees that can be glued at a single vertex by considering feasibility of the system~(\ref{ineq_main}).  This step is discussed in Sections~\ref{sec:deg4} and~\ref{sec:deg5}.
\end{enumerate}

\begin{figure}[h]
\centering
\includegraphics[scale=0.3, bb =0 0 369 329, clip]{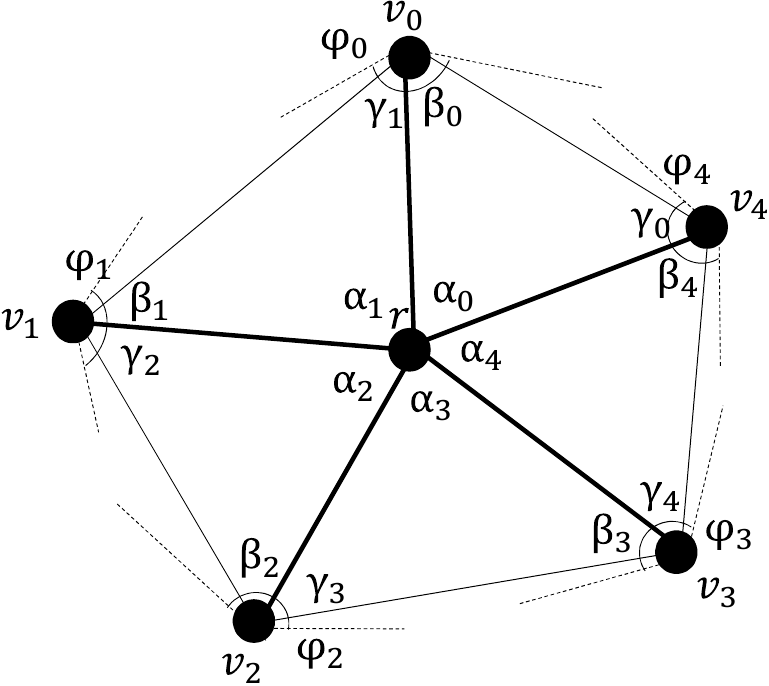} 
  \caption{Construction of a greedy drawing of $T$ by combining drawings of subtrees $T_0,\dots,T_{d-1}$ ($d=5$)}      
\label{fig:gluing}
\end{figure}

\section{Rooted trees that can be drawn with open angles}
\label{sec:open_angle}
N\"{o}llenburg and Prutkin\cite{NP17} introduced a procedure for computing the supremum of opening angles of a rooted tree.
By analyzing their results, we present an explicit description of rooted trees that can be drawn with an open angle.

To describe the result, we introduce some definitions.
Let $P$ be a path.
A tree $T$ is called a \emph{degree-$k$ caterpillar} associated with $P$ if $T$ is constructed by attaching at most $k-1$ leaves to each terminal vertex of $P$, and at most $k-2$ leaves  to each internal vertex of $P$.
We call the number of degree-$k$ vertices the \emph{weight} of $T$.
A leaf of $T$ that is adjacent to a terminal vertex of $P$ is called an \emph{end leaf} of $T$.

If a tree $T$ can be drawn with an opening angle $\varphi -\epsilon$ for any $\epsilon > 0$, but not $\varphi$, we indicate this information by writing $|\angle{T}|_* = \varphi^-$.
As subdividing a tree does not affect the supremum of opening angles, we describe the result under the assumption that a tree does not contain any degree-$2$ vertices.

Here, we review  results in \cite{NP17}.
Let $T$ be a rooted tree with degree-$1$ root.
If all vertices of $T$ are within distance $2$ from the root, i.e., if $T$ is a star, it is easy to determine $|\angle{T}|_*$.
We have $|\angle{T}|_* =180^\circ, (120^\circ)^-, (60^\circ)^-$ if  $T$ is either a single edge, a triple, or a quadruple respectively. 
If $T$ contains a vertex of degree greater than $5$, $T$ cannot be drawn with an open angle.
If there is a vertex of distance greater than $2$, one can compute $|\angle{T}|_*$  by applying the following lemma recursively.
We remark that if $|\angle{T}|_* \neq 180^\circ$, then it means the tree $T$ contains a vertex of degree~$\geq 3$, and thus $|\angle{T}|_* \leq 120^\circ$.
\begin{lem}(\cite[Lemmas~3.1--3.6]{NP17})
Let $T_i$, for $i=1,2,3$, be a rooted tree with a degree-$1$ common root $r'$ that contains no degree-$2$ vertices.
Suppose that each $T_i$ can be drawn with an open angle, and the supremum $\varphi_i \coloneqq |\angle{T_i}|_*$ of opening angles is not equal to $180^\circ$.
\begin{itemize}
\item[(I)] 
Let $T \coloneqq T_1 + ar' + rr'$ be the rooted tree with root $r$, as depicted in Figure~\ref{fig:case2_graph}. If $90^\circ < \varphi_1 \leq 120^\circ$, we have $|\angle{T}|_*  = (45^\circ+\frac{\varphi_1}{2})^-$.
If $\varphi_1 \leq 90^\circ$, we have $|\angle{T}|_*  = \varphi_1^-$ (see Figures~\ref{fig:case2_1} and \ref{fig:case2_2} for an intuition).
\item[(II)] 
Let $T \coloneqq T_1 + ar' + br'+ rr'$ be the rooted tree with root $r$, as depicted in Figure~\ref{fig:case3_graph}. If $0^\circ < \varphi_1 \leq 120^\circ$, then we have $|\angle{T}|_*  = (\frac{\varphi_1}{2})^-$ (see Figure~\ref{fig:case3} for an intuition).
\item[(III)] 
Let $T \coloneqq T_1 + T_2 + rr'$ be the rooted tree with root $r$, as depicted in Figure~\ref{fig:case4_graph}.
If $90^\circ < \varphi_1,\varphi_2 \leq 120^\circ$, we have $|\angle{T}|_*  = (\varphi_1+\varphi_2 -180^\circ)^-$. 
If $\varphi_1 \leq 90^\circ$ or $\varphi_2 \leq 90^\circ$, then $T$ cannot be drawn with an open angle
(see Figure~\ref{fig:case4} for an intuition).
\item[(IV)] 
Let $T \coloneqq T_1 + T_2 + ar'+rr'$ be the rooted tree with root $r$, as depicted in Figure~\ref{fig:case5_graph}.
If $90^\circ < \varphi_2 \leq \varphi_1 \leq 120^\circ$, we have $|\angle{T}|_*  = (\frac{3}{4}\varphi_1+\frac{1}{2}\varphi_2 -112.5^\circ)^-$. 
If $\varphi_1 \leq 90^\circ$ or $\varphi_2 \leq 90^\circ$, then $T$ cannot be drawn with an open angle
(see Figure~\ref{fig:case5} for an intuition).
\item[(V)] 
Let $T \coloneqq T_1 + T_2 + T_3+rr'$ be the rooted tree with root $r$, as depicted in Figure~\ref{fig:case6_graph}.
Then, $T$ cannot be drawn with an open angle. 
\end{itemize}
\label{lem:combination}
\end{lem}

\begin{figure}[h]
\begin{minipage}{0.33\linewidth}
\centering
\includegraphics[scale=0.2, bb = 0 0 343 208,clip]{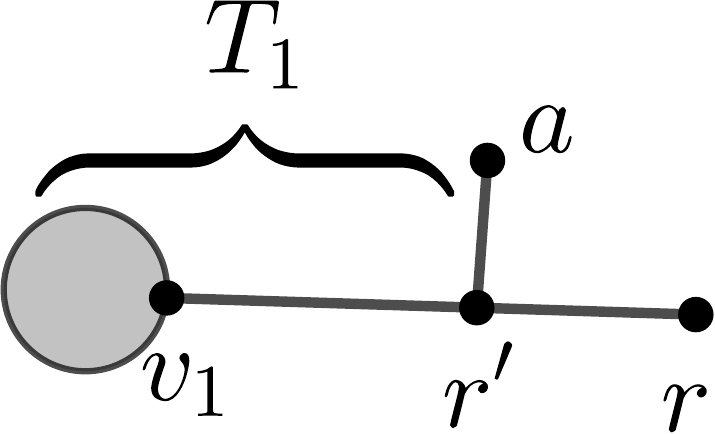}
\subcaption{Case~(I)}
\label{fig:case2_graph}
\end{minipage}
\begin{minipage}{0.33\linewidth}
\centering
\includegraphics[scale=0.2, bb = 0 0 355 281,clip]{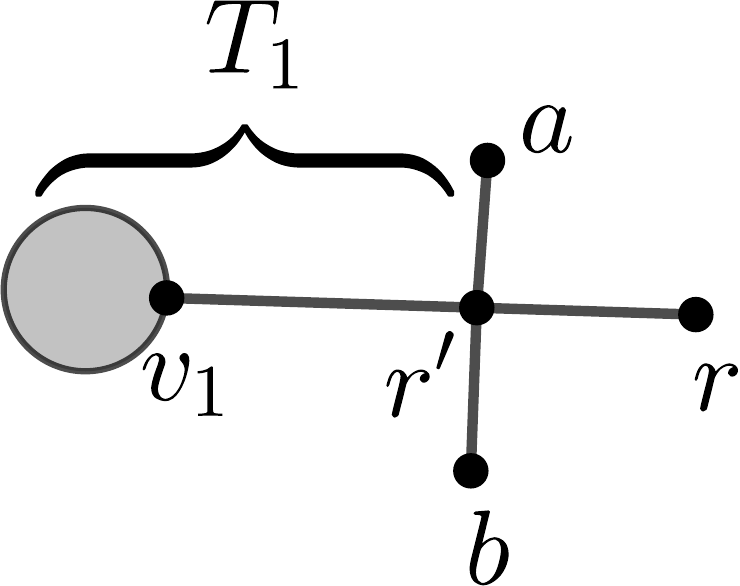}
\subcaption{Case~(II)}
\label{fig:case3_graph}
\end{minipage}
\begin{minipage}{0.33\linewidth}
\centering
\includegraphics[scale=0.2, bb =  0 0 387 459,clip]{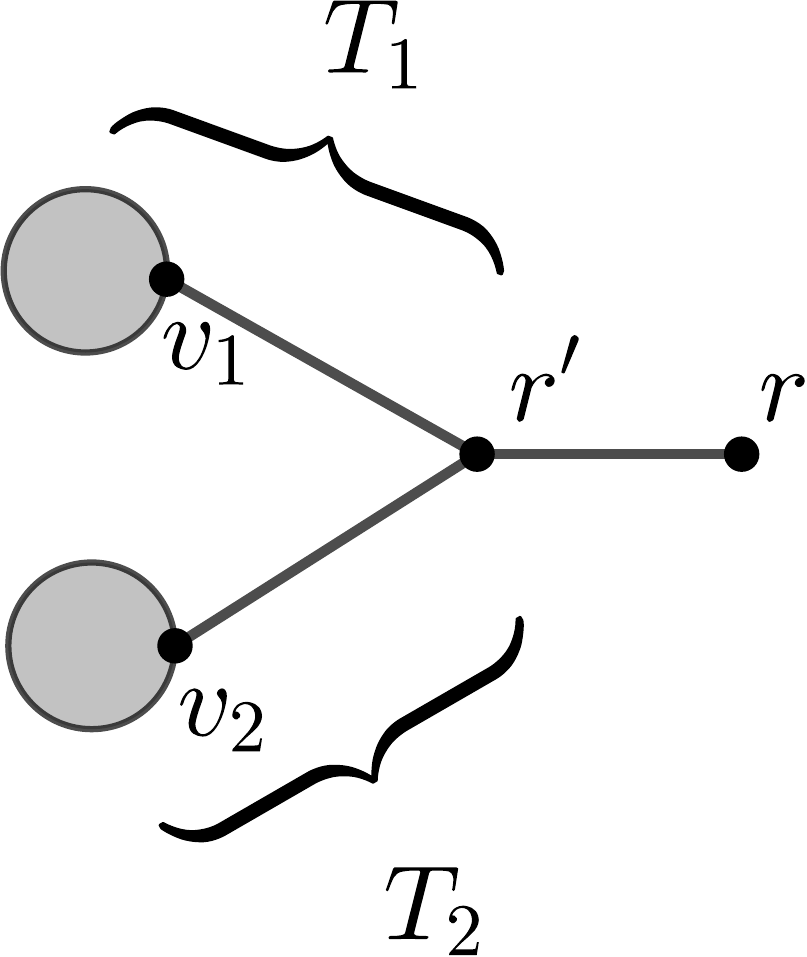}
\subcaption{Case~(III)}
\label{fig:case4_graph}
\end{minipage}
\begin{minipage}{0.33\linewidth}
\centering
\includegraphics[scale=0.2, bb =  0 0 389 459,clip]{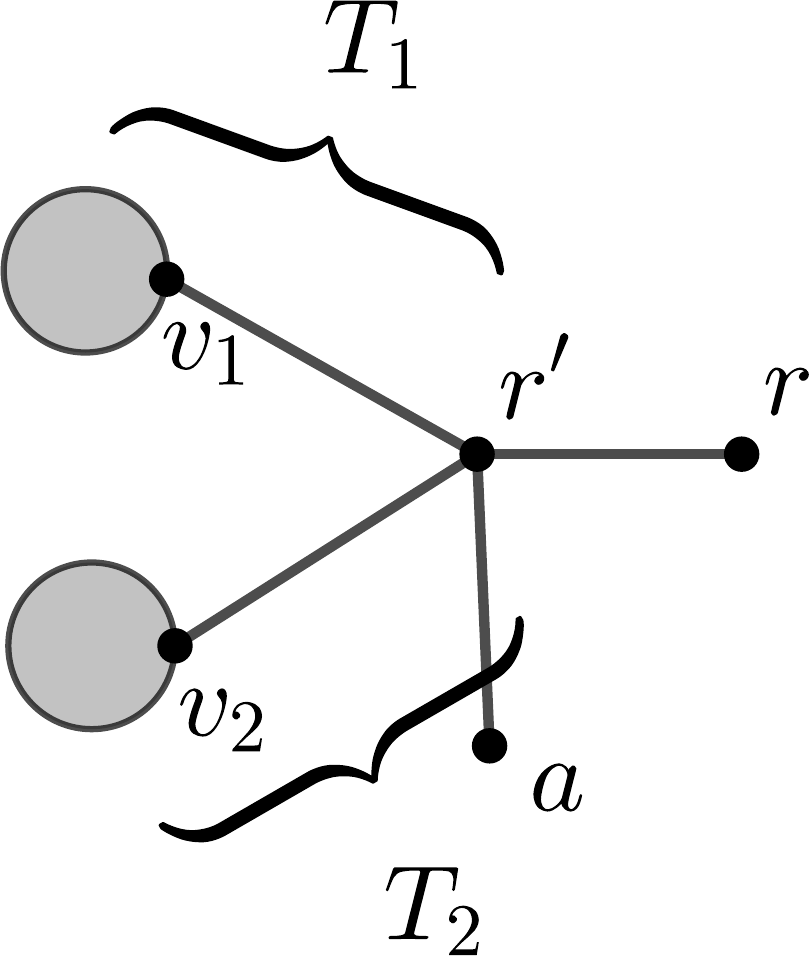}
\subcaption{Case~(IV)}
\label{fig:case5_graph}
\end{minipage}
\begin{minipage}{0.33\linewidth}
\centering
\includegraphics[scale=0.2, bb = 0 0 442 394,clip]{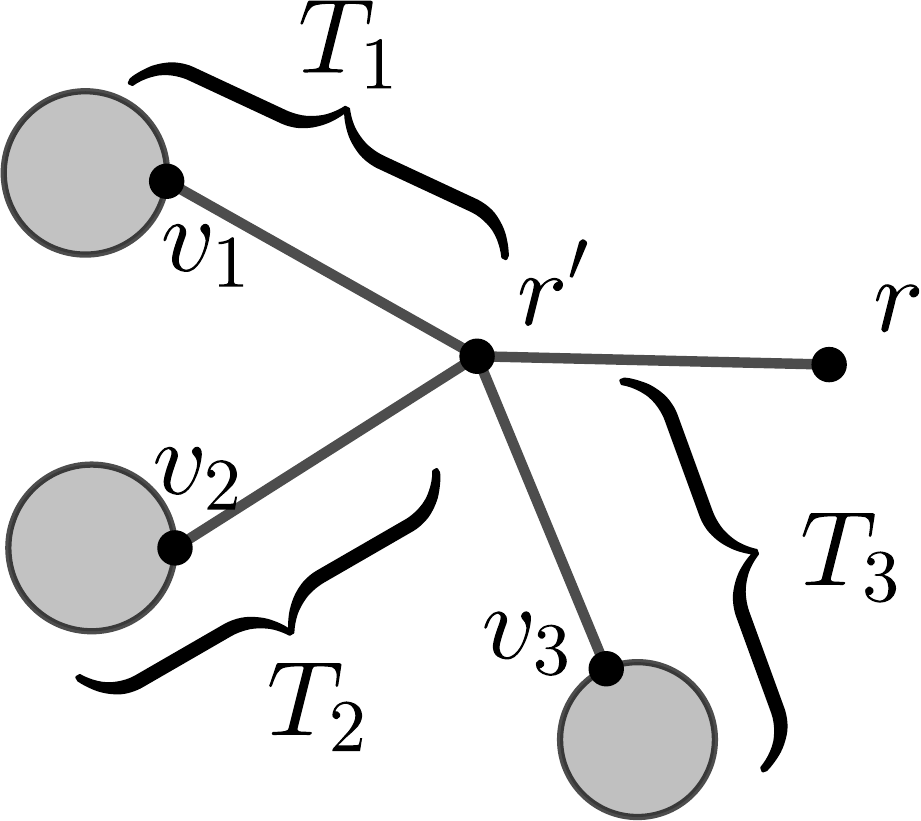}
\subcaption{Case~(V)}
\label{fig:case6_graph}
\end{minipage}
\caption{Rooted trees in Lemma~\ref{lem:combination}}
\end{figure}

\begin{figure}[h]
\begin{minipage}{0.33\linewidth}
\centering
\includegraphics[scale=0.15, bb =   0 0 846 948,clip]{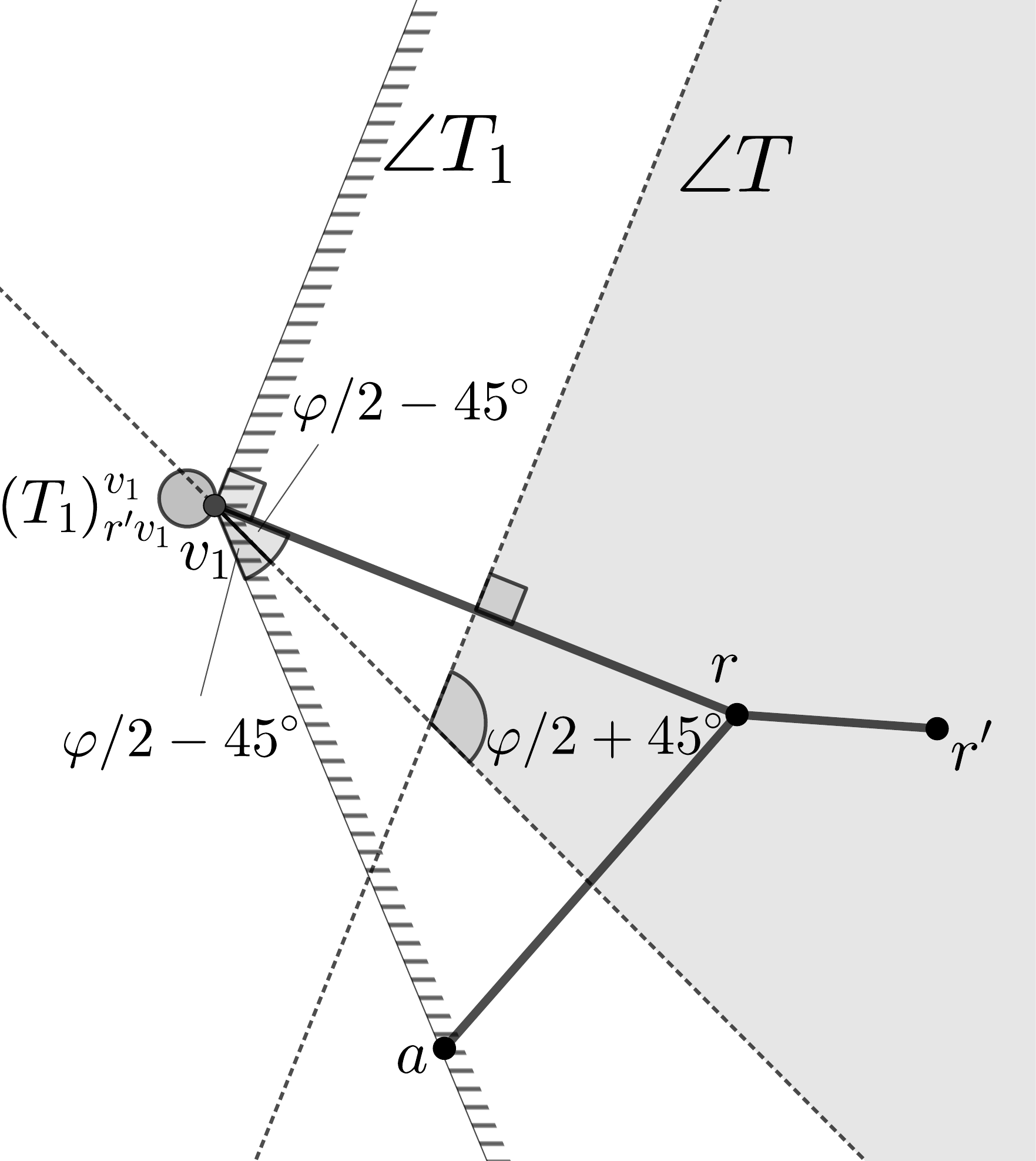}
\subcaption{Case~(I) ($\varphi_1 \geq 90^\circ$)}
\label{fig:case2_1}
\end{minipage}
\begin{minipage}{0.33\linewidth}
\centering
\includegraphics[scale=0.15, bb =  0 0 846 949,clip]{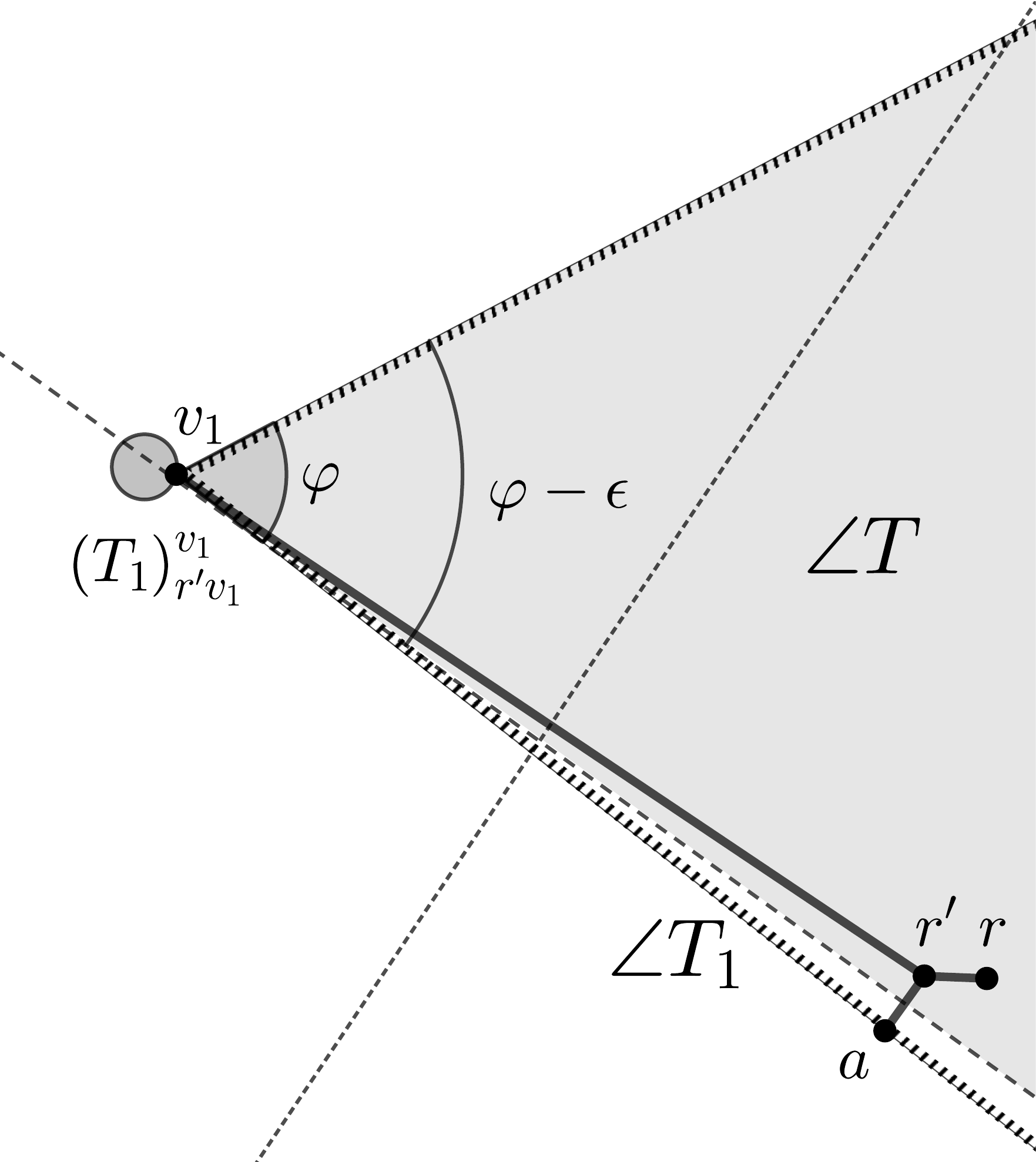}
\subcaption{Case~(I) ($\varphi_1 < 90^\circ$)}
\label{fig:case2_2}
\end{minipage}
\begin{minipage}{0.33\linewidth}
\centering
\includegraphics[scale=0.15, bb =  0 0 784 936,clip]{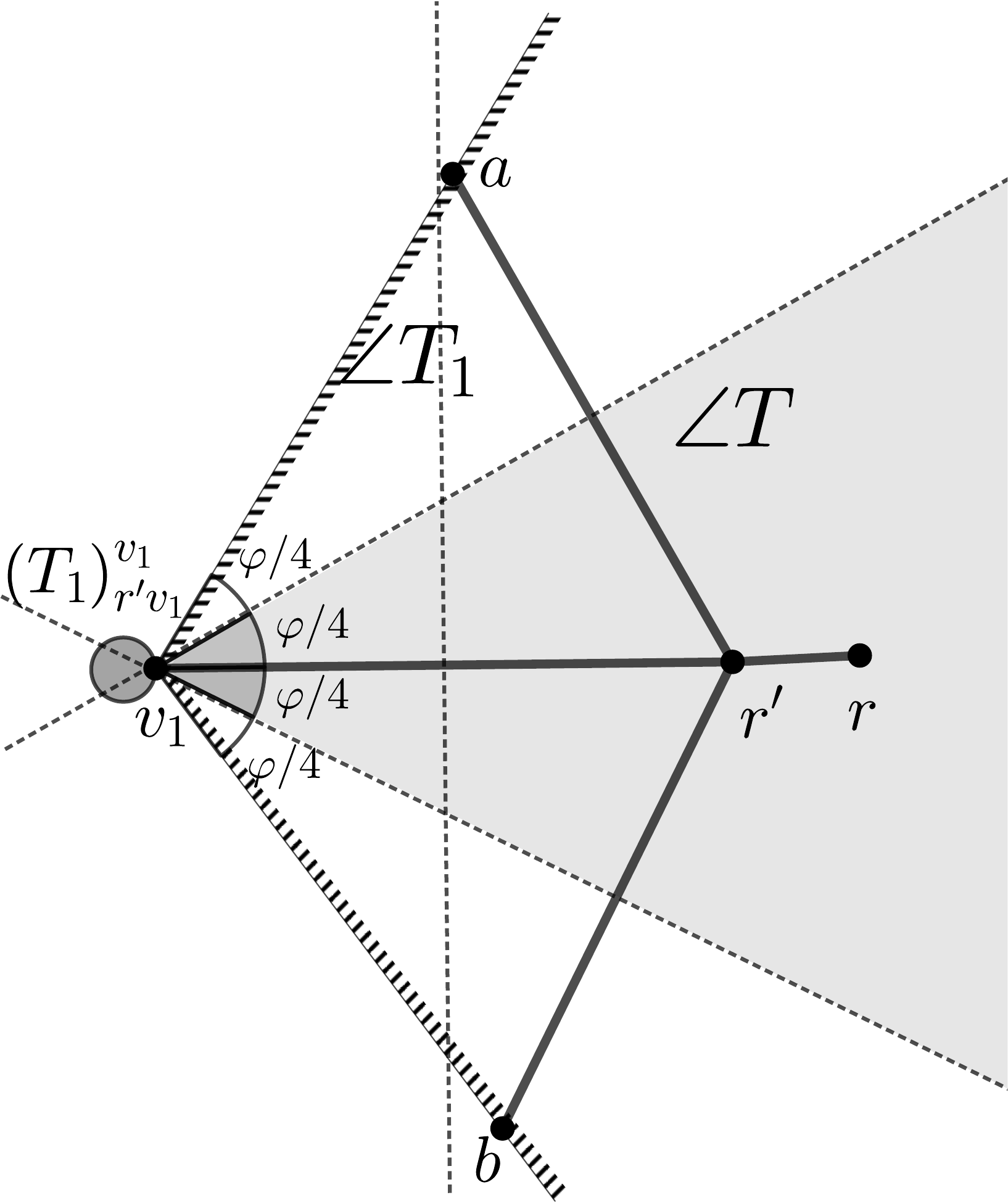}
\subcaption{Case~(II)}
\label{fig:case3}
\end{minipage}
\begin{minipage}{0.33\linewidth}
\centering
\includegraphics[scale=0.15, bb =  0 0 815 945,clip]{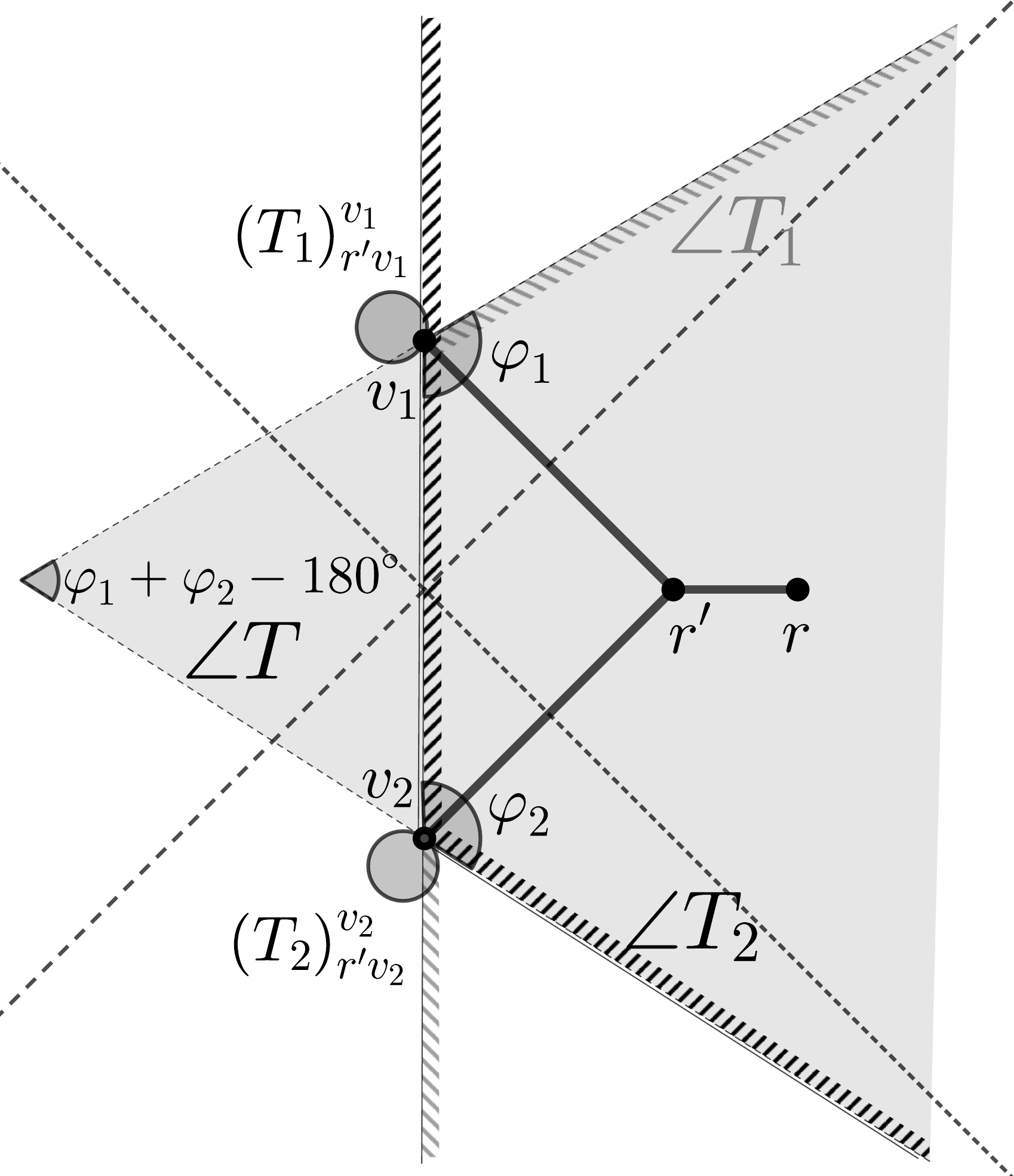}
\subcaption{Case~(III)}
\label{fig:case4}
\end{minipage}
\begin{minipage}{0.33\linewidth}
\centering
\includegraphics[scale=0.15, bb = 0 0 852 928,clip]{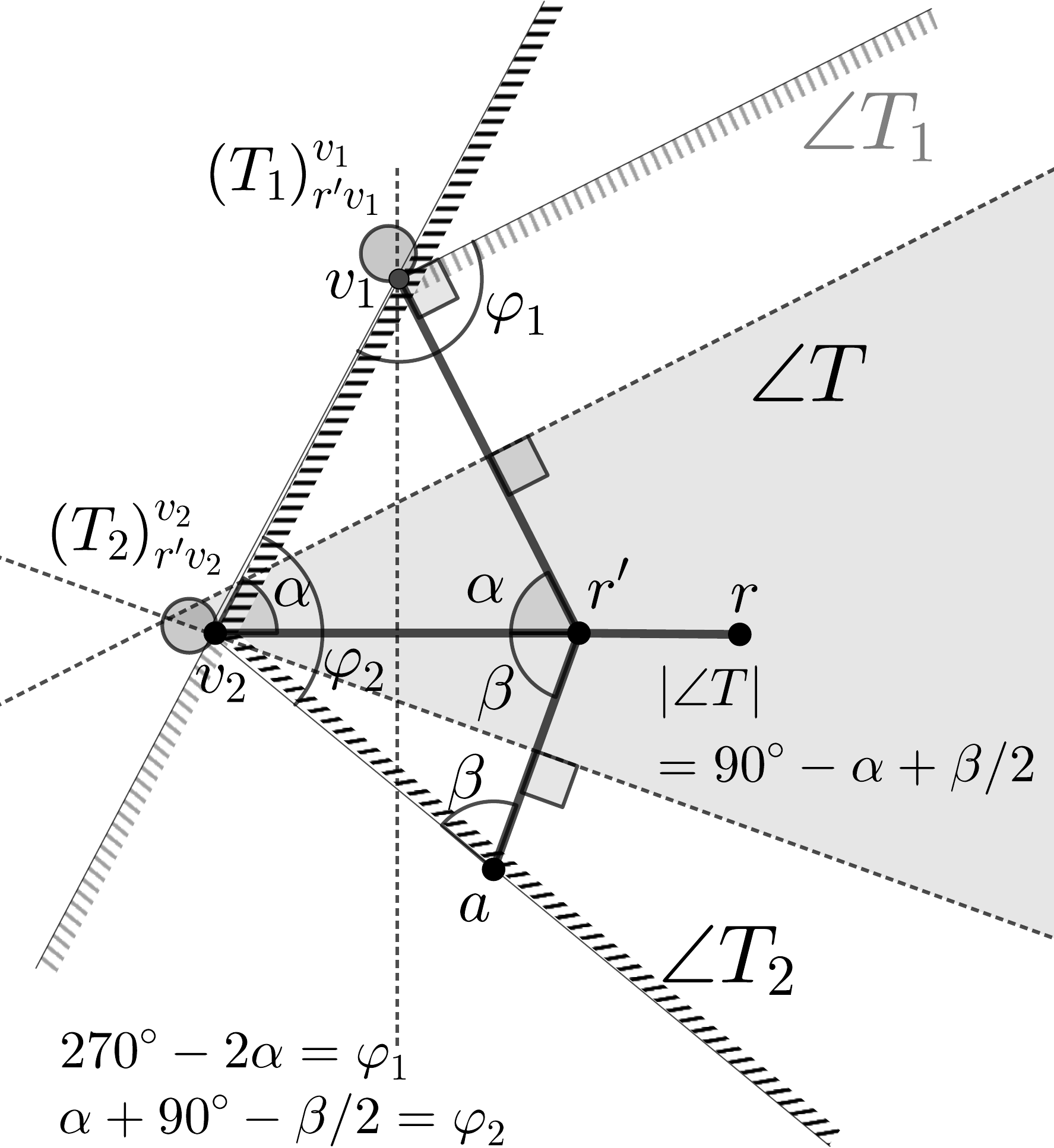}
\subcaption{ Case~(IV)}
\label{fig:case5}
\end{minipage}
\caption{Optimal constructions for Cases~(II)-(IV).  Each subtree $T^{v_i}_{r'v_i}$ is drawn infinitesimally small. One can construct proper greedy drawings by perturbating these drawings. For more details, see~\cite{NP17}.}
\end{figure}

For more details, see~\cite{NP17}. Using this lemma, we can prove the following proposition.

\begin{prop}\label{prop:open_angle}
Let $T$ be a rooted tree with degree-$1$ root $r$ that contains no degree-$2$ vertices.
Then, $T$ has a greedy drawing with an open angle if and only if $T$ satisfies one of the following conditions.

\begin{enumerate}
	\item[(A)] $T$ is a single edge.  In this case, we have $|\angle{T}|_*=180^\circ$ and call $T$ a type-$A$ tree. 
	 
	\item[(B)] $T$ is a degree-$3$ caterpillar of weight $n \geq 1$ and $r$ is an end leaf of $T$.  In this case, we have  $|\angle{T}|_* = (90^\circ + 60^\circ \times \frac{1}{2^n})^-$ and call $T$ a type-$B_n$ tree. See Figure~\ref{typeB}.
	
	\item[(C)] $T$ is a degree-$4$ caterpillar of weight $n \geq 1$ and
                     $r$ is an end leaf of $T$.  
                     Let $v$ be the degree-$4$ vertex farthest from $r$ (with respect to the graph distance),
                     and $k$ be the weight of the degree-$3$ path formed by the degree-$3$ vertices that are farther from $r$ than $v$.
                     In this case, we have  $|\angle{T}|_* =  (120^\circ )^- \times \frac{1}{2^n}$ if $k=0$ and  $|\angle{T}|_* =  (90^\circ + 60^\circ \times \frac{1}{2^k})^- \times \frac{1}{2^n}$ otherwise. We call $T$ a type-$C_{k,n}$ tree. See Figure \ref{typeC}.

	\item[(D)]  $T$ is obtained by attaching a degree-$4$ caterpillar of weight $n \geq 0$ and two degree-$3$ caterpillars of weight $k,l \geq 1$  ($l \geq k$) to a single vertex $v$, called the \emph{joint vertex}.
                     $r$ is an end leaf of the attached degree-$4$ caterpillar that is farthest from $v$  (with respect to the graph distance). 
                      In this case, we have $|\angle{T}|_* =  (60^\circ \times \frac{1}{2^k} +  60^\circ \times \frac{1}{2^l})^- \times \frac{1}{2^n}$ and call $T$ a type-$D_{k,l,n}$ tree.  See Figure \ref{typeD}.
	
	\item[(E)] $T$ is obtained by attaching a degree-$4$ caterpillar of weight $n \geq 0$ and two degree-$3$ caterpillars of weight $k,l \geq 1$ ($l \geq k$) and a single edge to a single vertex $v$, called  the \emph{joint vertex}.
                     $r$ is an end leaf of the attached degree-$4$ caterpillar that is farthest from $v$ (with respect to the graph distance). 
                    In this case, we have $|\angle{T}|_* =  (45^\circ \times \frac{1}{2^k} +  30^\circ \times \frac{1}{2^l})^- \times \frac{1}{2^n}$ and call $T$ a type-$E_{k,l,n}$ tree.  See Figure \ref{typeE}.

\end{enumerate}

\end{prop}
\begin{figure}[htb]
\begin{minipage}[t]{0.45\linewidth}
\centering
\includegraphics[scale=0.25, bb =0 0 432 84, clip]{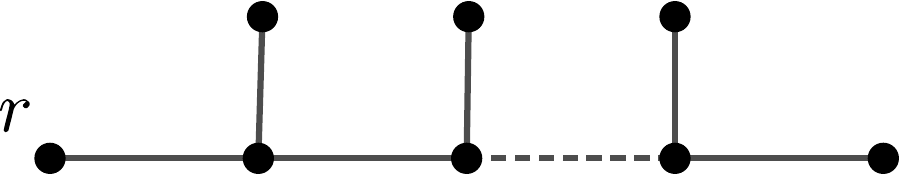} 
  \caption{Type-$B_n$ tree}      
\label{typeB}
\end{minipage}
\begin{minipage}[t]{0.45\linewidth}
\centering
\includegraphics[scale=0.25,bb = 0 0 432 116, clip]{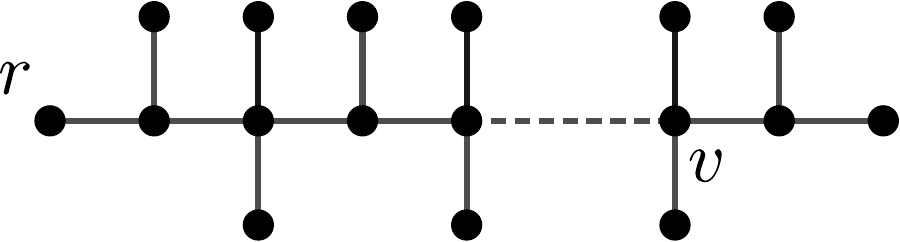} 
    \caption{Type-$C_{k,n}$ tree} 
\label{typeC}
\end{minipage}
\\
\\
\\
\begin{minipage}[t]{0.45\linewidth}
\centering
\includegraphics[scale=0.25, bb =  0 0 382 316, clip]{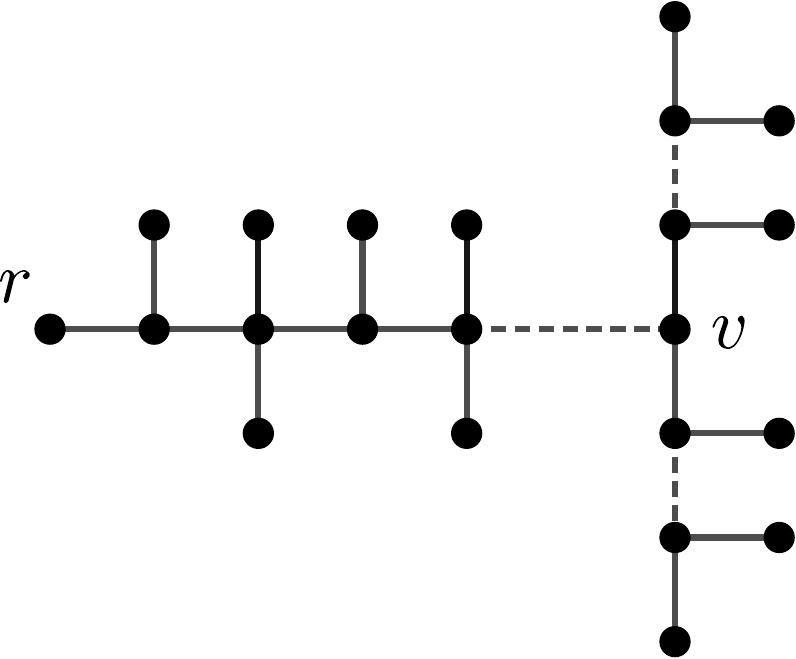} 
    \caption{Type-$D_{k,l,n}$ tree} 
\label{typeD}
\end{minipage}
\begin{minipage}[t]{0.45\linewidth}
\centering
\includegraphics[scale=0.25, bb = 0 0 382 316, clip]{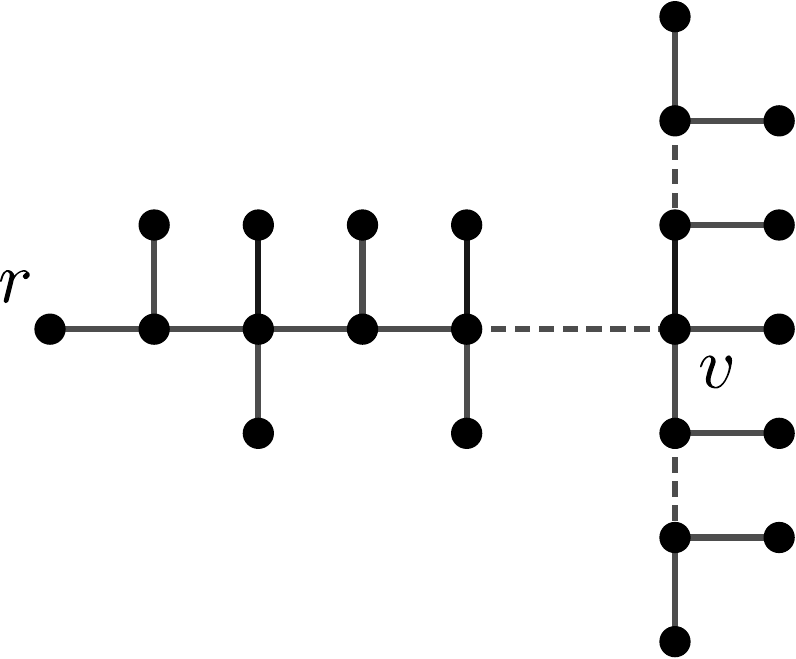} 
  \caption{Type-$E_{k,l,n}$ tree}      
\label{typeE}
\end{minipage}
\end{figure}
\begin{proof}
Let $r'$ be the neighboring vertex of $r$, and $T'$ the tree obtained from $T$ by removing $r$ and the leaves adjacent to $r'$.
Let $t$ be the (graph) distance between $r$ and a farthest vertex of $T$ from $r$.
We prove the following two claims separately.
\begin{itemize}
\item
Each tree $T$ in Cases~(A)--(E) has a greedy drawing with an open angle, and the supremum of opening angles is as described in the proposition.
\item
If $T$ has a greedy drawing with an open angle, then $T$ must be one of the trees that appear in Cases~(A)--(E).
\end{itemize}
We  first prove the first claim by induction on $t$.
Suppose $T$ is a tree that appears in one of Cases~(A)--(E).
If $t=2$, then $T$ must be a type-$A$, $B_1$, or $C_{0,1}$ tree, and the claim is clearly true.
Now we consider the case $t > 2$ by considering each type separately.

\begin{description}[itemsep=0cm,leftmargin=0cm,topsep=1ex]
\item[Case 1: $T$ is the type-$B_n$ tree.]
In this case, $T' $ is the type-$B_{n-1}$ tree.
By induction hypothesis, $T'$ has a greedy drawing with an open angle, and $|\angle{T'}|_*  = (90^\circ + 60^\circ \times \frac{1}{2^{n-1}})^-$.
Therefore, $T$ has a greedy drawing with an open angle, and we have $|\angle{T}|_* = (90^\circ + 60^\circ \times \frac{1}{2^{n}})^-$ by Lemma~\ref{lem:combination}~(I).

\item[Case 2:  $T$ is a type-$C_{k,n}$ tree.]
Let us first assume that the degree of $r'$ is $3$. Then,  the tree $T'$ is a type-$C_{k,n}$ tree.
By induction hypothesis,  we have $|\angle{T'}|_* =  (120^\circ )^- \times \frac{1}{2^n}$ if $k=0$ and  $|\angle{T'}|_* =  (90^\circ + 60^\circ \times \frac{1}{2^k})^- \times \frac{1}{2^n}$ otherwise.
By Lemma~\ref{lem:combination}~(I), $T$ has a greedy drawing with an open angle, and $|\angle{T}|_*=|\angle{T'}|_*$, which implies that $|\angle{T}|_*$ is as described in the proposition. 
We next assume that the degree of $r'$ is $4$.
Then, if $n \geq 2$, $T'$ is a type-$C_{k,n-1}$ tree, and we have $|\angle{T'}|_* =  (120^\circ )^- \times \frac{1}{2^{n-1}}$ if $k=0$, and  $|\angle{T'}|_* =  (90^\circ + 60^\circ \times \frac{1}{2^k})^- \times \frac{1}{2^{n-1}}$ otherwise, by induction hypothesis.
If $n=1$, then $T'$ is a type-$B_k$ tree, and we have $|\angle{T'}|_* =  (90^\circ + 60^\circ \times \frac{1}{2^k})^-$ by induction hypothesis.
By Lemma~\ref{lem:combination}~(II), $T$ has a greedy drawing with an open angle in every case,  and $|\angle{T}|_*=\frac{1}{2}|\angle{T'}|$, which implies 
$|\angle{T}|_* =  (120^\circ )^- \times \frac{1}{2^{n}}$ if $k=0$ and  $|\angle{T}|_* =  (90^\circ + 60^\circ \times \frac{1}{2^k})^- \times \frac{1}{2^{n}}$ otherwise.
Since the degree of $r'$ must be either $3$ or $4$, the claim is true in Case~2.
\item[Case 3: $T$ is a type-$D_{k,l,n}$ tree.]
We first assume that $r'$ is the joint vertex (and thus $n=0$). Then, $T$ is the tree in Case~(III) where $T_1$ and $T_2$ are the type-$B_k$ and type-$B_l$ trees respectively.
By induction hypothesis, we have  $|\angle{T_1}|_* =  (90^\circ + 60^\circ \times \frac{1}{2^{k}})^-$ and  $|\angle{T_2}|_* = (90^\circ + 60^\circ \times \frac{1}{2^{l}})^-$.
Thus, we have $|\angle{T}|_* =  (90^\circ + 60^\circ \times \frac{1}{2^{k}})^- +  (90^\circ + 60^\circ \times \frac{1}{2^{l}})^- -180^\circ = (60^\circ \times \frac{1}{2^k} +  60^\circ \times \frac{1}{2^l})^-$ by Lemma~\ref{lem:combination}~(III).
Next, we consider the case that $r'$ is not the joint vertex.
Then, $T'$ is either a type-$D_{k,l,n}$ tree (if the degree of $r'$ is $3$) or a  type-$D_{k,l,n-1}$ tree (if the degree of $r'$ is $4$).
By induction hypothesis and Lemma~\ref{lem:combination}~(I) and (II), in either case, $T$ has a greedy drawing with an open angle, and we have $|\angle{T}|_* = (60^\circ \times \frac{1}{2^k} +  60^\circ \times \frac{1}{2^l})^- \times \frac{1}{2^n}$.
Since the degree of $r'$ must be either $3$ or $4$, the claim is true in Case~3.
\item[Case 4: $T$ is a type-$E_{k,l,n}$ tree.]
If $r'$ is the joint vertex (and thus $n=0$), then $T$ is the tree in Case~(IV) where $T_1$ and $T_2$ are type-$B_k$ and type-$B_l$ trees respectively.
Thus, we have $|\angle{T}|_* = \frac{3}{4} \times (90^\circ+60^\circ \times \frac{1}{2^k}) + \frac{1}{2} \times (90^\circ+60^\circ \times \frac{1}{2^l}) = 45^\circ \times \frac{1}{2^k} + 30^\circ \times  \frac{1}{2^l}$ by induction hypothesis and Lemma~\ref{lem:combination}~(IV).
Next, we consider the case that $r'$ is not the joint vertex.
Then, $T'$ is either a type-$E_{k,l,n}$ tree (if the degree of $r'$ is $3$) or a type-$E_{k,l,n-1}$ tree (if the degree of $r'$ is $4$).
By induction hypothesis and Lemma~\ref{lem:combination}~(I) and (II), in either case,  $T$ has a greedy drawing with an open angle, and we have $|\angle{T}|_* =  (45^\circ \times \frac{1}{2^k} +  30^\circ \times \frac{1}{2^l})^- \times \frac{1}{2^n}$.
Since the degree of $r'$ must be either $3$ or $4$, the claim is true in Case~4.
\end{description}
Therefore, the first claim is proved for all cases.
Next, we prove the second claim by induction on $t$.
Suppose that $T$ has a greedy drawing with an open angle.
If $t=2$, then $T$ must be a single edge (type-$A$ tree), a triple (type-$B_1$ tree), or a quadruple (type-$C_{0,1}$ tree), and hence the claim is true.
Now we assume that  $t > 2$.
By Lemma~\ref{lem:combination}, $T$ must be one of the forms in Case~(I)--(IV) in Lemma~\ref{lem:combination}. 
We consider each case separately.

\begin{description}[itemsep=0cm,leftmargin=0cm,topsep=1ex]
\item[Case 1:  $T$ is of the form in Case~(I).]
By induction hypothesis, $T_1$ must be one of the trees appearing in Cases~(B)--(E).
Suppose $T_1$ is the type-$B_n$ tree. Then, $T$ is the type-$B_{n+1}$ tree.
If $T_1$ is  a type-$C_{k,n}$, $D_{k,l,n}$, or $E_{k,l,n}$ tree,
then $T$ is also a type-$C_{k,n}$, $D_{k,l,n}$, or $E_{k,l,n}$ tree respectively.
Therefore, in every case, $T$ is one of the trees appearing in Cases~(B)--(E).

\item[Case 2:  $T$ is of the form in Case~(II).]
By induction hypothesis, $T_1$ must be one of the trees appearing in Cases~(B)--(E).
Suppose $T_1$ is the type-$B_n$ tree. Then, $T$ is a type-$C_{n,1}$ tree.
If $T_1$ is a type-$C_{k,n}$, $D_{k,l,n}$, or $E_{k,l,n}$ tree,
then $T$ is a type-$C_{k,n+1}$, $D_{k,l,n+1}$, or $E_{k,l,n+1}$ tree respectively.
Thus,  in every case, $T$ is one of the trees appearing in Cases~(B)--(E).

\item[Case 3:  $T$ is of the form in Case~(III) or (IV).]
Since $|\angle{T_1}|_*,|\angle{T_2}|_* > 90^\circ$, subtrees $T_1$ and $T_2$ must be type-$B_n$ and type-$B_{n'}$ trees for some $n,n' \geq 1$ by induction hypothesis.
Thus, $T$ is a type-$D_{n,n',0}$ tree in Case~(III) and a type-$E_{n,n',0}$ tree in Case~(IV).
Thus,  in every case, $T$ is one of the trees appearing in Cases~(D) and (E).
\end{description}
Therefore, the second claim is proved for all cases.
\end{proof}
\\
\\
With this proposition, we can completely determine the rooted trees that have greedy drawings with opening angles greater than $7.5^\circ$.
The classification is described in Table~\ref{table:open_angle}.

\begin{table}[h]
\centering
  \begin{tabular}{|c|c|}
  \hline
   $|\angle{T}|_*=180^\circ$ & $A$ \\
   \hline
   $|\angle{T}|_*=(120^\circ)^-$ & $B_1$ \\
     \hline
   $|\angle{T}|_*=(105^\circ)^-$ & $B_2$ \\
     \hline
   $90^\circ < |\angle{T}|_* \leq 97.5^\circ$ & $B_n$ ($n \geq 3$) \\
     \hline
   $|\angle{T}|_* = (60^\circ)^-$ & $C_{0,1}$, $C_{1,1}$, $D_{1,1,0}$ \\
     \hline
   $45^\circ < |\angle{T}|_* < 60^\circ$ & $C_{k,1}$ ($k \geq 2$) \\
   \hline
   $|\angle{T}|_* = (45^\circ)^-$ & $D_{1,2,0}$ \\
     \hline
   $|\angle{T}|_*=(37.5^\circ)^-$ & $D_{1,3,0}$, $E_{1,1,0}$  \\
     \hline
   $30^\circ < |\angle{T}|_* \leq 33.75^\circ$ &  $D_{1,l,0}$ ($l \geq 4$) \\
   \hline
   $|\angle{T}|_*=(30^\circ)^-$ & $C_{0,2}$, $C_{1,2}$, $D_{2,2,0}$, $D_{1,1,1}$, $E_{1,2,0}$ \\
     \hline
    $22.5^\circ < |\angle{T}|_* \leq 26.25^\circ$ & $C_{k,2}$ ($k \geq 2$), $E_{1,l,0}$ ($l \geq 3$)  \\
     \hline
   $15^\circ < |\angle{T}|_* \leq 22.5^\circ$ & $D_{1,l,1}$ ($l \geq 2$), $D_{2,l,0}$ ($l \geq 3$),  $E_{1,1,1}$, $E_{2,2,0}$ 
   \\ \hline
   $|\angle{T}|_* = (15^\circ)^-$ & $C_{0,3}$, $C_{1,3}$,  $D_{3,3,0}$, $D_{2,2,1}$, $D_{1,1,2}$ $E_{2,3,0}$,  $E_{1,2,1}$
   \\ \hline
   $|\angle{T}|_* = (13.125^\circ)^-$ & $C_{2,3}$, $E_{2,4,0}$, $E_{1,3,1}$
   \\ \hline
   $11.25^\circ < |\angle{T}|_* < 13.125^\circ$ & $C_{k,3}$  ($k \geq 3$),  $E_{2,l,0}$ ($l \geq 5$), $E_{1,l,1}$ ($l \geq 4$)  
   \\ \hline
   $7.5^\circ < |\angle{T}|_* \leq 11.25^\circ$ &  $D_{3,l,0}$ ($l \geq 4$), $D_{2,l,1}$ ($l \geq 3$), $D_{1,l,2}$ ($l \geq 2$), $E_{3,3,0}$, $E_{2,2,1}$, $E_{1,1,2}$ \\
     \hline
  \end{tabular}
  \caption{Classification of rooted trees with the supremum of opening angles $> 7.5^\circ$}
    \label{table:open_angle}
\end{table}
For $I \subset [0^\circ, 180^\circ]$, we say that a rooted tree $T$ has \emph{angle type $I$} if the supremum of opening angles of $T$ is in the range $I$. 
For any $I \subset (7.5^\circ, 180^\circ]$, we can obtain an explicit description of the rooted trees of angle type $I$ using Table~\ref{table:open_angle}.
For example, the root trees having angle type $[45^\circ, 90^\circ]$ are $C_{k,1}$ ($k \geq 0$), $D_{1,1,0}$, $D_{1,2,0}$.
We denote by $\Phi$ the set of the suprema of opening angles of rooted trees, and by $\tilde{\Phi}$ the topological closure of $\Phi$.

N\"{o}llenburg and Prutkin~\cite[Proposition~6.2]{NP17} implicitly proved that any greedy-drawable tree can be constructed by conjoining some rooted trees that can be drawn with open angles.
Here, we make this fact more explicit.
Let $T_1',T'_2$ be subtrees of a tree $T$ with degree-$1$ roots $r'_1,r'_2$, and $v'_1,v'_2$ be the neighbors of $r'_1,r'_2$ respectively.
The subtrees $T'_1$ and $T'_2$ are said to be \emph{independent} if $T'_2 \setminus \{ r'_2 \} \subset T^{r'_1}_{v'_1r'_1}$ and $T'_1 \setminus \{ r'_1\} \subset T^{r'_2}_{v'_2r'_2}$.
(It may be the case that $v'_1 = r'_2$ and $v'_2 = r'_1$.)
We use the following fact, proved in~\cite{NP17}.
\begin{lem}(\cite[Lemma~2.15]{NP17})
Let $T$ be a greedy-drawable tree and $T'_1, T'_2$ be independent subtrees of $T$.
In any greedy drawing of $T$, we have $|\angle{T'_1}|>0$ or $|\angle{T'_2}|>0$.
\label{lem:open_closed}
\end{lem}
For an edge $uv$ of a tree $T$, we let $\widetilde{T}_{uv}^v \coloneqq T_{uv}^v + uv$ and regard $\widetilde{T}_{uv}^v$ as a rooted tree with root $u$.
\begin{lem}
Let $T=(V(T),E(T))$ be a greedy-drawable tree. Then, there exists a vertex $u \in V(T)$ such that $|\angle{\widetilde{T}_{uv}^v}|_* > 0 $  for every neighboring vertex $v$ of $u$. 
\label{lem:all_open_angles}
\end{lem}
\begin{proof}
For contradiction, we assume that, for any vertex $u$ of $T$, there exists a neighboring vertex $v$ of $u$ with $|\angle{\widetilde{T}_{uv}^v}|_* \leq 0$.
Consider a directed graph $D(T)$ with the vertex set $V(T)$ such that $(u,v) \in V(T) \times V(T)$ is an arc if and only if $uv \in E(T)$ and $|\angle{\widetilde{T}_{uv}^v}|_* \leq 0$.
By the assumption, there exists an edge outgoing from $w$ for any vertex $w$ in $D(T)$.
Thus, a directed cycle must exist in $D(T)$.
Since $T$ is a tree, the directed cycle must contain a directed cycle of length $2$.
If we denote such a cycle  by $a \rightarrow b \rightarrow a$,
then we have $|\angle{\widetilde{T}_{ab}^a}|_* \leq 0$ and $|\angle{\widetilde{T}_{ab}^b}|_*  \leq 0$.
By Lemma~\ref{lem:open_closed}, this is a contradiction.
\end{proof}

\section{Greedy-drawable trees with maximum degree $\leq 4$}
\label{sec:deg4}
N\"{o}llenburg and Prutkin~\cite{NP17} provided a simple characterization of greedy-drawable trees with maximum degree $\leq 4$ in terms of opening angles and a linear-time recognition algorithm based on the characterization.
In this section, we present an explicit description of greedy-drawable trees with maximum degree $\leq 4$, based on the results of N\"{o}llenburg and Prutkin~\cite{NP17} and those in Section~\ref{sec:open_angle}.

Let $T$ be a greedy-drawable tree with maximum degree $\leq 4$.
By Lemma~\ref{lem:all_open_angles}, there is a vertex $r$ of $T$ with degree $d \leq 4$ such that the subtrees  $T_0 \coloneqq T_{rv_0}^{v_0}+rv_0,\dots,T_{d-1} \coloneqq T_{rv_{d-1}}^{v_{d-1}}+rv_{d-1}$ can be drawn with open angles, 
where $v_0,\dots,v_{d-1}$ are the neighbors of $r$.
Applying Lemma~\ref{lem:shrinking} to a greedy drawing of $T$, we construct a drawing of $T$ in which each $T_i$ is infinitesimally small and each $\angle{T_i}$ contains the original angle.
In this drawing, each $\angle{T_i}$ must contain the $d$-gon formed by the vertices $v_0,\dots,v_{d-1}$, and thus each apex of $\angle{T_j}$ ($i \neq j$).
Therefore, the following inequality holds:
\begin{equation}
\label{angle_ineq}
 \sum_{i=0}^{d-1}{|\angle{T_i}|_*} > (d-2)180^\circ. 
\end{equation}
N\"{o}llenburg and Prutkin~\cite[Lemmas~4.5--4.7]{NP17} proved a beautiful fact that $T$ is greedy-drawable if and only if the above inequality holds.
Using Proposition~\ref{prop:open_angle}, we can classify the possible combinations of angle types of $T_i$'s.
Combining the results for $d=2,3,4$, we obtain the following proposition.
\begin{prop}
\label{prop:deg4}
A tree $T$ with maximum degree $d \leq 4$ is greedy-drawable if and only if $T$ is a subgraph of a subdivision of a tree obtained by joining four trees $T_0,\dots,T_3$ that have angle types listed in Table~\ref{table:degree4_angle} at a single vertex; that is,
$T$ is a subgraph of a subdivision of a tree obtained by combining four trees $T_0,\dots,T_3$ as described in Table~\ref{table:degree4_main}.
\end{prop}

\begin{table}[h]
\centering
\scalebox{0.8}{
  \begin{tabular}{|c|c|c|c|}
  \hline
    $T_0$ & $T_1$ & $T_2$ & $T_3$ \\
    \hline
    $180^\circ$ & $180^\circ$& $(0^\circ,180^\circ]$ & $(0^\circ,180^\circ]$ \\
    \hline
    $180^\circ$ & $120^\circ$ & $60^\circ$ &  $(0^\circ,60^\circ]$ \\
     \hline
    $180^\circ$ & $120^\circ$ & $52.5^\circ$ &  $(7.5^\circ,52.5^\circ]$ \\
     \hline
    $180^\circ$ & $120^\circ$ & $48.75^\circ$ &  $(11.25^\circ,48.75^\circ]$ \\
     \hline
      $180^\circ$ & $120^\circ$ & $(45^\circ,46.875^\circ]$ & $[15^\circ, 46.875^\circ]$ \\
    \hline
      $180^\circ$ & $120^\circ$ & $45^\circ$ & $(15^\circ, 45^\circ]$ \\
    \hline
     $180^\circ$ & $120^\circ$ & $37.5^\circ$ & $(22.5^\circ,37.5^\circ]$\\
    \hline
     $180^\circ$ & $120^\circ$ & $(30^\circ,33.75^\circ]$ & $[30^\circ ,33.75^\circ]$\\
    \hline
     $180^\circ$ & $(90^\circ, 105^\circ]$ & $(90^\circ, 105^\circ]$ & $(0^\circ,105^\circ]$ \\
    \hline
     $180^\circ$ & $105^\circ$ & $60^\circ$ & $(15^\circ,60^\circ]$\\
    \hline
     $180^\circ$ & $105^\circ$ & $52.5^\circ$ & $(22.5^\circ,52.5^\circ]$\\
    \hline
     $180^\circ$ & $105^\circ$ & $(45^\circ,48.75^\circ]$ &  $[30^\circ, 48.75^\circ]$ \\
    \hline
     $180^\circ$ & $105^\circ$ & $45^\circ$ &  $(30^\circ, 45^\circ]$ \\
    \hline
     $180^\circ$ & $97.5^\circ$ &  $60^\circ$ &  $(22.5^\circ, 60^\circ]$ \\
    \hline
     $180^\circ$ & $97.5^\circ$ &  $52.5^\circ$ &  $(30^\circ, 52.5^\circ]$ \\
    \hline
     $180^\circ$ & $97.5^\circ$ &  $(45^\circ, 48.75^\circ]$ &  $[37.5^\circ, 48.75^\circ]$ \\
      \hline
\end{tabular}

  \begin{tabular}{|c|c|c|c|}
  \hline
    $T_0$ & $T_1$ & $T_2$ & $T_3$ \\
    \hline
      $180^\circ$ & $(90^\circ, 97.5^\circ]$ &  $45^\circ$ &  $45^\circ$ \\
    \hline
       $180^\circ$ & $(90^\circ, 93.75^\circ]$ &  $60^\circ$ & $[30^\circ, 60^\circ]$  \\   
      \hline
       $180^\circ$ & $(90^\circ, 93.75^\circ]$ &  $52.5^\circ$ & $[37.5^\circ, 52.5^\circ]$  \\   
      \hline
       $180^\circ$ & $(90^\circ, 93.75^\circ]$ &  $[45^\circ, 48.75^\circ]$ & $[45^\circ, 48.75^\circ]$  \\   
      \hline
      $120^\circ$ & $120^\circ$ & $120^\circ$ & $(0^\circ, 120^\circ]$ \\
    \hline
       $120^\circ$ & $120^\circ$ & $105^\circ$ & $(15^\circ, 105^\circ]$ \\
       \hline
        $120^\circ$ & $120^\circ$ & $97.5^\circ$ & $(22.5^\circ, 97.5^\circ]$ \\
        \hline
         $120^\circ$ & $120^\circ$ & $(90^\circ, 93.75^\circ]$  & $[30^\circ, 93.75^\circ]$ \\
         \hline
         $120^\circ$ & $105^\circ$ & $105^\circ$  & $(30^\circ, 105^\circ]$ \\
          \hline
         $120^\circ$ & $105^\circ$ & $(90^\circ,97.5^\circ]$  & $[45^\circ, 97.5^\circ]$ \\
        \hline
           $120^\circ$ & $97.5^\circ$ &  $97.5^\circ$ & $(45^\circ, 97.5^\circ]$ \\
         \hline
           $120^\circ$ & $97.5^\circ$ &  $(90^\circ,93.75^\circ]$ & $[52.5^\circ, 93.75^\circ]$ \\
         \hline
           $120^\circ$ & $(90^\circ, 93.75^\circ]$ &  $(90^\circ,93.75^\circ]$ & $[60^\circ, 93.75^\circ]$ \\
         \hline
           $105^\circ$ & $105^\circ$ & $105^\circ$  & $(45^\circ, 105^\circ]$ \\
           \hline
           $105^\circ$ & $105^\circ$ & $(90^\circ,97.5^\circ]$  & $[60^\circ, 97.5^\circ]$ \\
           \hline
           $(90^\circ, 105^\circ]$ &  $(90^\circ, 105^\circ]$ & $(90^\circ, 105^\circ]$ & $(90^\circ, 105^\circ]$\\
           \hline          
  \end{tabular}
}
  \caption{Angle types of  subtrees  $T_0,\dots,T_3$ of greedy-drawable trees with maximum degree 4}
  \label{table:degree4_angle}
\end{table}

\begin{table}[h]
\centering
\scalebox{0.8}{
  \begin{tabular}{|c|c|c|c|}
  \hline
    $T_0$ & $T_1$ & $T_2$ & $T_3$ \\
    \hline
   $A$ & $A$ & $E_{k,l,n}$ & $E_{k,l,n}$ \\
\hline
$A$ & $B_1$ & $C_{1,1},D_{1,1,0}$ & $E_{k,l,n}$ \\
\hline
$A$ & $B_1$ & $C_{2,1}$ & $C_{k,3}, D_{3,l,0}, D_{2,l,1}, D_{1,l,2}, E_{3,3,0}, E_{2,2,1}, E_{1,1,2},  E_{2,l,0}, E_{1,l,1}$ \\
\hline
$A$ & $B_1$ & $C_{3,1}$ & $C_{k,3},  D_{3,3,0}, D_{2,2,1}, D_{1,1,2}, E_{2,l,0}, E_{1,l,1}$ \\
\hline
$A$ & $B_1$ & $C_{k,1}$ & $C_{1,3}, D_{3,3,0}, D_{2,2,1}, D_{1,1,2}, D_{1,l,1}, D_{2,l,0}, E_{2,3,0}, E_{1,2,1},  E_{1,l,0}$ \\
\hline
$A$ & $B_1$ & $D_{1,2,0}$ & $C_{k,2}, D_{2,l,0}, D_{1,l,1}, E_{2,2,0},E_{1,1,1}, E_{1,l,0}$\\
\hline
$A$ & $B_1$ & $D_{1,3,0}, E_{1,1,0}$ & $C_{k,2}, ,D_{2,2,0}, D_{1,1,1}, E_{1,l,0}$\\
\hline
$A$ & $B_1$ & $D_{1,l,0}$ & $C_{1,2}, D_{2,2,0}, D_{1,1,1}, D_{1,l,0}, E_{1,2,0}$ \\
\hline
$A$ & $B_n$ & $B_n$ & $E_{k,l,n}$ \\
\hline
$A$ & $B_2$ & $C_{1,1},D_{1,1,0}$ & $C_{k,2}, D_{2,l,0}, D_{1,l,1}, E_{2,2,0}, E_{1,1,1}, E_{1,l,0}$\\
\hline
$A$ & $B_2$ & $C_{2,1}$ & $C_{k,2}, D_{2,2,0}, D_{1,1,1}, E_{1,l,0}$ \\
\hline
$A$ & $B_2$ & $C_{k,1}$ & $C_{1,2}, D_{2,2,0}, D_{1,1,1},  D_{1,l,0}, E_{1,2,0}$ \\
\hline
$A$ & $B_2$ & $D_{1,2,0}$ & $D_{1,l,0},E_{1,1,0}$\\
\hline
$A$ & $B_3$ & $C_{1,1},D_{1,1,0}$ & $C_{k,2}, D_{2,2,0}, D_{1,1,1}, E_{1,l,0}$\\
\hline
$A$ & $B_3$ & $C_{2,1}$ & $C_{k,1}, D_{1,l,0}, E_{1,1,0}$\\
\hline
$A$ & $B_3$ & $C_{k,1}$ & $C_{k,1}, D_{1,3,0}, E_{1,1,0}$\\
\hline
$A$ & $B_n$ & $D_{1,2,0}$ & $D_{1,2,0}$\\
\hline
$A$ & $B_n$ & $C_{1,1},D_{1,1,0}$ & $C_{1,2}, D_{2,2,0}, D_{1,1,1},  D_{1,l,0}, E_{1,2,0}$\\
\hline
$A$ & $B_n$ & $C_{2,1}$ & $C_{k,1}, D_{1,3,0}, E_{1,1,0}$\\
\hline
$A$ & $B_n$ & $C_{k,1}, D_{1,2,0}$ & $C_{k,1}, D_{1,2,0}$\\
\hline
$B_1$ & $B_1$ & $B_1$ & $E_{k,l,n}$ \\
\hline
$B_1$ & $B_1$ & $B_2$ & $C_{k,2}, D_{2,l,0}, D_{1,l,1}, E_{2,2,0}, E_{1,1,1} E_{1,l,0}$\\
\hline
$B_1$ & $B_1$ & $B_3$ & $C_{k,2}, D_{2,2,0},  D_{1,1,1}, E_{1,l,0}$ \\
\hline
$B_1$ & $B_1$ & $B_n$ & $C_{1,2}, D_{2,2,0}, D_{1,1,1}, D_{1,l,0}, E_{1,2,0}$\\
\hline
$B_1$ & $B_2$ & $B_2$ & $C_{k,1}, D_{1,l,0},E_{1,1,0}$\\
\hline
$B_1$ & $B_2$ & $B_n$ & $C_{k,1}, D_{1,2,0}$\\
\hline
$B_1$ & $B_3$ & $B_3$ & $C_{k,1}, D_{1,1,0}$\\
\hline
$B_1$ & $B_3$ & $B_n$ & $C_{2,1}, D_{1,1,0}$\\
\hline
$B_1$ & $B_n$ & $B_n$ & $C_{1,1},D_{1,1,0}$\\
\hline
$B_2$ & $B_2$ & $B_2$ & $C_{k,1},D_{1,1,0}$\\
\hline
$B_2$ & $B_2$ & $B_n$ & $C_{1,1},D_{1,1,0}$\\
\hline
$B_n$ & $B_n$ & $B_n$ & $B_n$\\
\hline
\end{tabular}
}
  \caption{Subtrees $T_0,\dots,T_3$ of maximal greedy-drawable trees with maximum degree 4}
  \label{table:degree4_main}
\end{table}

\newpage

\begin{proof}
We give a proof for the case that $|\angle{T_0}|_*=180^\circ$ and $|\angle{T_1}|_* \geq 120^\circ$.
Repeating the same discussion, we can prove the proposition.
We first assume that $|\angle{T_1}|_* = 180^\circ$, i.e., $T$ is the type-$A$ tree. 
Then, the inequality~(\ref{angle_ineq}) holds if and only if $|\angle{T_2}|_*>0^\circ$ and $|\angle{T_3}|_* > 0^\circ$, i.e., each of the trees $T_2$ and $T_3$ is a subgraph of a type-$E_{k,l,n}$ tree for some $k,l,n$.
This result corresponds to the result in the first row of Table~\ref{table:degree4_main}.
Next, let us assume that $|\angle{T_1}|_* < 180^\circ$. Then, we have $|\angle{T_1}|_* = 120^\circ$ by the result in Table~\ref{table:open_angle}, which implies that $T_1$ is the type-$B_1$ tree.
We consider several cases separately.

\begin{description}[itemsep=0cm,leftmargin=0cm,topsep=1ex]
\item[Case 1:] $60^\circ \leq |\angle{T_2}|_* \leq 120^\circ$.

In this case, $T_2$ is a subgraph of a type-$C_{1,1}$ or $D_{1,1,0}$ tree. (Here, we note that the type-$B_n$ tree is a subgraph of some type-$C_{1,1}$ tree.)
The inequality~(\ref{angle_ineq}) holds if and only if $|\angle{T_3}|_* > 0^\circ$, i.e., $T_3$ is a subgraph of a type-$E_{k,l,n}$ tree for some $k,l,n$.
Thus, we obtain the result in the second row of Table~\ref{table:degree4_main}.
\item[Case 2:]  $52.5^\circ \leq |\angle{T_2}|_* < 60^\circ$.

In this case, we have $|\angle{T_2}|_* = 52.5^\circ$ by the result in Table~\ref{table:open_angle}, and thus $T_2$ is a type-$C_{2,1}$ tree. 
The inequality~(\ref{angle_ineq}) holds  if and only if $|\angle{T_3}|_* > 7.5^\circ$, i.e., $T_3$ is a subgraph of  the following type trees: $C_{k,3}$, $D_{3,l,0}$, $D_{2,l,1}$, $D_{1,l,2}$, $E_{3,3,0}$, $E_{2,2,1}$, $E_{1,1,2}$,  $E_{2,l,0}$, $E_{1,l,1}$.
This corresponds to the result in the third row of Table~\ref{table:degree4_main}.
(Note that a type-$D_{k,l,n}$ tree is a subgraph of some type-$E_{k,l,n}$ tree, and a type-$C_{k,n}$ tree is a subgraph of some type-$D_{1,1,n}$ tree.)

\item[Case 3:] $48.75^\circ \leq |\angle{T_2}|_* < 52.5^\circ$.

In this case, we have $|\angle{T_2}|_* = 48.75^\circ$ by the result in Table~\ref{table:open_angle}, and thus $T_2$ is a type-$C_{3,1}$ tree. 
The inequality~(\ref{angle_ineq}) holds if and only if  $|\angle{T_3}|_* > 11.25^\circ$, i.e., $T_3$ is a subgraph of one of the following type trees: $C_{k,3}$, $D_{3,3,0}$, $D_{2,2,1}$, $D_{1,1,2}$, $E_{2,l,0}$, $E_{1,l,1}$.
This corresponds to the result in the fourth row of Table~\ref{table:degree4_main}.

\item[Case 4:] $45^\circ \leq |\angle{T_2}|_* < 48.75^\circ$.

In this case, we have $45^\circ \leq |\angle{T_2}|_* \leq 46.875^\circ$ by the result in Table~\ref{table:open_angle}, and thus $T_2$ is a type-$C_{k,1}$ tree for some $k \geq 4$.
Since $|\angle{T_3}|_* > 13.125^\circ$ is satisfied only for trees $T_3$ with $|\angle{T_3}|_* \geq 15^\circ$ (see Table~\ref{table:open_angle}),
the inequality~(\ref{angle_ineq}) holds if and only if $|\angle{T_3}|_* \geq 15^\circ$, i.e., $T_3$ is a subgraph of one of the following type trees: $C_{1,3}$, $D_{3,3,0}$, $D_{2,2,1}$, $D_{1,1,2}$, $D_{1,l,1}$, $D_{2,l,0}$, $E_{2,3,0}$, $E_{1,2,1}$, $E_{1,l,0}$.
This corresponds to the result in the fifth row of Table~\ref{table:degree4_main}. 

\item[Case 5:] $37.5^\circ \leq |\angle{T_2}|_* < 45^\circ$.

In this case, we have $|\angle{T_2}|_* = 37.5^\circ$ (see Table~\ref{table:open_angle}), and thus $T_2$ is a type-$D_{1,3,0}$ or $E_{1,1,0}$ tree. 
The inequality~(\ref{angle_ineq}) holds if and only if $|\angle{T_3}|_* > 22.5^\circ$, i.e., $T_3$ is a subgraph of one of the following type trees: 
$C_{k,2}$, $D_{2,2,0}$, $D_{1,l,1}$, $E_{1,l,0}$.
This corresponds to the result in the sixth row of Table~\ref{table:degree4_main}. 

\item[Case 6:] $30^\circ \leq |\angle{T_2}|_* < 37.5^\circ$.

In this case, we have $30^\circ \leq |\angle{T_2}|_* \leq 33.75^\circ$ (see Table~\ref{table:open_angle}), and thus $T_2$ is a type-$D_{1,l,0}$ tree for some $l \geq 4$.  
Since $|\angle{T_3}|_* > 26.25^\circ$ is satisfied only for trees $T_3$ with $|\angle{T_3}|_* \geq 30^\circ$,
the inequality~(\ref{angle_ineq}) holds if and only if $|\angle{T_3}|_* \geq 30^\circ$, i.e., $T_3$ is a subgraph of one of the following type trees: 
$C_{1,2}$, $D_{2,2,0}$, $D_{1,1,1}$, $D_{1,l,0}$, $E_{1,2,0}$.
This corresponds to the result in the seventh row of Table~\ref{table:degree4_main}. 

\item[Case 7:] $|\angle{T_2}|_* < 30^\circ$. 

Since $|\angle{T_3}|_* \leq |\angle{T_2}|_* < 30^\circ$, the inequality~(\ref{angle_ineq}) does not hold for any choice of trees $T_2$ and $T_3$.
\end{description}
By repeating the same discussion, we can prove the proposition.
\end{proof}

\section{Greedy-drawable trees with maximum degree $5$}
\label{sec:deg5}
This section presents an explicit description of greedy-drawable trees with maximum degree $5$.
Unlike the case of maximum degree $\leq 4$, greedy drawability cannot be characterized by the inequality~(\ref{angle_ineq})
and a complete combinatorial characterization is left open in~\cite{NP17}.

Let $T$ be a tree with maximum degree $5$.
Take a degree-$5$ vertex $r$, and let $v_0, \dots, v_4$ be the neighbors of $r$, and consider the rooted trees $T_i \coloneqq T^{v_i}_{rv_i} +rv_i \ (i = 0, \dots, 4)$.
If $T$ can be drawn greedily, each $T_i$ must satisfy $|\angle{T_i}|_* > 0$.
Indeed, if $|T_j|_* \leq 0$ for some $j$, then the digraph $D(T)$, defined in the proof of Lemma~\ref{lem:all_open_angles}, contains a directed cycle $r \rightarrow v_j \rightarrow r$, which is a contradiction.
Thus, $T$ is greedy-drawable if and only if $|\angle{T_i}|_* >0$ for each $i =0,\dots,4$, and
the inequality system~(\ref{ineq_main}) is feasible for some permutation $\tau$ on $\{ 0,\dots, 4\}$, where $\varphi_0=|\angle{T_0}|_*,\dots,\varphi_4=|\angle{T_4}|_*$.
We prove the following theorem:
\begin{thm}
\label{thm:main}
Let $T$ be a tree with a degree-$5$ vertex $r$.
Let $v_0, \dots, v_4$ be the neighboring vertices of $r$, and $T_i \coloneqq T^{v_i}_{rv_i} + rv_i$ for $i=0,\dots,4$.
The tree $T$ is greedy-drawable if and only if the subtrees $T_0,\dots,T_4$ are subgraphs of subdivisions of trees listed in Table~\ref{table:degree5_main}.
\end{thm}
\begin{table}[htb]
\centering
      \scalebox{0.8}{
  \begin{tabular}{|c|c|c|c|c|}
  \hline
    $T_0$ & $T_1$ & $T_2$ & $T_3$ & $T_4$ \\
    \hline
$A$ & $A$ & $B_1$ & $B_1$ & $E_{k,l,n}$ \\
\hline
$A$ & $A$ & $B_1$ & $B_2$ & $C_{k,2}, D_{2,l,0}, D_{1,l,1}, E_{2,2,0}, E_{1,1,1}, E_{1,l,0}$\\
\hline
$A$ & $A$ & $B_1$ & $B_3$ & $C_{k,2}, D_{2,2,0}, D_{1,1,1}, E_{1,l,0}$\\
\hline
$A$ & $A$ & $B_1$ & $B_n$ & $C_{1,2}, D_{2,2,0}, D_{1,1,1}, D_{1,l,0} E_{1,2,0}$\\
\hline
$A$ & $A$ & $B_2$ & $B_2$ & $C_{k,1}, D_{1,l,0},E_{1,1,0}$\\
\hline
$A$ & $A$ & $B_2$ & $B_n$ & $C_{k,1}, D_{1,2,0}$\\
\hline
$A$ & $A$ & $B_3$ & $B_3$ & $C_{k,1}, D_{1,1,0}$\\
\hline
$A$ & $A$ & $B_3$ & $B_n$ & $C_{2,1}, D_{1,1,0}$\\
\hline
$A$ & $A$ & $B_n$ & $B_n$ & $C_{1,1},D_{1,1,0}$\\
    \hline
$B_1$ & $B_1$ & $B_1$ & $B_n$ & $B_n$ \\
\hline
$B_1$ & $B_1$ & $B_2$ & $B_2$ & $B_n$\\
\hline
  \end{tabular}
}
\scalebox{0.8}{
  \begin{tabular}{|c|c|c|c|c|}
  \hline
    $T_0$ & $T_1$ & $T_2$ & $T_3$ & $T_4$ \\
    \hline
$A$ & $B_1$ & $B_1$ & $B_1$ & $ C_{k,1}, D_{1,4,0}, E_{1,1,0}$\\
\hline
$A$ & $B_1$ & $B_1$ & $B_2$ & $C_{k,1}, D_{1,2,0}$\\
\hline
$A$ & $B_1$ & $B_1$ & $B_3$ & $C_{4,1},D_{1,1,0}$\\
\hline
$A$ & $B_1$ & $B_1$ & $B_4$ & $C_{3,1},D_{1,1,0}$\\
\hline
$A$ & $B_1$ & $B_1$ & $B_n$ & $C_{2,1}, D_{1,1,0}$\\
\hline
$A$ & $B_1$ & $B_2$ & $B_4$ & $C_{1,1},D_{1,1,0}$\\
\hline
$A$ & $B_1$ & $B_n$ & $B_n$ & $B_n$\\
\hline
$A$ & $B_2$ & $B_3$ & $B_n$ & $B_n$\\
\hline
$A$ & $B_2$ & $B_4$ & $B_4$ & $B_n$\\
\hline
$A$ & $B_2$ & $B_4$ & $B_5$ & $B_5$\\
\hline
$A$ & $B_3$ & $B_3$ & $B_3$ & $B_6$\\
\hline
  \end{tabular}
}  \caption{Subtrees $T_0,\dots,T_4$ of maximal greedy-drawable trees with maximum degree $5$}
  \label{table:degree5_main}
\end{table}
We prove Theorem~\ref{thm:main} by distinguishing the following two cases.
\subsection{Case 1:  $\varphi_0 \neq 180^\circ$ or $\varphi_0 = \varphi_1 = 180^\circ$}
In this case, the conditions for greedy drawability are given in \cite{NP17}, which are summarized as follows:
\begin{itemize}
\item If $\varphi_0 \neq 180^\circ$, $T$ is greedy-drawable if and only if $\varphi_1,\dots,\varphi_4 > 0$ and $\sum_{i=0}^4{\varphi_i} > 540^\circ$ (\cite[Lemma~4.6]{NP17}).
\item If $\varphi_0 = \varphi_1 = \varphi_2 = 180^\circ$, $T$ is greedy-drawable if and only if  $\varphi_3,\varphi_4 > 0$ and $\varphi_3 + \varphi_4 > 120^\circ$ (\cite[Lemmas~4.8 and 4.9]{NP17}).
\item If $\varphi_0 = \varphi_1 = 180^\circ$ and $\varphi_2 \neq 180^\circ$, $T$ is greedy-drawable if and only if $\varphi_2,\varphi_3,\varphi_4 > 0$ and $\varphi_2 + \varphi_3 + \varphi_4 > 240^\circ$  (\cite[Lemma~4.10]{NP17}).
\end{itemize}
Using these results and Proposition~\ref{prop:open_angle}, we can explicitly describe the possible angle types of $T_0,\dots,T_4$ as listed in Table~\ref{table:degree5_case1}.
\begin{table}[h]
\centering
      \scalebox{0.8}{
  \begin{tabular}{|c|c|c|c|c|}
  \hline
    $T_0$ & $T_1$ & $T_2$ & $T_3$ & $T_4$ \\
    \hline
    $180^\circ$ & $180^\circ$  & $[120^\circ,180^\circ]$ & $[120^\circ, 180^\circ]$ & $(0^\circ,180^\circ]$ \\
    \hline
 $180^\circ$ & $180^\circ$  & $[120^\circ,180^\circ]$ & $105^\circ$ & $(15^\circ,105^\circ]$\\
\hline
 $180^\circ$ & $180^\circ$  & $[120^\circ,180^\circ]$ & $97.5^\circ$ & $(22.5^\circ, 97.5^\circ]$\\
\hline
 $180^\circ$ & $180^\circ$  & $[120^\circ,180^\circ]$ & $(90^\circ,93.75^\circ]$ & $[30^\circ, 93.75^\circ]$\\
\hline
 $180^\circ$ & $180^\circ$  & $105^\circ$ & $105^\circ$ & $(30^\circ,105^\circ]$\\
\hline
 $180^\circ$ & $180^\circ$  & $105^\circ$ & $(90^\circ,97.5^\circ]$ & $[45^\circ,97.5^\circ]$ \\
\hline
 $180^\circ$ & $180^\circ$  & $97.5^\circ$ & $97.5^\circ$ & $(45^\circ,97.5^\circ]$ \\
\hline
 $180^\circ$ & $180^\circ$  & $97.5^\circ$ & $(90^\circ,93.75^\circ]$ & $[52.5^\circ,93.75^\circ]$ \\
\hline
 $180^\circ$ & $180^\circ$  & $(90^\circ,93.75^\circ]$ & $(90^\circ,93.75^\circ]$ & $[60^\circ,93.75^\circ]$ \\
\hline
    $120^\circ$ & $120^\circ$ & $120^\circ$ & $(90^\circ, 120^\circ]$ & $(90^\circ, 120^\circ]$ \\
    \hline
    $120^\circ$ & $120^\circ$ & $105^\circ$ & $105^\circ$ & $(90^\circ,105^\circ]$ \\
     \hline 
  \end{tabular}
}
  \caption{Possible angle types of $T_0,\dots,T_4$ ($\varphi_0 = \varphi_1 = 180^\circ$ or $\varphi_0 \leq 120^\circ$)}
  \label{table:degree5_case1}
\end{table}

\subsection{Case 2: $\varphi_0 = 180^\circ$ and $\varphi_1,\dots,\varphi_4 \neq 180^\circ$}
Characterizing greedy-drawable trees in this case was left open in \cite{NP17}.
We resolve this case by proving the following proposition.
\begin{prop}
Suppose that $\varphi_0 = 180^\circ$ and $\varphi_1,\dots,\varphi_4 \neq 180^\circ$.
Then, the tree $T$ is greedy-drawable if and only if $T_0,\dots,T_4$ have one of the angle types listed in Table~\ref{table:degree5_case3}.
\label{degree5_main_prop}
\end{prop}
\begin{table}[htb]
\centering
  \begin{tabular}{|c|c|c|c|c|c|c|}
  \hline
  &  $T_0$ & $T_1$ & $T_2$ & $T_3$ & $T_4$ \\
    \hline
  I &  $180^\circ$ & $120^\circ$ &  $120^\circ$ &  $120^\circ$ & $[33.75^\circ, 120^\circ]$  \\
    \hline
  II &  $180^\circ$ & $120^\circ$ &  $120^\circ$ &  $105^\circ$ & $[45^\circ, 105^\circ]$ \\
    \hline
 III &  $180^\circ$ & $120^\circ$ &  $120^\circ$ &  $97.5^\circ$ & $[46.875^\circ, 97.5^\circ]$  \\
    \hline
 IV &  $180^\circ$ & $120^\circ$ &  $120^\circ$ &  $93.75^\circ$ & $[48.75^\circ, 93.75^\circ]$  \\
    \hline
 V &  $180^\circ$ & $120^\circ$ &  $120^\circ$ &  $(90^\circ, 91.875^\circ]$ & $[52.5^\circ, 91.875^\circ]$  \\
    \hline
 VI &  $180^\circ$&  $120^\circ$ &  $105^\circ$ &  $[93.75^\circ,105^\circ]$ & $[60^\circ, 105^\circ]$  \\
    \hline
 VII &  $180^\circ$ & $120^\circ$ &  $(90^\circ,105^\circ]$ &  $(90^\circ,105^\circ]$ & $(90^\circ,105^\circ]$  \\
    \hline
 VIII & $180^\circ$ &  $105^\circ$ &  $[97.5^\circ,105^\circ]$ &  $(90^\circ,105^\circ]$ & $(90^\circ,105^\circ]$  \\
    \hline
 IX &   $180^\circ$ &  $105^\circ$ &  $93.75^\circ$ &  $93.75^\circ$ & $(90^\circ,93.75^\circ]$ \\
    \hline
 X &  $180^\circ$ & $105^\circ$ &  $93.75^\circ$ &  $91.875^\circ$ & $91.875^\circ$ \\
    \hline
  XI & $180^\circ$ & $97.5^\circ$ &  $97.5^\circ$ &  $97.5^\circ$ & $[90.9375^\circ, 97.5^\circ]$\\
    \hline
  \end{tabular}
  \caption{Possible angle types of $T_0,\dots,T_4$ ($\varphi_0 = 180^\circ$, $\varphi_1,\dots,\varphi_4 \neq 180^\circ$)}
  \label{table:degree5_case3}
\end{table}

\begin{proof}
We prove Proposition~\ref{degree5_main_prop} with assistance of computer.
To prove the if-part, it suffices to verify feasibility of the system~(\ref{ineq_main}) (for some permutation $\tau$) for each vector $(\varphi_0,\dots,\varphi_4) $ listed as follows:
\[ 
\begin{tabular}{ll}
Case I: (180,120,120,120,33.75), &
Case II: (180,120,120,105, 45), \\
Case III: (180,120,120,97.5, 46.875),  &
Case IV: (180,120,120,93.75,48.75), \\
Case V: (180,120,120,90,52.5), &
Case VI: (180,120,105,93.75, 60), \\
Case VII: (180,120,90,90, 90), &
Case VIII: (180,105,97.5,90, 90), \\
Case IX: (180,105,93.75,93.75, 90),&
Case X: (180,105,93.75,91.875, 91.875), \\
Case XI: (180,97.5, 97.5, 97.5, 90.9375). &
\end{tabular}
 \]
Since the system~(\ref{ineq_main}) contains a non-linear equation, it is often hard to find a concrete solution.
To prove feasibility, we use the intermediate value theorem (as done in~\cite{NP17}).
Let $S_{\tau}(\varphi_0,\dots,\varphi_4)$ be the set of vectors $(\beta_0,\dots,\beta_4,\gamma_0,\dots,\gamma_4)$ that satisfy the following inequalities:
\begin{align}
\begin{split}
&0 < \beta_i, \gamma_i < 180, \\
&2\beta_i + \gamma_i < 180, \ \beta_i + 2\gamma_i < 180, \\
&\beta_i + \gamma_{i+1} < \varphi_{\tau(i)},\\
&\sum_{i=0}^4{(\beta_i + \gamma_i)} = 540. \\
\end{split}
\label{ineq_sub}
\end{align}
This system is obtained by eliminating $\alpha_i$'s from the system~(\ref{ineq_main}) and omitting the non-linear equation.
When $\tau$ and $\varphi_0, \dots, \varphi_4$ are clear from the context, we simply denote the solution set by $S$.
We then let 
$\omega (\beta_0,\dots,\beta_4,\gamma_0,\dots,\gamma_4) \coloneqq \prod_{i=0}^4{\sin (\beta_i)} - \prod_{i=0}^4{\sin (\gamma_i)}$.
To verify feasibility of the original system~(\ref{ineq_main}), it suffices to find a pair of vectors $(x_+,x_-)$ such that $x_+,x_- \in \widetilde{S}$ and $\omega (x_+) > 0$, $\omega (x_-) < 0$, where 
 $\widetilde{S}$ is the topological closure of  $S$.
Vectors $x_+,x_-$ that prove greedy drawability for Cases I--IX  are listed in Table~\ref{degree5_solution}.
We found those vectors by trial and errors;
we computed vertices of the solution set $S$ (by solving linear programs with certain objective functions) and selected ones that take positive and negative values.
To obtain exact solutions to linear programs, we used QSopt\_ex Rational LP Solver~\cite{ACDE}, developed by Applegate, Cook, Dash, and Espinoza.
For those vectors, we determined the values of the function $\omega$ to be positive or negative.
Since the function $\omega$ contains trigonometric functions, it is difficult to compute exact values.
To verify that $\omega$ is negative or positive for the vectors in Table~\ref{degree5_solution}, we employed MPFI library~\cite{RM05}  for arbitrary precision interval arithmetic through the interface of SageMath~\cite{S21}.

\renewcommand{\arraystretch}{1.2}
\begin{table}[htb]
\centering
{\tabcolsep = 0.6mm
\small
 \begin{tabular}{|c|c|cccccccccc|cccccccccc|}
 \hline
 &   $\tau$ & \multicolumn{10}{c|}{$x_+$} & \multicolumn{10}{c|}{$x_-$} \\
  \hline
&  & $\beta_0$ & $\beta_1$ & $\beta_2$ & $\beta_3$ & $\beta_4$ & $\gamma_0$ & $\gamma_1$ & $\gamma_2$ & $\gamma_3$ & $\gamma_4$ & $\beta_0$ & $\beta_1$ & $\beta_2$ & $\beta_3$ & $\beta_4$ & $\gamma_0$ & $\gamma_1$ & $\gamma_2$ & $\gamma_3$ & $\gamma_4$\\
\hline 
 I &  1, 2, 3, 4, 0 & $60$ & $60$ & $60$ & $\frac{105}{4}$ &  $\frac{345}{4}$ &  $60$ & $60$ & $60$ & $60$ & $\frac{15}{2}$ & $\frac{1545}{28}$ & $\frac{705}{14}$ & $\frac{285}{7}$ & $\frac{150}{7}$ & $\frac{4695}{56}$ & $\frac{3495}{56}$ & $\frac{1815}{28}$ & $\frac{975}{14}$ & $\frac{555}{7}$ & $\frac{345}{28}$ \\
    \hline
 II &  3, 1, 2, 4, 0 & $45$ & $60$ & $60$ & $30$ & $\frac{165}{2}$ & $\frac{135}{2}$ & $60$ &  $60$ & $60$ & $15$ & $\frac{285}{7}$ & $\frac{360}{7}$ & $\frac{300}{7}$ & $\frac{180}{7}$ & $\frac{1125}{14}$ & $\frac{975}{14}$ & $\frac{450}{7}$ & $\frac{480}{7}$ & $\frac{540}{7}$ & $\frac{135}{7}$  \\
  \hline
III & 3, 1, 2, 4, 0 & $\frac{75}{2}$ & $60$ &  $60$ & $\frac{375}{8}$ & $90$ & $\frac{285}{4}$ & $60$ & $60$ & $\frac{435}{8}$ & $0$ & $\frac{1905}{56}$ &  $\frac{1485}{28}$ &  $\frac{645}{14}$ & $\frac{225}{7}$ &  $\frac{9255}{112}$ &  $\frac{8175}{112}$ & $\frac{3555}{56}$ &  $\frac{1875}{28}$ &  $\frac{1035}{14}$ &  $\frac{825}{56}$ \\
 \hline
IV & 3, 1, 2, 4, 0 & $\frac{135}{4}$ &  $60$  & $60$ & $\frac{191}{4}$ & $\frac{179}{2}$ & $\frac{585}{8}$ & $60$ & $60$ & $\frac{439}{8}$ & $1$ & $\frac{855}{28}$ & $\frac{375}{7}$ &  $\frac{330}{7}$ & $\frac{240}{7}$ & $\frac{4635}{56}$ & $\frac{4185}{56}$ &  $\frac{885}{14}$ & $\frac{465}{7}$ &  $\frac{510}{7}$ &  $\frac{405}{28}$ \\
 \hline
V & 3, 1, 2, 4, 0 & $30$ & $60$ & $60$ & $\frac{103}{2}$ & $\frac{179}{2}$ & $75$ & $60$ & $60$ & $53$ & $1$ & $\frac{375}{14}$ & $\frac{375}{7}$ & $\frac{330}{7}$ & $\frac{240}{7}$  & $\frac{2265}{28}$ & $\frac{2145}{28}$ & $\frac{885}{14}$ & $\frac{465}{7}$ & $\frac{510}{7}$ & $\frac{255}{14}$
 \\
 \hline
VI &  3, 2, 1, 4, 0 & $\frac{105}{4}$ & $45$ & $60$ & $\frac{195}{4}$ &  $\frac{675}{8}$ &  $\frac{615}{8}$ & $\frac{135}{2}$ & $60$ & $ 60$ & $\frac{45}{4}$ & $\frac{345}{14}$ & $\frac{585}{14}$ & $\frac{375}{7}$ & $\frac{330}{7}$ & $\frac{585}{7}$  & $\frac{2175}{28}$ & $\frac{1935}{28}$ & $\frac{885}{14}$ & $\frac{465}{7}$ & $\frac{90}{7}$ \\
 \hline
 VII &  2, 1, 3, 4, 0 & $15$ & $60$ & $60$ & $75$ & $\frac{165}{2}$ & $\frac{165}{2}$ & $60$ & $60$ & $30$ & $15$ & $15$ & $45$ & $30$ & $60$ & $75$ & $\frac{165}{2}$ & $\frac{135}{2}$ &  $75$ & $60$ & $30$  \\
  \hline
 VIII & 3, 1, 2, 4, 0  & $\frac{75}{4}$ & $45$ & $60$ & $\frac{285}{4}$ & $\frac{645}{8}$ & $\frac{645}{8}$ & $\frac{135}{2}$ & $60$ & $\frac{75}{2}$ & $\frac{75}{4}$ & $\frac{75}{4}$ & $\frac{75}{2}$ & $45$ & $\frac{255}{4}$ & $\frac{615}{8}$ & $\frac{645}{8}$  & $\frac{285}{4}$ & $\frac{135}{2}$ & $\frac{105}{2}$ & $\frac{105}{4}$  \\
  \hline
IX & 4, 1, 2, 3, 0  & $\frac{165}{8}$ & $45$ & $60$ & $\frac{585}{8}$ & $\frac{1275}{16}$ & $\frac{1275}{16}$ & $\frac{135}{2}$ &  $60$   & $\frac{135}{4}$ & $\frac{165}{8}$ & $\frac{165}{8}$ & $\frac{165}{4}$ & $\frac{105}{2}$ & $\frac{555}{8}$ & $\frac{1245}{16}$ & $\frac{1275}{16}$ & $\frac{555}{8}$ & $\frac{255}{4}$ & $\frac{165}{4}$ & $\frac{195}{8}$  \\
 \hline
 X &  3, 1, 2, 4, 0 & $\frac{45}{2}$  & $45$ & $60$ & $\frac{585}{8}$ & $\frac{645}{8}$  & $\frac{315}{4}$ & $\frac{135}{2}$ & $60$ & $\frac{135}{4}$ & $\frac{75}{4}$ &  $\frac{45}{2}$ & $\frac{165}{4}$ & $\frac{105}{2}$ & $\frac{555}{8}$ &  $\frac{315}{4}$ & $\frac{315}{4}$ & $\frac{555}{8}$ & $\frac{255}{4}$ & $\frac{165}{4}$ & $\frac{45}{2}$  \\
  \hline
XI &  1, 2, 3, 4, 0  & $\frac{405}{16}$ & $\frac{75}{2}$ & $60$ & $\frac{285}{4}$ &  $\frac{2365}{32}$ & $\frac{2475}{32}$ & $\frac{285}{4}$ & $60$ & $\frac{75}{2}$ & $\frac{315}{16}$ & $\frac{405}{16}$  &  $\frac{285}{8}$ & $\frac{225}{4}$ & $\frac{555}{8}$ & $\frac{2535}{32}$  & $\frac{2475}{32} $ & $\frac{1155}{16}$ & $\frac{495}{8}$ & $\frac{165}{4}$ & $\frac{345}{16}$  \\
 \hline
\end{tabular}
}
\caption{$x_+$ and $x_-$ for maximal greedy-drawable trees}
\label{degree5_solution}
\end{table}
\renewcommand{\arraystretch}{1}

On the other hand, we prove the only-if part by verifying infeasibility of the system~(\ref{ineq_main}) for the maximal vectors $(\varphi_0,\dots,\varphi_4) \in \tilde{\Phi}$ (with respect to the lexicographic order) that do not correspond to the angle types listed in Table~\ref{table:degree5_case3}. The maximal vectors are described as follows:
\[
\begin{array}{l}
(180,120,120,120,31.875),  
(180,120,120,105,37.5),  
(180,120,120,97.5,45.9375), \\
(180,120,120,93.75,46.875),
(180,120,120,91.875,48.75),
(180,120,120,60, 60),  \\
(180,120,105,105,52.5), 
(180,120,105,91.875,60), 
(180,120,97.5,97.5,60), \\
(180,105,105,105,60),  
(180,105,93.75,91.875,90.9375), 
(180,105,91.875,91.875,91.875), \\
(180,97.5,97.5,93.75,93.75), 
(180,97.5,97.5,97.5,90.46875).   
\end{array}   
\]
We must verify infeasibility of the system~(\ref{ineq_main}) for each permutation $\tau$ on $\{ 0,1,\dots,4\}$ and for each vector $(\varphi_0,\dots,\varphi_4)$ listed above.
This was achieved through the following three steps.

\subsection*{Step 1. }
If $S$ is empty, then the system~(\ref{ineq_main}) is clearly infeasible.
Since $S$ is defined by linear inequalities, emptiness of $S$ can be checked by linear programming.
Using QSopt\_ex Rational LP Solver~\cite{ACDE}, we verified emptiness of $S$
except for the 14 cases listed in Table~\ref{table:remaining_cases}.
(We ignore the vectors obtained by applying permutations that correspond to rotations or reflections to those vectors)
\begin{table}
\centering
\begin{tabular}{|c|c|}
\hline
& $(\varphi_{\tau(0)}, \dots, \varphi_{\tau (4)})$ \\
\hline
(a) & $(120,120,120,31.875,180)$ \\
\hline
(b) &  $(105,120,120,37.5,180)$ \\
\hline
(c) & $(97.5, 120, 120, 45.9375, 180)$\\
\hline
(d) & $(93.75, 120, 120, 46.875, 180)$\\
\hline
(e) & $(91.875, 120, 120, 48.75, 180)$\\
\hline
(f) & $(105,120,105, 52.5, 180)$ \\
\hline
(g) & $(105, 105, 120, 52.5, 180)$ \\
\hline
\end{tabular}
\begin{tabular}{|c|c|}
\hline
& $(\varphi_{\tau(0)}, \dots, \varphi_{\tau (4)})$ \\
\hline
(h) & $(91.875,120,105,60,180)$ \\
\hline
(i) & $(60,120,91.875,105,180)$ \\
\hline 
(j) & $(91.875, 105, 120, 60, 180)$ \\
\hline 
(k) & $(97.5,97.5,120,60,180)$ \\
\hline
(l) & $(90.9375,105,93.75,91.875,180)$ \\
\hline
(m) & $(90.9375,93.75,105,91.875,180)$ \\
\hline
(o) & $(97.5,97.5,97.5,90.46875,180)$ \\
\hline
\end{tabular}
\caption{Remaining cases: $S_{\tau}(\varphi_0,\dots,\varphi_4)$ is not empty}
\label{table:remaining_cases}
\end{table}

\subsection*{Step 2.}
For the 14 cases in Table~\ref{table:remaining_cases}, the set $S$ is non-empty, and we must show that adding the equation $\omega = 0$ makes the system~(\ref{ineq_main}) infeasible.
Since $S$ is connected and $\omega$ is continuous on $S$, it can be done by verifying that the function $\omega$ is always positive or always negative on $S$.
However, since $\omega$ is a non-linear function, it is difficult to verify directly.
Therefore, we first compute the axis-aligned bounding box $B(S)$ of $S$  (by determining the maximum and minimum of $\beta_i$ and $\gamma_i$ over $S$ for all $i$) and estimate the value of $\omega$ on $B(S)$.
If $\omega$ is always positive or always negative on $B(S)$,   the system~(\ref{ineq_main}) is clearly infeasible (see Figure~\ref{fig:bounding_box1} for an intuition).
For the vectors $(\varphi_{\tau (0)},\dots,\varphi_{\tau (4)})$ in Table~\ref{table:remaining_cases},  one can easily verify that $B(S) \subset [0,90]^{10}$ (see Table~\ref{degree5_boundingbox}).
Therefore, letting $B(S) = [\beta_0^-,\beta_0^+] \times \dots \times [\gamma_4^-,\gamma_4^+]$, we have 
$\prod_{i=0}^4{\sin (\beta_i^-)} - \prod_{i=0}^4{\sin (\gamma_i^+)} \leq \omega (x) \leq\prod_{i=0}^4{\sin (\beta_i^+)} - \prod_{i=0}^4{\sin (\gamma_i^-)}$ for all $x \in B(S)$.
Thus, if $\prod_{i=0}^4{\sin (\beta_i^-)} - \prod_{i=0}^4{\sin (\gamma_i^+)}$  and $\prod_{i=0}^4{\sin (\beta_i^+)} - \prod_{i=0}^4{\sin (\gamma_i^-)}$ have the same sign,
then $\omega$ is always positive or always negative on $S$.
We determined the signs of $\prod_{i=0}^4{\sin (\beta_i^-)} - \prod_{i=0}^4{\sin (\gamma_i^+)}$  and $\prod_{i=0}^4{\sin (\beta_i^+)} - \prod_{i=0}^4{\sin (\gamma_i^-)}$ using MPFI library~\cite{RM05} through SageMath~\cite{S21}.
Computational results are summarized in Table~\ref{degree5_boundingbox}.
With the exceptions of Cases (c), (d), (e), and (j),  we were able to verify that the original system~(\ref{ineq_main}) is infeasible.

\renewcommand{\arraystretch}{1.2}
\begin{table}[htb]
\centering
{\tabcolsep = 0.1mm
\scriptsize
  \begin{tabular}{|c|cccccccccc|cccccccccc|c|}
    \hline
 &   $\beta_0^+$ & $\beta_1^+$ & $\beta_2^+$ & $\beta_3^+$ & $\beta_4^+$ & $\gamma_0^+$ & $\gamma_1^+$ & $\gamma_2^+$ & $\gamma_3^+$ & $\gamma_4^+$ & $\beta_0^-$ & $\beta_1^-$ & $\beta_2^-$ & $\beta_3^-$ & $\beta_4^-$ & $\gamma_0^-$ & $\gamma_1^-$ & $\gamma_2^-$ & $\gamma_3^-$ & $\gamma_4^-$ &  $\omega (x)$ \\
\hline 
  (a) &   $\frac{975}{16}$ & $\frac{975}{16}$ & $\frac{975}{16}$ & $\frac{255}{8}$ & $90$ &  $\frac{4185}{64}$ & $\frac{1055}{16}$ & $\frac{1935}{28}$ & $\frac{1095}{14}$ & $\frac{465}{56}$ & $\frac{1575}{32}$ & $\frac{385}{8}$ & $\frac{585}{14}$ & $\frac{165}{7}$ & $\frac{9615}{112}$ & $\frac{465}{8}$ & $\frac{465}{8}$ & $\frac{465}{8}$ & $\frac{465}{8}$ & $0$ &
 \begin{tabular}{l} $ \geq 0.04019886771016 - 3.25 \times 10^{-15}$ \\ $ \leq 0.352662021312439 + 9.76 \times 10^{-16}$ \end{tabular}\\
    \hline
(b) & 	 $45$ & $60$ & $60$ & $\frac{75}{2}$ & $90$ & $\frac{1125}{16}$ & $\frac{255}{4}$ & $\frac{465}{7}$ & $\frac{510}{7}$ & $\frac{45}{14}$ & $\frac{315}{8}$ & $\frac{105}{2}$ & $\frac{330}{7}$ & $\frac{240}{7}$ & $\frac{2475}{28}$ & $\frac{135}{2}$ & $60$ & $60$ &  $60$ & $0$ &  \begin{tabular}{l} $\geq 0.16628181995631 - 4.00 \times 10^{-15}$ \\ $\leq 0.322844500932659 + 8.88 \times 10^{-16}$ \end{tabular} \\
\hline
(c) & $\frac{375}{8}$ & $\frac{1035}{16}$ & $\frac{1995}{32}$ & $\frac{735}{16}$ & $90$ & $\frac{9945}{128}$ & $\frac{2095}{32}$ & $\frac{3735}{56}$ & $\frac{2055}{28}$ & $\frac{1425}{112}$ &  $\frac{1575}{64}$ &  $\frac{785}{16}$ &  $\frac{1305}{28}$ &  $\frac{465}{14}$ &  $\frac{18735}{224}$ &  $\frac{1065}{16}$ &  $\frac{405}{8}$ &   $\frac{885}{16}$ & $\frac{885}{16}$ & $0$ & \begin{tabular}{l} $ \geq -0.04784331869952 - 2.38 \times 10^{-15}$ \\ $ \leq 0.41994431818384 + 2.82 \times 10^{-15}$ \end{tabular}\\
\hline
(d) & $\frac{165}{4}$ & $\frac{255}{4}$ & $\frac{495}{8}$ & $\frac{375}{8}$ & $90$ &$\frac{5025}{64}$ & $\frac{1035}{16}$ & $\frac{1845}{28}$ & $\frac{1005}{14}$ & $\frac{585}{56}$ & $\frac{735}{32}$ & $\frac{405}{8}$ & $\frac{675}{14}$ & $\frac{255}{7}$ & $\frac{9495}{112}$ & $\frac{555}{8}$ & $\frac{105}{2}$ & $\frac{225}{4}$ & $\frac{225}{4}$ & $0$ &  \begin{tabular}{l} $ \geq -0.00626089365276 - 3.13 \times 10^{-15}$ \\ $ \leq 0.38064127480888 + 2.60 \times 10^{-15}$ \end{tabular}\\
\hline
(e) & $\frac{165}{4}$ & $\frac{1035}{16}$ & $\frac{1995}{32}$ & $\frac{195}{4}$ & $90$ & $\frac{2565}{32}$ & $65$ & $\frac{1845}{28}$ & $\frac{1005}{14}$ &$\frac{345}{28}$ & $\frac{315}{16}$ & $50$ & $\frac{675}{14}$ & $\frac{255}{7}$ & $\frac{4695}{56}$ & $\frac{555}{8}$ &  $\frac{405}{8}$ & $\frac{885}{16}$ & $\frac{885}{16}$ & $0$ & \begin{tabular}{l} $ \geq 0.39692841575826 - 2.70 \times 10^{-15}$ \\ $ \leq -0.05161138747194 + 3.35 \times 10^{-15}$ \end{tabular}\\
\hline
(f) &  $45$ & $60$ & $60$ & $\frac{105}{2}$ & $90$ & $\frac{1155}{16}$ & $\frac{265}{4}$ & $\frac{495}{7}$ & $\frac{465}{7}$ & $\frac{75}{14}$ & $\frac{285}{8}$ & $\frac{95}{2}$ & $\frac{270}{7}$ & $\frac{330}{7}$ &  $\frac{2445}{28}$ & $\frac{135}{2}$ & $60$ & $60$ & $45$ & $0$ & \begin{tabular}{l} $ \geq 0.42073914509770 - 3.48 \times 10^{-15}$ \\ $ \leq 0.12567774999668 + 5.19 \times 10^{-15}$ \end{tabular}\\
\hline
(g) & $45$ & $60$ & $\frac{135}{2}$ & $\frac{105}{2}$ & $90$ & $\frac{1215}{16}$ &  $\frac{285}{4}$ & $\frac{450}{7}$ & $\frac{480}{7}$ &  $\frac{135}{14}$ &  $\frac{225}{8}$ &  $\frac{75}{2}$ &  $\frac{360}{7}$ & $\frac{300}{7}$ &  $\frac{2385}{28}$ &  $\frac{135}{2}$ &  $60$ & $45$ & $\frac{105}{2}$ & $0$ & \begin{tabular}{l} $ \geq 0.44884628439700 - 3.80 \times 10^{-15}$ \\ $ \leq 0.02302084943542 + 1.68 \times 10^{-15}$ \end{tabular}\\
\hline
(h) & 	$\frac{135}{4}$ & $\frac{975}{16}$ & $\frac{1935}{32}$ & $60$ & $90$ & $\frac{315}{4}$ & $\frac{525}{8}$ & $\frac{1935}{28}$ & $\frac{885}{14}$ & $\frac{45}{7}$ & $\frac{45}{2}$ & $\frac{195}{4}$ & $\frac{585}{14}$ & $\frac{375}{7}$ & $\frac{1215}{14}$ & $\frac{585}{8}$ & $\frac{465}{8}$ & $\frac{945}{16}$ & $\frac{705}{16}$ & $0$ & \begin{tabular}{l} $ \geq 0.07059309513054 - 2.99 \times 10^{-15}$ \\ $ \leq 0.365922064215845 + 9.30 \times 10^{-16}$ \end{tabular}\\
\hline
(i) & 	$\frac{45}{14}$ & $\frac{1725}{28}$ & $\frac{3405}{56}$ & $75$ & $\frac{1215}{16}$ & $90$ & $60$ & $\frac{1005}{16}$ & $\frac{75}{2}$ & $\frac{135}{4}$ & $0$ & $\frac{915}{16}$ & $\frac{435}{8}$ & $\frac{285}{4}$ & $\frac{585}{8}$ & $\frac{2475}{28}$ & $\frac{795}{14}$ & $\frac{1635}{28}$ & $30$ & $\frac{225}{8}$ & \begin{tabular}{l} $ \geq -0.260537672566709 - 9.48 \times 10^{-16}$ \\ $ \leq -0.127529352077151 + 6.17 \times 10^{-16}$ \end{tabular}\\
\hline
(j) & 	$\frac{135}{4}$ & $\frac{975}{16}$ & $\frac{2175}{32}$ & $60$ & $90$ & $\frac{165}{2}$ & $\frac{565}{8}$ & $\frac{1755}{28}$ & $\frac{915}{14}$ & $\frac{75}{7}$ & $15$ & $\frac{155}{4}$ & $\frac{765}{14}$ & $\frac{345}{7}$ & $\frac{1185}{14}$ & $\frac{585}{8}$ & $\frac{465}{8}$ & $\frac{705}{16}$ & $\frac{825}{16}$ & $0$  & \begin{tabular}{l} $\geq -0.040708347750012 - 9.41 \times 10^{-16}$ \\ $\leq 0.38984862337978 + 5.44 \times 10^{-15}$ \end{tabular} \\
\hline
(k) & 	 $30$ & $45$ & $\frac{255}{4}$ & $60$ & $90$ & $\frac{315}{4}$ & $\frac{145}{2}$ & $\frac{855}{14}$ & $\frac{435}{7}$ & $\frac{30}{7}$ & $\frac{45}{2}$ &  $35$ & $\frac{405}{7}$ & $\frac{390}{7}$ & $\frac{615}{7}$ & $75$ & $\frac{135}{2}$ & $\frac{105}{2}$ & $\frac{225}{4}$ &$ 0$ & \begin{tabular}{l} $\geq 0.099362774740883 - 9.62 \times 10^{-16}$ \\ $\leq 0.274610072620768 + 8.01 \times 10^{-16}$  \end{tabular} \\
\hline
(l) & 	$\frac{405}{16}$ & $\frac{195}{4}$ & $\frac{495}{8}$ & $\frac{1185}{16}$ & $\frac{2595}{32}$ & $\frac{315}{4}$ & $\frac{1095}{16}$ & $\frac{495}{8}$ & $\frac{75}{2}$ & $\frac{165}{8}$ & $\frac{45}{2}$ & $\frac{345}{8}$ & $\frac{225}{4}$ & $\frac{285}{4}$ & $\frac{1275}{16}$ & $\frac{2475}{32}$ & $\frac{525}{8}$ & $\frac{225}{4}$ & $\frac{255}{8}$ & $\frac{285}{16}$ & \begin{tabular}{l} $\geq 0.03013991922475 - 3.53 \times 10^{-15}$ \\ $\leq 0.14994363475076 + 5.09 \times 10^{-15}$   \end{tabular} \\
\hline
(m) & $\frac{315}{16}$ & $\frac{75}{2}$ & $\frac{495}{8}$ & $\frac{1095}{16}$ & $\frac{2505}{32}$ & $\frac{1305}{16}$ & $\frac{1185}{16}$ & $\frac{495}{8}$ & $\frac{195}{4}$ & $\frac{105}{4}$ & $\frac{135}{8}$ & $\frac{255}{8}$ & $\frac{225}{4}$ & $\frac{525}{8}$ & $\frac{615}{8}$ & $\frac{2565}{32}$ & $\frac{285}{4}$ & $\frac{225}{4}$ & $\frac{345}{8}$ & $\frac{375}{16}$ & \begin{tabular}{l} $\geq -0.16587845769802 - 3.03 \times 10^{-15}$ \\ $\leq -0.04621984365605 + 2.29 \times 10^{-15}$  \end{tabular} \\
\hline
(o) & 	 $\frac{435}{16}$ & $\frac{315}{8}$ & $\frac{975}{16}$ & $\frac{2295}{32}$ & $\frac{645}{8}$ & $\frac{4935}{64}$ & $\frac{2295}{32}$ & $\frac{975}{16}$ & $\frac{315}{8}$ & $\frac{645}{32}$ & $\frac{825}{32}$ & $\frac{585}{16}$ & $\frac{465}{8}$ & $\frac{1125}{16}$ & $\frac{5115}{64}$ & $\frac{2445}{32}$ & $\frac{1125}{16}$ & $\frac{465}{8}$ & $\frac{585}{16}$ & $\frac{75}{4}$ & \begin{tabular}{l} $\geq 0.02710311900070 - 3.65 \times 10^{-15}$ \\ $\leq 0.08854800480938 + 2.10 \times  10^{-15}$ \end{tabular} \\
\hline
\end{tabular}
}
  \caption{Bounding box of $S$ and upper and lower bounds of $\omega (x)$}
  \label{degree5_boundingbox}
\end{table}
\renewcommand{\arraystretch}{1}

\subsection*{Step 3.}
For Cases (c), (d), (e), and (j) in Table~\ref{table:remaining_cases}, the bounding box $B(S)$ is too large to approximate $S$, and it is therefore impossible to prove that $\omega$ is always positive or always negative on $S$ in Step~2.
In those cases, we partition the solution space $S$ into $S_1,\dots,S_{k+1}$ by some parallel hyperplanes $H_1,\dots,H_k$
and then consider the bounding boxes $B(S_1),\dots,B(S_{k+1})$.
Then, we obtain a finer approximation $B(S_1) \cup \dots \cup B(S_{k+1})$ of $S$. (See Figure~\ref{fig:bounding_box2} for an intuition.)
If $\omega$ always takes the same sign on each $B(S_i)$, the original system~(\ref{ineq_main}) is infeasible.
For Case (c), we considered a partition of $S$ induced by hyperplanes $\gamma_4 = 7$ and $\gamma_4 = 11$, and obtained the results presented in Table~\ref{table:boundingbox2c}. 
For Cases (d), (e), and (j),  the obtained results are listed in Tables~\ref{table:boundingbox2d}--\ref{table:boundingbox2j}, respectively.
\\
\\
Combining the results in Steps~1--3, we complete the proof of the only-if part.
\begin{figure}[htb]
\begin{minipage}[t]{0.45\linewidth}
\centering
\includegraphics[scale=0.4, bb =  0 0 285 262, clip]{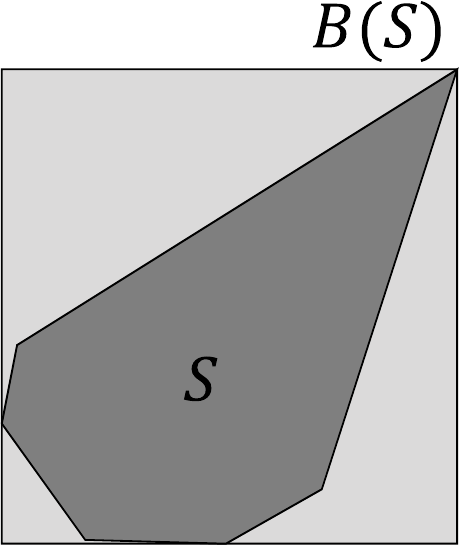} 
  \caption{Bounding box of $S$}      
\label{fig:bounding_box1}
\end{minipage}
\begin{minipage}[t]{0.45\linewidth}
\centering
\includegraphics[scale=0.4, bb =  0 0 222 264, clip]{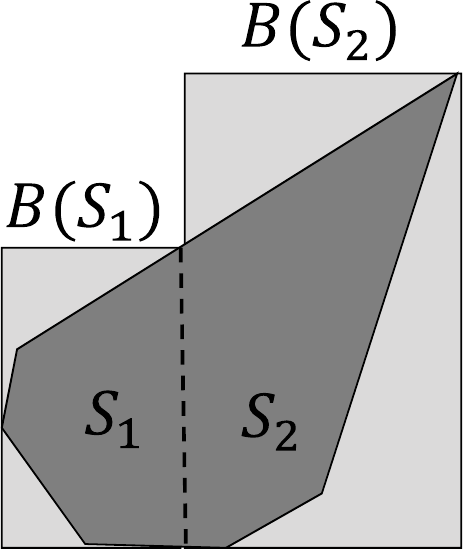} 
    \caption{Bounding boxes of $S_1$ and $S_2$} 
\label{fig:bounding_box2}
\end{minipage}
\end{figure}

\renewcommand{\arraystretch}{1.2}
\begin{table}[htb]
\centering
{\tabcolsep = 0.1mm
\scriptsize
  \begin{tabular}{|c|cccccccccc|cccccccccc|c|}
    \hline
 &   $\beta_0^+$ & $\beta_1^+$ & $\beta_2^+$ & $\beta_3^+$ & $\beta_4^+$ & $\gamma_0^+$ & $\gamma_1^+$ & $\gamma_2^+$ & $\gamma_3^+$ & $\gamma_4^+$ & $\beta_0^-$ & $\beta_1^-$ & $\beta_2^-$ & $\beta_3^-$ & $\beta_4^-$ & $\gamma_0^-$ & $\gamma_1^-$ & $\gamma_2^-$ & $\gamma_3^-$ & $\gamma_4^-$ &  $\omega (x)$ \\
\hline 
$S_1$ & $\frac{375}{8}$ & $\frac{1035}{16}$ & $\frac{1995}{32}$ & $\frac{735}{16}$ & $90$ & $\frac{9945}{128}$ & $\frac{2095}{32}$ & $\frac{3735}{56}$ & $\frac{1078}{15}$ & $7$ & $\frac{1575}{64}$ & $\frac{785}{16}$ & $\frac{1305}{28}$ & $\frac{ 544}{15}$ & $\frac{21503}{256}$ & $\frac{1065}{16}$ & $\frac{405}{8}$ & $\frac{885}{16}$ & $\frac{885}{16}$ & $0$ & \begin{tabular}{l} $ \geq 0.03993726639257 - 5.24 \times 10^{-15}$ \\ $\leq 0.41994431818384 + 2.82 \times 10^{-15}$ \end{tabular}\\
    \hline
  $S_2$ &   $\frac{319}{8}$ & $\frac{979}{16}$ & $\frac{1939}{32}$ & $\frac{623}{16}$ & $\frac{173}{2}$ & $\frac{9609}{128}$ & $\frac{6173}{96}$ & $\frac{3735}{56}$ & $\frac{1094}{15}$ & $11$ & $\frac{1911}{64}$ & $\frac{2467}{48}$ & $\frac{1305}{28}$ & $\frac{512}{15}$ & $\frac{21439}{256}$ & $\frac{1121}{16}$ & $\frac{461}{8}$ & $\frac{941}{16}$ & $\frac{941}{16}$ & $ 7$ & \begin{tabular}{l} $ \geq 0.01183838535761 - 5.48 \times 10^{-15}$ \\ $\leq 0.23617122360425 + 2.63 \times 10^{-15}$ \end{tabular}\\
    \hline
$S_3$ & $\frac{287}{8}$ & $\frac{227}{4}$ & $\frac{107}{2}$ & $\frac{559}{16}$ & $\frac{169}{2}$ & $\frac{ 9417}{128}$ & $\frac{6109}{96}$ & $\frac{3735}{56}$ & $\frac{2055}{28}$ & $\frac{1425}{112}$ & $\frac{2103}{64}$ & $\frac{2531}{48}$ & $\frac{1305}{28}$ & $\frac{465}{14}$ & $\frac{18735}{224}$ & $\frac{1153}{16}$ & $\frac{493}{8}$ & $\frac{253}{4}$ & $\frac{133}{2}$ & $11$ & \begin{tabular}{l} $ \geq 0.00422323812900 - 2.80 \times 10^{-15}$ \\ $\leq 0.09377125077374 + 5.77 \times 10^{-15}$ \end{tabular}\\
    \hline
\end{tabular}
}
  \caption{Bounding box of $S_i$ and upper and lower bounds of $\omega (x)$ (Case (c), $S_1 \coloneqq S \cap \{ \gamma_4 \leq 7 \}$, $S_2 \coloneqq S \cap \{ 7 \leq \gamma_4 \leq 11 \}$, $S_3 \coloneqq S \cap \{ \gamma_4 \geq 11 \}$)}
  \label{table:boundingbox2c}
\mbox{}\\
{\tabcolsep = 0.1mm
\scriptsize
  \begin{tabular}{|c|cccccccccc|cccccccccc|c|}
    \hline
 &   $\beta_0^+$ & $\beta_1^+$ & $\beta_2^+$ & $\beta_3^+$ & $\beta_4^+$ & $\gamma_0^+$ & $\gamma_1^+$ & $\gamma_2^+$ & $\gamma_3^+$ & $\gamma_4^+$ & $\beta_0^-$ & $\beta_1^-$ & $\beta_2^-$ & $\beta_3^-$ & $\beta_4^-$ & $\gamma_0^-$ & $\gamma_1^-$ & $\gamma_2^-$ & $\gamma_3^-$ & $\gamma_4^-$ &  $\omega (x)$ \\
\hline 
$S_1$ &$\frac{165}{4}$ &  $\frac{255}{4}$ & $\frac{495}{8}$ & $\frac{375}{8}$ & $90$ & $\frac{5025}{64}$ & $\frac{1035}{16}$ & $\frac{1845}{28}$ & $\frac{211}{3}$ & $5$ & $\frac{735}{32}$ & $\frac{405}{8}$ & $\frac{675}{14}$ & $\frac{118}{3}$ & $\frac{10895}{128}$ & $\frac{555}{8}$ & $\frac{105}{2}$ & $\frac{225}{4}$ & $\frac{225}{4}$ & $0$ & \begin{tabular}{l} $ \geq 0.075681007375979 - 8.99 \times 10^{-16}$ \\ $\leq 0.38064127480888 + 2.60 \times 10^{-15}$ \end{tabular}\\
\hline
$S_2$ & $\frac{145}{4}$ & $\frac{245}{4}$ & $\frac{485}{8}$ & $\frac{335}{8}$ & $\frac{175}{2}$ & $\frac{4905}{64}$ & $\frac{3065}{48}$ & $\frac{1845}{28}$ & $\frac{1005}{14}$ & $\frac{585}{56}$ & $\frac{855}{32}$ & $\frac{1255}{24}$ & $\frac{675}{14}$ & $\frac{255}{7}$ & $\frac{9495}{112}$ & $\frac{575}{8}$ & $\frac{115}{2}$ & $\frac{235}{4}$ & $\frac{235}{4}$ & $5$ & \begin{tabular}{l} $ \geq 0.01953883026443 - 1.70 \times 10^{-15}$ \\ $\leq 0.25021013905337 + 4.07 \times 10^{-15}$ \end{tabular}\\
\hline
\end{tabular}
}
  \caption{Bounding box of $S_i$ and upper and lower bounds of $\omega (x)$ (Case (d), $S_1 \coloneqq S \cap \{ \gamma_4 \leq 5 \}$, $S_2 \coloneqq S \cap \{ \gamma_4 \geq 5 \}$)}
  \label{table:boundingbox2d}
\end{table}

\begin{table}
{\tabcolsep = 0.1mm
\scriptsize
  \begin{tabular}{|c|cccccccccc|cccccccccc|c|}
    \hline
 &   $\beta_0^+$ & $\beta_1^+$ & $\beta_2^+$ & $\beta_3^+$ & $\beta_4^+$ & $\gamma_0^+$ & $\gamma_1^+$ & $\gamma_2^+$ & $\gamma_3^+$ & $\gamma_4^+$ & $\beta_0^-$ & $\beta_1^-$ & $\beta_2^-$ & $\beta_3^-$ & $\beta_4^-$ & $\gamma_0^-$ & $\gamma_1^-$ & $\gamma_2^-$ & $\gamma_3^-$ & $\gamma_4^-$ &  $\omega (x)$ \\
\hline 
$S_1$ & $\frac{165}{4}$ & $\frac{1035}{16}$ & $\frac{1995}{32}$ & $\frac{195}{4}$ & $90$ & $\frac{2565}{32}$ & $65$ & $\frac{1845}{28}$ & $\frac{2111}{30}$ & $7$ & $\frac{315}{16}$ & $50$ & $\frac{675}{14}$ & $\frac{589}{15}$ & $\frac{5387}{64}$ & $\frac{555}{8}$ & $\frac{405}{8}$ & $\frac{885}{16}$ & $\frac{885}{16}$ & $0$ & \begin{tabular}{l} $ \geq 0.027606528320338 - 6.58 \times 10^{-16}$ \\ $\leq 0.39692841575826 + 2.70 \times 10^{-15}$ \end{tabular} \\
\hline
$S_2$ & $\frac{137}{4}$ & $\frac{979}{16}$ & $\frac{1939}{32}$ & $\frac{167}{4}$ & $\frac{173}{2}$ & $\frac{2481}{32}$ & $\frac{383}{6}$ & $\frac{1845}{28}$ & $\frac{427}{6}$ & $10$ & $\frac{399}{16}$ & $\frac{157}{3}$ & $\frac{675}{14}$ & $\frac{113}{3}$ & $\frac{5375}{64}$ & $\frac{583}{8}$ & $\frac{461}{8}$ & $\frac{941}{16}$ & $\frac{941}{16}$ & $7$ & \begin{tabular}{l} $ \geq 0.01976589000208 - 1.08 \times 10^{-15}$ \\ $\leq 0.21354032967021 + 5.02 \times 10^{-15}$ \end{tabular} \\
\hline
$S_3$ & $\frac{125}{4}$ & $\frac{235}{4}$ & $\frac{115}{2}$ & $\frac{155}{4}$ & $85$ & $\frac{2445}{32}$ & $\frac{190}{3}$ & $\frac{1845}{28}$ & $\frac{1005}{14}$ & $\frac{345}{28}$ & $\frac{435}{16}$ & $\frac{160}{3}$ & $\frac{675}{14}$ & $\frac{255}{7}$ & $\frac{4695}{56}$ & $\frac{595}{8}$ & $\frac{485}{8}$ & $\frac{245}{4}$ & $\frac{125}{2}$ & $10$ & \begin{tabular}{l} $ \geq 0.00062581827947 - 2.19 \times 10^{-15}$ \\ $\leq 0.11990580211028 + 2.31 \times 10^{-15}$ \end{tabular} \\
\hline
\end{tabular}
}
  \caption{Bounding box of $S_i$ and upper and lower bounds of $\omega (x)$ (Case (e), $S_1 \coloneqq S \cap \{ \gamma_4 \leq 7 \}$, $S_2 \coloneqq S \cap \{ 7 \leq \gamma_4 \leq 10 \}$, $S_3 \coloneqq S \cap \{ \gamma_4 \geq 10 \}$)}
  \label{table:boundingbox2e}
\end{table}

\begin{table}
{\tabcolsep = 0.1mm
\scriptsize
  \begin{tabular}{|c|cccccccccc|cccccccccc|c|}
    \hline
 &   $\beta_0^+$ & $\beta_1^+$ & $\beta_2^+$ & $\beta_3^+$ & $\beta_4^+$ & $\gamma_0^+$ & $\gamma_1^+$ & $\gamma_2^+$ & $\gamma_3^+$ & $\gamma_4^+$ & $\beta_0^-$ & $\beta_1^-$ & $\beta_2^-$ & $\beta_3^-$ & $\beta_4^-$ & $\gamma_0^-$ & $\gamma_1^-$ & $\gamma_2^-$ & $\gamma_3^-$ & $\gamma_4^-$ &  $\omega (x)$ \\
\hline 
  $S_1$ &  $\frac{135}{4}$ & $\frac{975}{16}$ &  $\frac{2175}{32}$ & $60$ & $90$ &  $\frac{165}{2}$ & $\frac{565}{8}$ & $\frac{1755}{28}$ & $\frac{1931}{30}$ & $7$ & $15$ & $\frac{155}{4}$ & $\frac{765}{14}$ & $\frac{769}{15}$ & $\frac{679}{8}$ & $\frac{585}{8}$ & $\frac{465}{8}$ &
$\frac{705}{16}$ & $\frac{825}{16}$ & $0$
& \begin{tabular}{l} $ \geq 0.01134976797036 - 1.92 \times 10^{-15}$ \\ $\leq 0.38984862337978 + 5.44 \times 10^{-15}$ \end{tabular}\\
    \hline
  $S_2$ &  $\frac{107}{4}$ & $\frac{199}{4}$ & $\frac{499}{8}$ & $53$ & $\frac{173}{2}$ & $\frac{639}{8}$ &  $\frac{1667}{24}$ & $\frac{1755}{28}$ &  $\frac{915}{14}$ & $\frac{75}{7}$ & $\frac{81}{4}$ & $\frac{493}{12}$ & $\frac{765}{14}$ &  $\frac{345}{7}$ &  $\frac{1185}{14}$ & $\frac{613}{8}$ & $\frac{521}{8}$ & $\frac{221}{4}$ & $\frac{461}{8}$ &  $7$ &
 \begin{tabular}{l} $ \geq 0.00159668015617 - 2.56 \times 10^{-15}$ \\ $\leq 0.167983216625798 + 9.33 \times 10^{-16}$ \end{tabular}\\
    \hline
\end{tabular}
}
  \caption{Bounding box of $S_i$ and upper and lower bounds of $\omega (x)$ (Case (j), $S_1 \coloneqq S \cap \{ \gamma_4 \leq 7 \}$, $S_2 \coloneqq S \cap \{ \gamma_4 \geq 7 \}$)}
  \label{table:boundingbox2j}
\end{table}
\renewcommand{\arraystretch}{1}

\end{proof}

\newpage

\section{Greedy-drawable pseudo-trees}
In this section, we characterize greedy-drawable \emph{pseudo-trees}, i.e., graphs obtained by adding one edge to a tree.
Let ${\cal T}$ be a pseudo-tree.
Then, ${\cal T}$ contains a single cycle. Let ${\cal C} =v_0,\dots,v_{m-1},v_0$ be the cycle in ${\cal T}$.
If we remove the edges of ${\cal C}$ from ${\cal T}$, we obtain $m$ trees.
Let $T_i$ be the tree that contains $v_i$.
We consider $\widetilde{T}_i \coloneqq T_i + v_i\widetilde{v}_i$, where $\widetilde{v}_i$ is a new vertex.
Then, we consider the new pseudo-tree $\widetilde{{\cal T}}$, called \emph{auxiliary pseudo-tree} of ${\cal T}$, obtained  by 
joining each $\widetilde{T_i}$ to the cycle $\widetilde{C}=\widetilde{v}_0,\dots,\widetilde{v}_{m-1},\widetilde{v}_0$.
We first observe the following fact.
\begin{lem}
If a pseudo-tree ${\cal T}$ has a (planar) greedy drawing, then the drawing must be outerplanar. 
\end{lem}
\begin{proof}
Suppose that there is an edge $uv$ that is drawn inside the region bounded by the cycle.
Then, the entire cycle must be contained on the same side of $\mathrm{axis}(uv)$ by Proposition~\ref{prop:halfspace}.
However, since $\mathrm{axis}(uv)$ passes through the interior of the region, it is impossible.
\end{proof}
\\
\\
Thus, if ${\cal T}$ has a greedy drawing, then $\widetilde{{\cal T}}$ also has a greedy drawing.
A greedy drawing of   $\widetilde{{\cal T}}$ is constructed by drawing each edge $v_i\widetilde{v}_i$ infinitesimally small. 
Thus, we can always interpret a greedy drawing of ${\cal T}$ as a greedy drawing  of $\widetilde{{\cal T}}$ in which each edge $v_i\widetilde{v}_i$ is drawn infinitesimally small.
With this interpretation, we define $\mathrm{polygon}(T_i) \coloneqq \mathrm{polygon}(\widetilde{T}_i)$ and $\angle{T_i} \coloneqq \angle{\widetilde{T}_i}$.
Then, we can apply an argument similar to the proof of \cite[Lemma~2.16]{NP17}, and obtain the following lemma.
\begin{lem}
\label{lem:angle_sum_pt}
In any (planar) greedy drawing of the pseudo-tree ${\cal T}$, the following inequality holds:
\begin{equation} 
\sum_{i=0,\dots,m-1,|\angle{T_i}| > 0}{|\angle{T_i}|} > 180^\circ (m-2).
\label{ineq:pt}
\end{equation}
\end{lem}
As we will see in Theorem~\ref{thm:pseudo_trees}, the above inequality also works as a sufficient condition for greedy-drawability of pseudo-trees ${\cal T}$ having at least four subtrees $T_i$ with $|\angle{T_i}| \neq 180^\circ$.
To deal with the other case, we make some preparations.
Suppose that there are at most three subtrees $T_i$ with $|\angle{T_i}|_* \neq 180^\circ$ in ${\cal T}$.
Select three vertices $v_j,v_k,v_l$, called \emph{connection vertices}, on the cycle ${\cal C}$ in such a way that $T_j,T_k,T_l$ include all such subtrees.
We define a \emph{$Y$-transformed tree} of ${\cal T}$ to be the tree obtained from $T_j,T_k,T_l$ by connecting each of the vertices $v_j,v_k,v_l$ to a new vertex $w$, called the \emph{central vertex}  (See Figure~\ref{fig:y_transformed_tree}).

Now suppose  $m=3$, and $\widetilde{{\cal T}}$ is greedy-drawable.
Then $\widetilde{{\cal T}}$ has a greedy drawing in which each subtree $T_i$ ($=\widetilde{T}_i \setminus \widetilde{v}_i$) with $|\angle{T_i}|_* > 0$ is infinitesimally small. 
It can be proved by almost the same argument as the proof of  Lemma~2.18 in \cite{NP17}.
We give a brief overview here.
Since replacing a subpath of a greedy path with a line segment keeps greediness (Lemma~2.5 in \cite{NP17}), there is a greedy path between every vertex in $\widetilde{T}_i$ and every vertex in $\widetilde{T_j}$ ($i \neq j$) using the edge $\widetilde{v}_i\widetilde{v}_j$.
Therefore, $\widetilde{T}_i$ is contained in the half-plane $h_{\widetilde{v}_i\widetilde{v}_j}^{\widetilde{v}_i}$ and the half-plane $h_{uv}^u$ for every edge $uv$ in $\widetilde{T}_j$ ($i \neq j$) 
with $d_{\widetilde{{\cal T}}}(\widetilde{v_j}, u) < d_{\widetilde{{\cal T}}}(\widetilde{v_j}, v) $, where $d_{\widetilde{{\cal T}}}$ is the graph distance on $\widetilde{{\cal T}}$.
Then, applying the same argument as the proof of  Lemma~2.14~(a) in~\cite{NP17}, we can show that the apex of $\angle{T_i}$ is also contained in all of those half-planes.
Thus, if we draw $T_i$ infinitesimally small at the apex of $\angle{T_i}$ in such a way that $\angle{T_i}$ contains the original angle, the constructed drawing is a greedy drawing.
\begin{figure}[h]
\centering
\includegraphics[scale=0.3, bb =  0 0 776 366, clip]{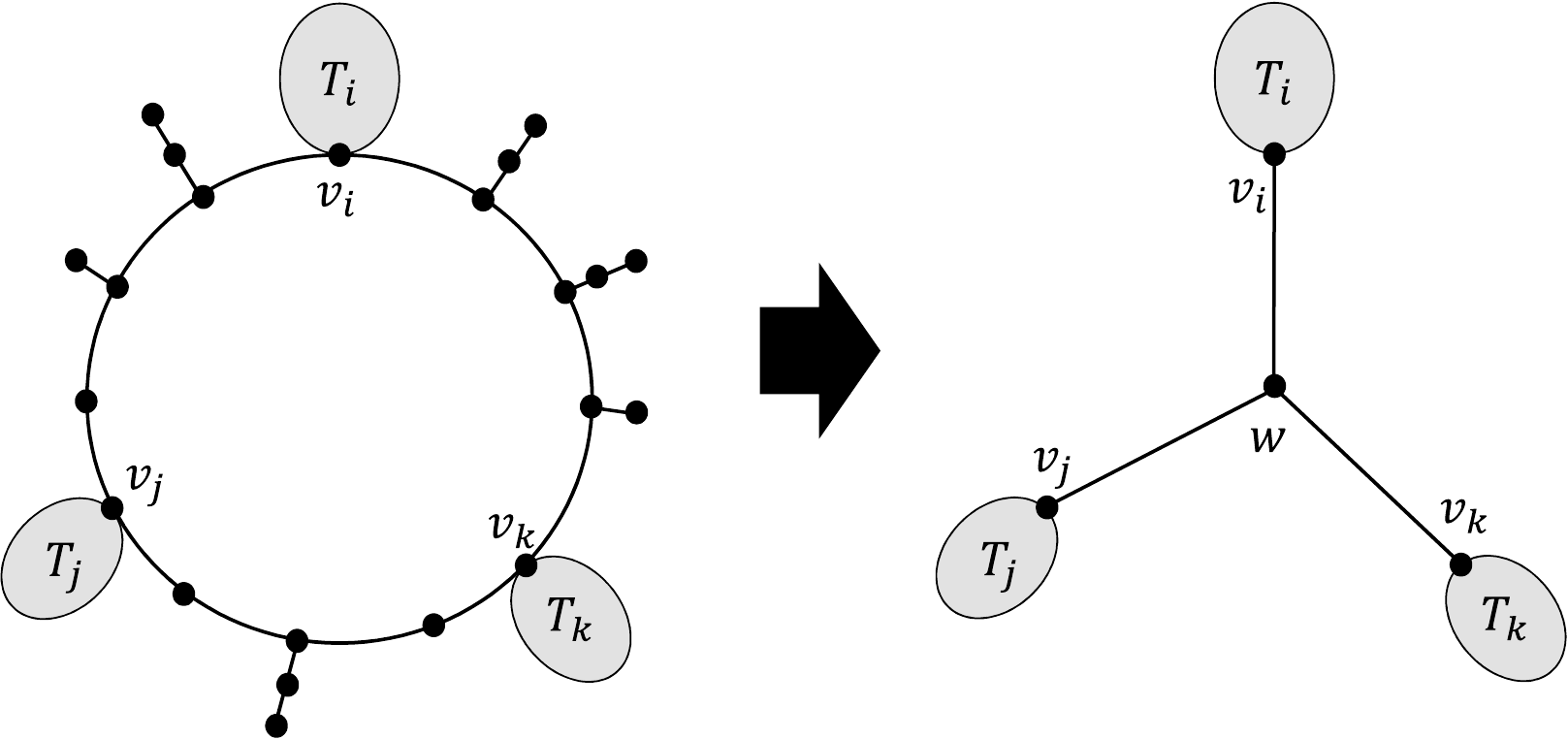} 
    \caption{Constructing a $Y$-transformed tree} 
\label{fig:y_transformed_tree}
\end{figure}

Now we are ready to present a characterization of greedy-drawable pseudo-trees.
\begin{thm}
Let ${\cal T}$ be a pseudo-tree with a cycle ${\cal C}=v_0,\dots,v_m,v_0$.
For $i=0,\dots,m$, let $T_i$ be the connected component of ${\cal T} \setminus {\cal C}$, which is a tree, that contains $v_i$.
Let $\varphi_0 \coloneqq |\angle{T_0}|_*,\dots,\varphi_{m-1} \coloneqq |\angle{T_{m-1}}|_*$, where we set $\varphi_i = 0$ if $|\angle{T_i}|_* \leq 0$.
Then, the pseudo-tree ${\cal T}$ is greedy-drawable if and only if one of the following conditions holds:
\begin{itemize}
\item[(a)] There are at least four angles $\varphi_i$ with  $\varphi_i \neq 180^\circ$, and the inequality $\sum_{i=0}^{m-1}{\varphi_i} > 180^\circ (m-2)$ holds.
\item[(b)] There are at most three angles $\varphi_i$ with $\varphi_i \neq 180^\circ$, and a $Y$-transformed tree of ${\cal T}$ is greedy-drawable. 
\end{itemize}
\label{thm:pseudo_trees}
\end{thm}
\begin{proof}
Let $k$ be the number of angles $\varphi_i$ with $\varphi_i \neq 180^\circ$.
We first prove the only-if part.
Suppose ${\cal T}$ is greedy-drawable.
Then, we have $\sum_{i=0}^{m-1}{\varphi_i} > 180^\circ (m-2)$ by Lemma~\ref{lem:angle_sum_pt}.
If $k \geq 4$, the condition~(a) is already fulfilled, and thus we consider the case that $k \leq  3$. 
Let $T$ be a $Y$-transformed tree of ${\cal T}$, and $v_a,v_b,v_c$ the connection vertices of $T$.
We prove that $T$ is greedy-drawable.
Since $\varphi_i = 180^\circ$ for all $i \notin \{ a,b,c \}$, the above inequality implies $\varphi_a + \varphi_b + \varphi_c > 180^\circ$.
If $\varphi_a,\varphi_b, \varphi_c > 0$, then, $T$ is greedy-drawable by Lemma~4.5 in \cite{NP17}.
Thus, we suppose one of $\varphi_i$, for $i \in \{ a,b,c \}$, satisfies $\varphi_i = 0$.
Without loss of generality, we assume that $\varphi_c = 0$.
Then, in any greedy drawing of ${\cal T}$, we have $|\angle{T_c}| \leq 0$, which also implies $|\angle{T_a}| > 0, |\angle{T_b}| > 0$ by Lemma~\ref{lem:angle_sum_pt}.
Now we consider the auxiliary pseudo-tree $\widetilde{{\cal T}}$ and
the graph  $\widetilde{{\cal T}}'$ obtained from the subtrees $\widetilde{T}_a,\widetilde{T}_b,\widetilde{T}_c$ by connecting each pair of the vertices $\widetilde{v}_a,\widetilde{v}_b,\widetilde{v}_c$.
To construct a greedy drawing of $T$, we first construct a greedy drawing of $\widetilde{{\cal T}}'$,  and then transform it into a greedy drawing of $T$.

Take a greedy drawing of ${\widetilde{\cal T}}$.
Slightly perturbating the drawing if necessary, we assume that $|\angle{T_c}| < 0$.
We transform the drawing into a drawing  $\Gamma_{\widetilde{\cal T}}$ in which $T_a$ ($=\widetilde{T}_a \setminus v_a$) and $T_b$ ($= \widetilde{T}_b \setminus v_b$) are drawn infinitesimally small.
Let $\Gamma_{\widetilde{{\cal T}}'}$ be the drawing of $\widetilde{{\cal T}}'$ obtained by restricting the drawing $\Gamma_{\widetilde{{\cal T}}}$ into $\widetilde{T}_a,\widetilde{T}_b,\widetilde{T_c}$,
and connecting each pair of $v_a, v_b, v_c$ by a line segment.
Since replacing a subpath of a greedy path with a line segment keeps greediness (Lemma~2.5 in \cite{NP17}), $\Gamma_{\widetilde{{\cal T}}'}$ is a greedy drawing of $\widetilde{{\cal T}}'$.

Now we transform the drawing $\Gamma_{\widetilde{{\cal T}}'}$ into a greedy drawing of $T$.
In $\Gamma_{\widetilde{{\cal T}}'}$, the tree $T_c$ is already drawn greedily, and we have to draw $T_{ab} \coloneqq T \setminus T_c$ ($=  v_cw + wv_a + wv_b + T_a + T_b$) appropriately.
We first observe that the tree $T_{ab}$ can be drawn greedily.
Indeed, since $\varphi_a+\varphi_b > 180^\circ$ and $\varphi_a,\varphi_b \leq 120^\circ$, we have $\varphi_a,\varphi_b > 60^\circ$.
By the result in Table~\ref{table:open_angle}, this implies $\varphi_a,\varphi_b > 90^\circ$.
By Lemma~\ref{lem:combination}~(III), the tree $T_{ab}$ (with root $v_c$) has a greedy drawing with an open angle, and the supremum of opening angles is $\varphi_a + \varphi_b - 180^\circ$. 
Then, our task is to place an infinitesimally small drawing of $T_{ab}\setminus v_c$ at an appropriate place in $\mathrm{polygon}(T_c)$ in such a way that (1) $\angle{T_{ab}}$ contains $T_c$, (2) $h_{v_cw}^{v_c}$ contains $T_{c}$, and (3) $h_{v_cw}^w$ contains $T_{ab}$.
Intuitively, these conditions can be satisfied if we draw $T_{ab} \setminus v_c$ inside of  $\mathrm{polygon}(T_c)$ sufficiently far from $T_c$ in such a way that $\angle{T_{ab}}$ is sufficiently large (see Figure~\ref{fig:ps_to_t}).
Here, $\mathrm{polygon}(T_c)$ is bounded, and the interior of $\mathrm{polygon}(T_c)$ is partitioned into some regions by the edges of $T_c$.
We consider a vertex $p$ of $\mathrm{polygon}(T_{c})$ that lies on the boundary of the region containing $\widetilde{{\cal T}}' \setminus T_c$.
We show that if we place the vertex $w$ and an infinitesimally small drawing of $(T_{ab})_{v_cw}^w$ inside of $\mathrm{polygon}(T_c)$ sufficiently close to $p$, the conditions~(1)--(3) are satisfied.

Let $l_1$ and $l_2$ be the supporting lines of the two edges of $\mathrm{polygon}(T_{c})$ incident to $p$. See Figure~\ref{fig:ps_to_t} for an illustration.
Let  $u,u',v,v'$ ($d_{T_c}(v_c,u) < d_{T_c}(v_c,v), d_{T_c}(v_c,u') < d_{T_c}(v_c,v')$) be the vertices of $T_c$ such that $l_1=\mathrm{axis}(uv)$ and $l_2=\mathrm{axis}(u'v')$, where $d_{T_c}$ is the graph distance on $T_c$,
and  let $l_1^+\coloneqq h_{uv}^u$, $l_2^+ \coloneqq h_{u'v'}^{u'}$.
Since the apex of $\angle{T_a}$ lies in $l_1^+ \cap l_2^+$ (recall that $T_a$ is drawn infinitesimally small in $\Gamma_{\widetilde{T}'}$), and $\angle{T_a}$ contains $v \notin l_1^+ $ and $v' \notin l_2^+$, 
one of the rays of $\angle{T_a}$ intersects with $l_1$, and is pointing outward $l_1^+$, and 
the other ray intersects with $l_2$, and is pointing outward $l_2^+$.
Now we remark that $\mathrm{axis}(\widetilde{v}_cv_c)$ separates $v_a,v_b, \widetilde{v}_a, \widetilde{v}_b, \widetilde{v}_c$ from $T_c$ by greediness of $\widetilde{{\cal T}}'$ (recall Proposition~\ref{prop:halfspace}).
Since $\mathrm{axis}(\widetilde{v}_cv_c)$ does not cross any edge of $T_c$ (by Proposition~\ref{prop:halfspace}), and
$v_a,v_b, \widetilde{v}_a, \widetilde{v}_b, \widetilde{v}_c, p$ belong to the same connected component of $\mathrm{polygon}(T_c) \setminus T_c$, the apex $p$ is also separated from  $T_c$, 
and thus from $v$ and $v'$ by $\mathrm{axis}(\widetilde{v}_cv_c)$.
Therefore, both of  the two unbounded edges of $\angle{T_a} \cap \angle{T_b}$ must intersect with $\mathrm{axis}(\widetilde{v}_cv_c)$. 
One of the unbounded edges is pointed outward $l_1^+$, the other is pointed outward $l_2^+$. 
Let $x$ be the intersection of the lines obtained by extending the two unbounded edges of $\angle{T_a} \cap \angle{T_b}$.
Then, $x$ belong to $\mathrm{polygon}(T_c)$.
Let $R$ be the unbounded region bounded by the rays obtained by extending  the two unbounded edges to $x$.
Then, a simple calculation shows that the angle between the two extreme rays of $R$ is less than $|\angle{T_a}|+|\angle{T_b}|-180^\circ$, and thus less than $\varphi_a+\varphi_b-180^\circ$.
Therefore, we can construct a drawing of $T_{ab}$ in which  $\angle{T_{ab}}$ contains $R$ by placing $T_{ab} \setminus v_c$ sufficiently close to $p$ (see Figure~\ref{fig:ps_to_t}).
Since $T_c$ is contained in $R$, the constructed drawing satisfies the condition~(1).

Now we verify the condition~(2). See Figure~\ref{fig:ps_to_t2}.
We first construct a finer description of a region (finer than $R$) in which $T_c$ must lie.
Let $v'_c$ be a vertex of $T_c$ adjacent to $v_c$.
Then, $v'_c$ must lie outside of the circle $S(p)$  with radius $|pv_c|$ centered at $p$.
Otherwise, $h_{v_cv'_c}^{v_c}$ does not contain $p$, which contradicts to the assumption that $p$ is a vertex of  $\mathrm{polygon}(T_c)$.
Next, consider a vertex $v''_c$ ($\neq v_c$) of $T_c$ that is adjacent to $v'_c$.
Then, $v'_c$ must lie outside of the circle with radius $|pv'_c|$ centered at $p$.
Since $|pv_c| < |pv'_c|$, this implies that $v'_c$ must also lie outside of $S(p)$.
Continuing this discussion, we can show that  $T_c$ must lie outside of $S(p)$.
This implies that the region $\widetilde{R} \coloneqq R \setminus S(p)$ contains $T_c$.
We prove that $\widetilde{R}$ is contained in $h_{pv_c}^{v_c}$.
Let $q$ and $r$ be the intersections of $S(p)$ with the two extreme rays of $R$.
Let $M$ be the midpoint of the line segment $pv_c$, and let $M_q$ and $M_r$ be the intersections of  $\mathrm{axis}(pv_c)$ with the line segments $pr$ and $pq$ respectively.
Since the angle between the two extreme rays of $R$ is less than $\varphi_a+\varphi_b-180^\circ$, and thus less than $60^\circ$, we have $\angle{qpr} < 60^\circ$.
This leads to that $|pM_q| < 2|pM| =|pq|$, and $|pM_r| < 2|pM| =|pr|$, which implies that $q,r \in h_{pv_c}^{v_c}$.
Therefore, $\widetilde{R}$ is contained in $h_{pv_c}^{v_c}$.
Since $w$ is placed sufficiently close to $p$, it implies $T_c$ is contained in $h_{wv_c}^{v_c}$.
Since the condition~(3) is clearly satisfied, the constructed drawing is a greedy drawing of $T$.
Therefore, the if part is proved for all cases.

\begin{figure}[h]
\centering
\includegraphics[scale=0.5, bb =  0 0 816 390, clip]{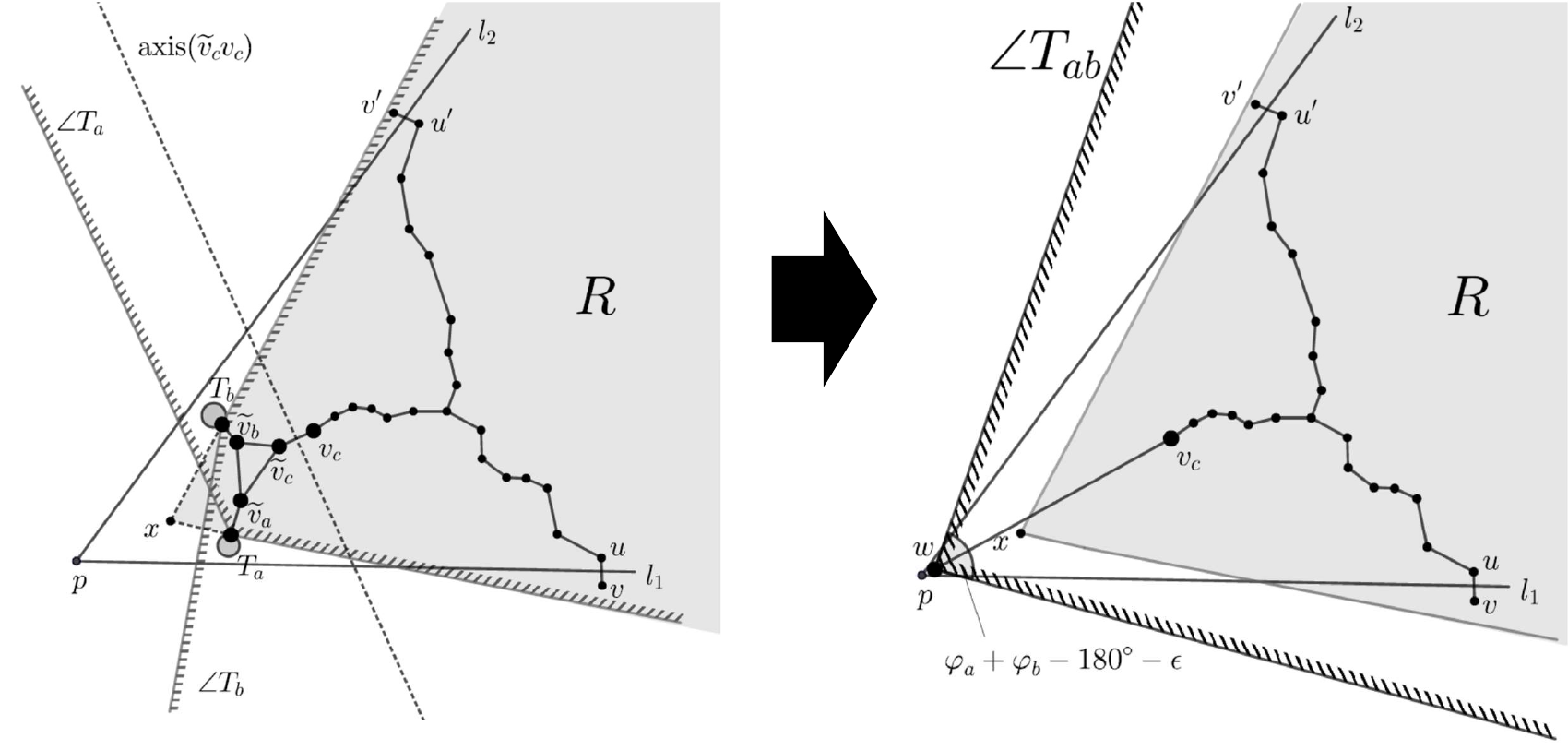} 
    \caption{Transforming a greedy drawing of $\widetilde{{\cal T}}'$ into a greedy drawing of $T$} 
\label{fig:ps_to_t}
\end{figure}
\begin{figure}[h]
\centering
\includegraphics[scale=0.25, bb =   0 0 750 949, clip]{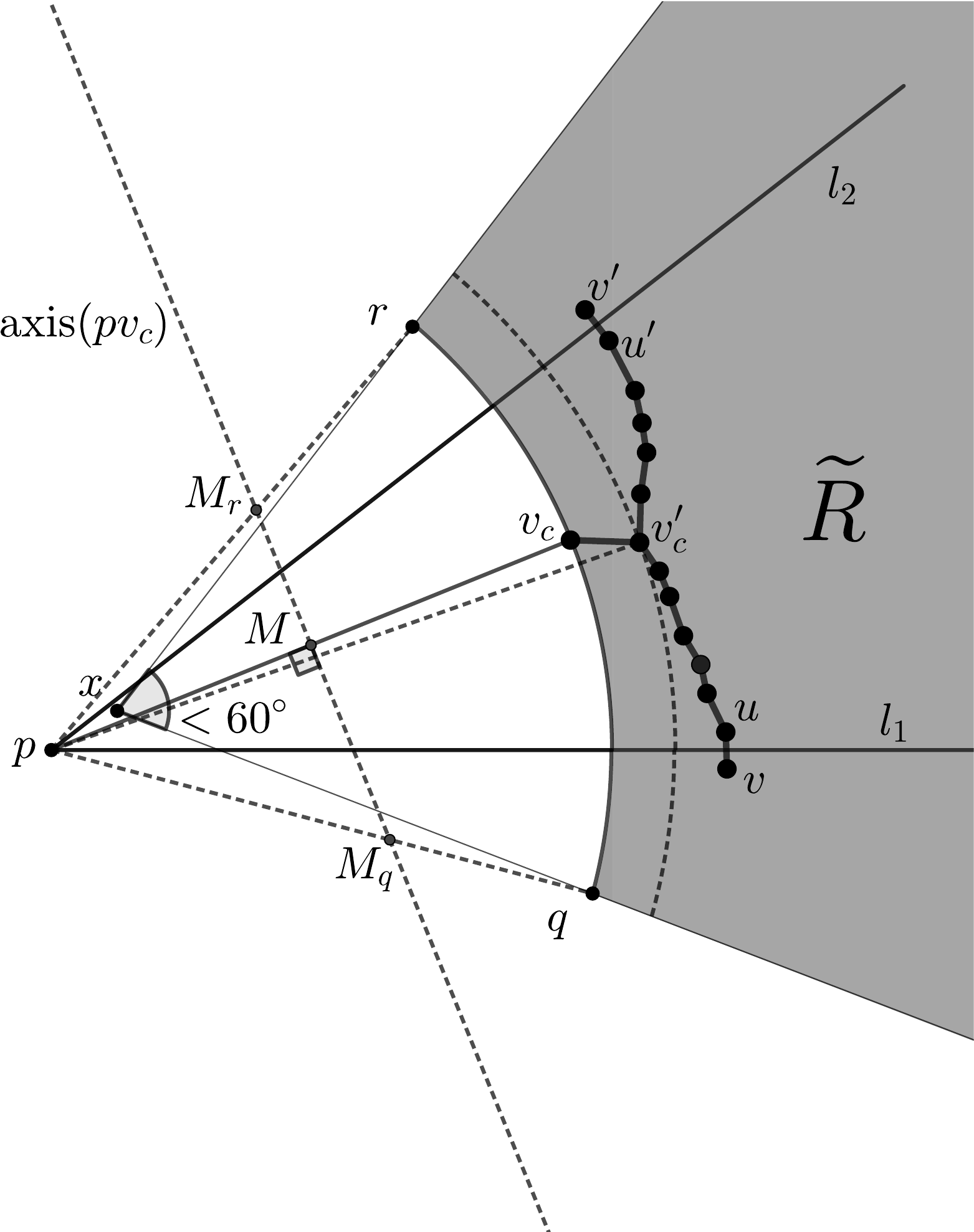} 
    \caption{Proving that $T_{uv}^{u}$ is contained in $h_{v_cp}^{v_c}$} 
\label{fig:ps_to_t2}
\end{figure}

Now we prove the if-part.
Assume that ${\cal T}$ satisfies the condition~(a).
Then, at most five angles $\varphi_i$ can satisfy $\varphi_i \neq 180^\circ$, i.e., $\varphi_i \leq 120^\circ$.
Otherwise, $\sum_{i=0}^{m-1}{\varphi_i} \leq 180^\circ (m-6) + 120^\circ \times 6 = 180^\circ (m-2)$, which contradicts to the assumption.
Let $\psi_0 \coloneqq \varphi_{i_0},\dots, \psi_{k-1} \coloneqq \varphi_{i_{k-1}}$ ($i_0 < \cdots < i_{k-1}$) be the angles with $\varphi_{i_j} \neq 180^\circ$, i.e., $\varphi_{i_j} \leq 120^\circ$.
Let us denote by ${\cal V}_j$, for $j=0,\dots,k-1$, the set of inner vertices of the $v_{i_j}v_{i_{j+1}}$-path in ${\cal C}$ that contains the edge $v_{i_j}v_{i_j+1}$, where the indices are interpreted modulo $m+1$. 
Based on the above observation, we consider the following three cases.

\begin{description}[itemsep=0cm,leftmargin=0cm,topsep=1ex]
\item[(Case 1)] ${\cal T}$ satisfies the condition~(a), and there are exactly five angles $\varphi_i$ satisfying $\varphi_i \leq 120^\circ$.

In this case, no angle  $\varphi_i$ satisfies $\varphi_i \leq 90^\circ$.
Suppose otherwise. Then, some angle $\varphi_j$ satisfies $\varphi_j \leq 90^\circ$, which implies $\varphi_j \leq 60^\circ$ by the result in Table~\ref{table:open_angle}.
This leads to that $\sum_{i=0}^{m-1}{\varphi_i} \leq 180^\circ (m-5) + 120^\circ \times 4 + 60^\circ = 180^\circ (m-2)$, which is a contradiction.
Let us consider a pentagon $\Gamma_1$ with vertices $W_0,\dots,W_4$, labeled counterclockwise,  satisfying
\begin{itemize}
\item $90^\circ <  \angle{W_j} < \psi_{j}$ for $j=0,\dots,4$, .
\item The bisector $\mathrm{axis}(W_jW_{j+1})$ does not pass through $W_{j+3}$, for $j=0,\dots,4$.
\end{itemize}
Here, the indices of $W_j$ are interpreted modulo $5$.
Since $\sum_{j=0}^{4}{\psi_j} > 540^\circ$, it is easy to construct a pentagon satisfying the first condition.
Slightly perturbating this pentagon, we can construct a pentagon $\Gamma_1$ satisfying both of the two conditions.
Then, we place each vertex $v_{i_j}$ at $W_j$, and
the vertices in ${\cal V}_j$ on the line segment $W_jW_{j+1}$ in the following way.
We first note that $\angle{W_j} > 90^\circ$ and $\angle{W_{j+1}} > 90^\circ$, and thus the bisector $\mathrm{axis}(W_jW_{j+1})$ passes through either the interior of the line segment $W_{j-1}W_{j-2}$ or that of $W_{j+2}W_{j+3}$.
In the former case, we place the vertices in ${\cal V}_j$ on $W_jW_{j+1}$ sufficiently close to $W_j$.
In the latter case, we place those vertices on $W_jW_{j+1}$ sufficiently close to $W_{j+1}$. See Figure~\ref{fig:pentagon1}.
For $k=0,\dots,4$, we call the set of vertices placed sufficiently close to (or placed at) $W_k$ the \emph{$\epsilon$-neighborhood} of $W_{k}$.

Since a greedy path clearly exists between every pair of vertices in ${\cal V}_j$,
we prove that there is a greedy $wv$-path for every vertex $w \in {\cal V}_j$ and every vertex $v \notin {\cal V}_j$.  
Let $l_j \coloneqq \mathrm{axis}(W_jW_{j+1})$, and $l_j^k$ the perpendicular line of $W_jW_{j+1}$ at $W_k$ for $k = j,j+1$. 
Suppose that  $w$ is in the $\epsilon$-neighborhood of $W_j$.
Then, the counterclockwise path from $w$  to a vertex in the $\epsilon$-neighborhood of $W_{j+1},W_{j+2},W_{j+3}$ is greedy.
Indeed, the perpendicular bisector of each edge on the path is sufficiently close to one of $l_j, l_{j+1},l_{j+2}, l_j^j,l_{j+1}^{j+1},l_{j+1}^{j+2},l_{j+2}^{j+2},l_{j+2}^{j+3},l_{j+3}^{j+3}$, and one can easily check that the half-plane defined for each edge on the path
contains $v$ (recall Proposition~\ref{prop:halfspace}).
On the other hand, the clockwise path from $w$ to a vertex in the $\epsilon$-neighborhood of $W_{j-1}$ is greedy. 
Next, suppose that  $w$ is in the $\epsilon$-neighborhood of $W_{j+1}$.
Then, the counterclockwise path from $w$  to a vertex in the $\epsilon$-neighborhood of $W_{j+2}$ is greedy, and the clockwise path from $w$ to a vertex in the $\epsilon$-neighborhood of $W_{j},W_{j-1},W_{j-2}$ is greedy.
Therefore, the constructed drawing of ${\cal C}$, denoted hereinafter by $\Gamma_2$, is greedy. 

Now we replace each line segment $W_jW_{j+1}$ by a sufficiently flat convex polygonal chain (see Figure~\ref{fig:polygonal_chain}).  
Then, the perpendicular bisector of each edge on the polygonal chain is sufficiently close to one of $l_j, l_j^j, l_j^{j+1}$. 
If we draw each $T_i$ infinitesimally small in such a way that each $\angle{T_i}$ contains $\Gamma_2$, then the constructed drawing is a greedy drawing of ${\cal T}$.

\begin{figure}[h]
\begin{minipage}[t]{0.5\linewidth}
\centering
\includegraphics[scale=0.2, bb = 0 0 846 949, clip]{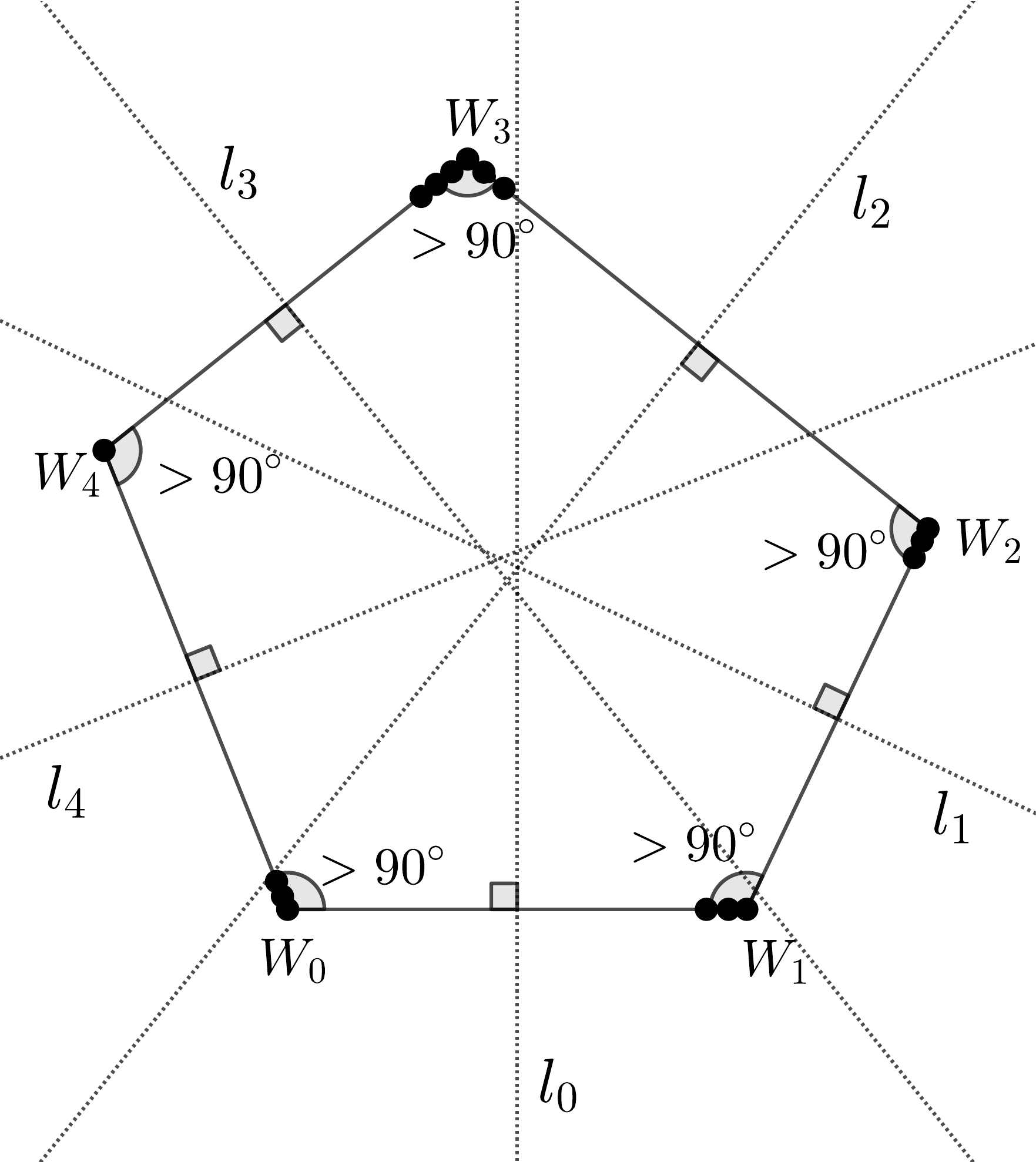} 
    \caption{Pentagon $\Gamma_1$ and drawing $\Gamma_2$} 
\label{fig:pentagon1}
\end{minipage}
\begin{minipage}[t]{0.5\linewidth}
\centering
\includegraphics[scale=0.25, bb =  0 0 775 279, clip]{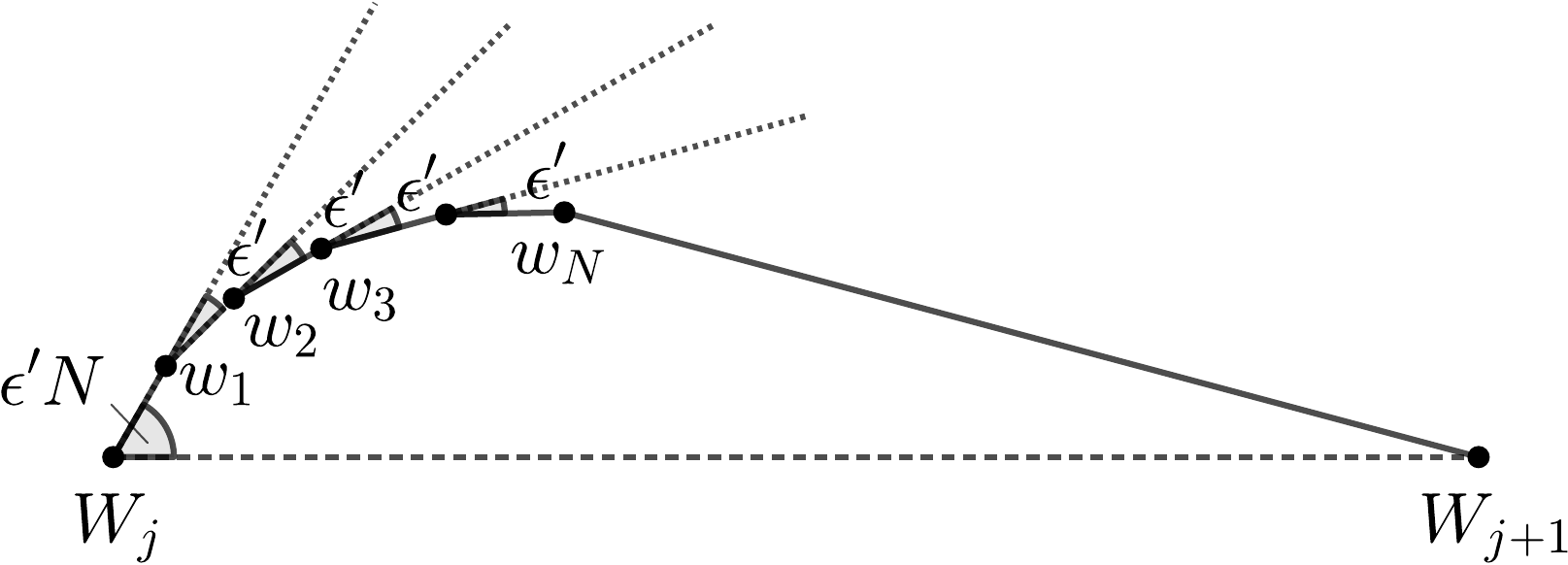} 
    \caption{Replacing the line segment $W_jW_{j+1}$ with a convex polygonal chain ($w_1,\dots,w_N$ are the vertices in ${\cal V}_j$)} 
\label{fig:polygonal_chain}
\end{minipage}
\end{figure}

\item[(Case 2)]  ${\cal T}$ satisfies the condition~(a), and there are exactly four angles $\varphi_i$ satisfying $\varphi_i \leq 120^\circ$.

In this case, at most one angle $\varphi_i$ satisfies $\varphi_i \leq 90^\circ$.
Suppose otherwise. Then, at least two angles $\varphi_j$ satisfy $\varphi_j \leq 90^\circ$, and these two angles satisfy $\varphi_j \leq 60^\circ$ by the result in Table~\ref{table:open_angle}.
Therefore, we have $\sum_{i=0}^{m-1}{\varphi_i} \leq  180^\circ (m-4) + 120^\circ \times 2 + 60^\circ \times 2 =  180^\circ (m-2)$, which is a contradiction.
We also remark that no angle $\varphi_i$ satisfies $\varphi_i =0$ because otherwise we have $\sum_{i=0}^{m-1}{\varphi_i} \leq  180^\circ (m-4) + 120^\circ \times 3 + 0 =  180^\circ (m-2)$, a contradiction. 
Without loss of generality, let $\psi_0$ be the smallest angle among the four angles $\psi_0,\dots,\psi_3$.
Then, we have $90^\circ < \psi_1, \psi_2, \psi_3 \leq 120^\circ$.
We consider the following three cases separately.

\item[(Case 2a)] $0^\circ < \psi_0 \leq 45^\circ$.

We consider a quadrilateral $\Gamma'_1$ with vertices $W_0,\dots,W_3$, labeled counterclockwise, satisfying
\begin{itemize}
\item  $\angle{W_0} = \psi_0 - \epsilon$, $\angle{W_1} = \psi_1 - \epsilon$, $\angle{W_2} = 360^\circ - \psi_0 -\psi_1 - \psi_3 + 3\epsilon$, $\angle{W_3} = \psi_3 - \epsilon$,
\item $\angle{W_0W_1W_3} = 90^\circ + \epsilon$, $\angle{W_0W_3W_1} = 90^\circ - \psi_0$, $\angle{W_2W_1W_3} = \psi_1- 90^\circ - 2\epsilon$, \\ $\angle{W_2W_3W_1} = \psi_0 + \psi_3 - 90^\circ - \epsilon$,
\end{itemize}
where $\epsilon > 0$ is a sufficiently small number. See Figure~\ref{fig:quad_type1}.
Note that all the angles appearing above are positive.
We first show that there is a greedy path between every pair of vertices of the quadrilateral $\Gamma'_1$.
First, observe that $\angle{W_2W_3W_1} < 45^\circ+120^\circ-90^\circ < 75^\circ$ and  $\angle{W_2} > 360^\circ-45^\circ-120^\circ-120^\circ = 75^\circ$,
and thus $\angle{W_2W_3W_1} < \angle{W_2}$. Similarly, we have $\angle{W_2W_1W_3} < 30^\circ < \angle{W_2}$.
It follows that $|W_1W_3| > |W_1W_2|$ and $|W_1W_3| > |W_2W_3|$.
Moreover, since $\angle{W_0W_1W_3} > 90^\circ$, we have $|W_0W_3| > |W_1W_3|$ and $|W_0W_3| > |W_0W_1|$.
We also observe that $\angle{W_0W_3W_1} > 45^\circ > \angle{W_0}$, which implies $|W_0W_1| > |W_1W_3|$.
Therefore, there is a greedy path between every pair of vertices of $\Gamma'_1$.
The greedy routes are summarized as follows:
\begin{itemize}
\item $W_0 \rightarrow W_3  \rightarrow W_2 \rightarrow W_1$,
\item $W_1 \rightarrow W_0$, $W_1  \rightarrow W_2 \rightarrow W_3$,
\item $W_2 \rightarrow W_1 \rightarrow W_0$, $W_2  \rightarrow W_3$,
\item $W_3 \rightarrow W_2 \rightarrow W_1$, $W_3  \rightarrow W_0$.
\end{itemize}
Then, we place the vertices in ${\cal V}_0, {\cal V}_1, {\cal V}_2, {\cal V}_3$ on the edges of $\Gamma'_1$ sufficiently close to $W_1, W_1, W_3, W_0$ respectively.
Since $\angle{W_3} > 90^\circ$ and $\angle{W_2W_3W_1} < 90^\circ$, the perpendicular bisector of  each edge drawn small on $W_2W_3$ passes through the interior of $W_0W_1$. 
This implies that move along each small edge on $W_2W_3$ toward $W_2$ decreases the distance to $W_1$ and $W_2$.
Continuing this discussion, one can prove that the constructed drawing of ${\cal C}$ is also greedy.
Next, we construct drawings of $T_{i_0},\dots,T_{i_3}$.
We first remark that the inequality $\sum_{j=0}^3{\psi_j} > 360^\circ$ implies $\psi_2 > 360^\circ -\psi_0 -\psi_1 - \psi_3$.
This leads to $\angle{W_2} < \psi_2$ because $\epsilon$ is sufficiently small.
Therefore, we have $\angle{W_j} < \psi_j$ for all $j=0,\dots,3$.
Thus, if we draw each of $T_{i_j}$ sufficiently small in such a way that $\angle{T_{i_j}}$ contains $\Gamma'_1$, the constructed drawing is greedy.
Finally, we replace the edges of $\Gamma'_1$ by convex polygonal chains, and draw subtrees $T_j$ with $|\angle{T_j}|_*=180^\circ$ as in Case~1.
Then, the constructed drawing is a greedy drawing of ${\cal T}$.

\item[(Case 2b)] $45^\circ< \psi_0  \leq 90^\circ$. 

We first remark that $\psi_0 \leq 60^\circ$ by the result in Table~\ref{table:open_angle}.
We consider a quadrilateral $\Gamma''_1$ with vertices $W_0,\dots,W_3$, labeled counterclockwise,  satisfying
\begin{itemize}
\item $\angle{W_0} = \psi_0 -\epsilon$,  $\angle{W_1} = \psi_1 -\epsilon$, $\angle{W_2} = 360^\circ - \psi_0 - \psi_1 - \psi_3 + 4\epsilon$, $\angle{W_3} = \psi_3 - 2\epsilon$, 
\item $\angle{W_0W_3W_1} = \psi_0 - 2\epsilon$, $\angle{W_0W_1W_3} = 180^\circ - 2\psi_0 + 3\epsilon$, $\angle{W_2W_3W_1} = \psi_3 - \psi_0$, \\ $\angle{W_2W_1W_3} = 2\psi_0 + \psi_1 - 180^\circ - 4\epsilon$,
\end{itemize}
where $\epsilon > 0$ is a sufficiently small number. See Figure~\ref{fig:quad_type2}.
Note that all the angles appearing above are positive.
We first show that there is a greedy path between every pair of vertices of the quadrilateral $\Gamma''_1$.
We first observe that $\angle{W_0} < 60^\circ, \angle{W_0W_3W_1} < 60^\circ, \angle{W_0W_1W_3} > 60^\circ$, which implies $|W_0W_1| < |W_0W_3|$ and $|W_1W_3| < |W_0W_3|$.
Since $\angle{W_2W_1W_3} < 2\times 60^\circ + 120^\circ -180^\circ = 60^\circ$ and $\angle{W_2} > 360^\circ - 120^\circ -120^\circ -60^\circ = 60^\circ$, we have  $\angle{W_2W_1W_3} > \angle{W_2}$.
This implies $|W_2W_3| < |W_1W_3|$.
Moreover, we have $\angle{W_1W_3W_0} < \angle{W_0}$, which leads to $|W_1W_3| > |W_1W_0|$.
Therefore, there is a greedy path between every pair of vertices of $\Gamma''_1$.
The greedy routes are summarized as follows:
\begin{itemize}
\item $W_0 \rightarrow W_1  \rightarrow W_2 \rightarrow W_3$,
\item $W_1 \rightarrow W_0$, $W_1  \rightarrow W_2 \rightarrow W_3$,
\item $W_2 \rightarrow W_1 \rightarrow W_0$, $W_2  \rightarrow W_3$,
\item $W_3 \rightarrow W_2$, $W_3  \rightarrow W_0 \rightarrow W_1$.
\end{itemize}
Then, we place the vertices in ${\cal V}_0, {\cal V}_1, {\cal V}_2, {\cal V}_3$ on the edges of $\Gamma''_1$ sufficiently close to $W_0, W_1, W_3, W_3$ respectively.
By a similar discussion to Case~2a, one can prove that the constructed drawing of ${\cal C}$ is also greedy.
Then, replacing each line segment by a convex polygonal chain, and drawing the remaining subtrees as in Case~2a, we obtain a greedy drawing of ${\cal T}$.
\item[(Case 2c)] $90^\circ < \psi_0 \leq 120^\circ$.

We consider a square $\Gamma''_1$ with vertices $W_0,\dots,W_3$, labeled counterclockwise.
Then, we place the vertices in ${\cal V}_0, {\cal V}_1, {\cal V}_2, {\cal V}_3$ on the edges of $\Gamma''_1$ sufficiently close to $W_0, W_1, W_2, W_3$ respectively.
See Figure~\ref{fig:quad_type3}.
Then, replacing each line segment by a convex polygonal chain, and drawing the remaining subtrees as in Case~2a, we obtain a greedy drawing of ${\cal T}$.

Therefore, we can always construct a greedy drawing of ${\cal T}$ in Case~2.

\begin{figure}[h]
\begin{minipage}[t]{0.33\linewidth}
\centering
\includegraphics[scale=0.2, bb =0 0 759 823, clip]{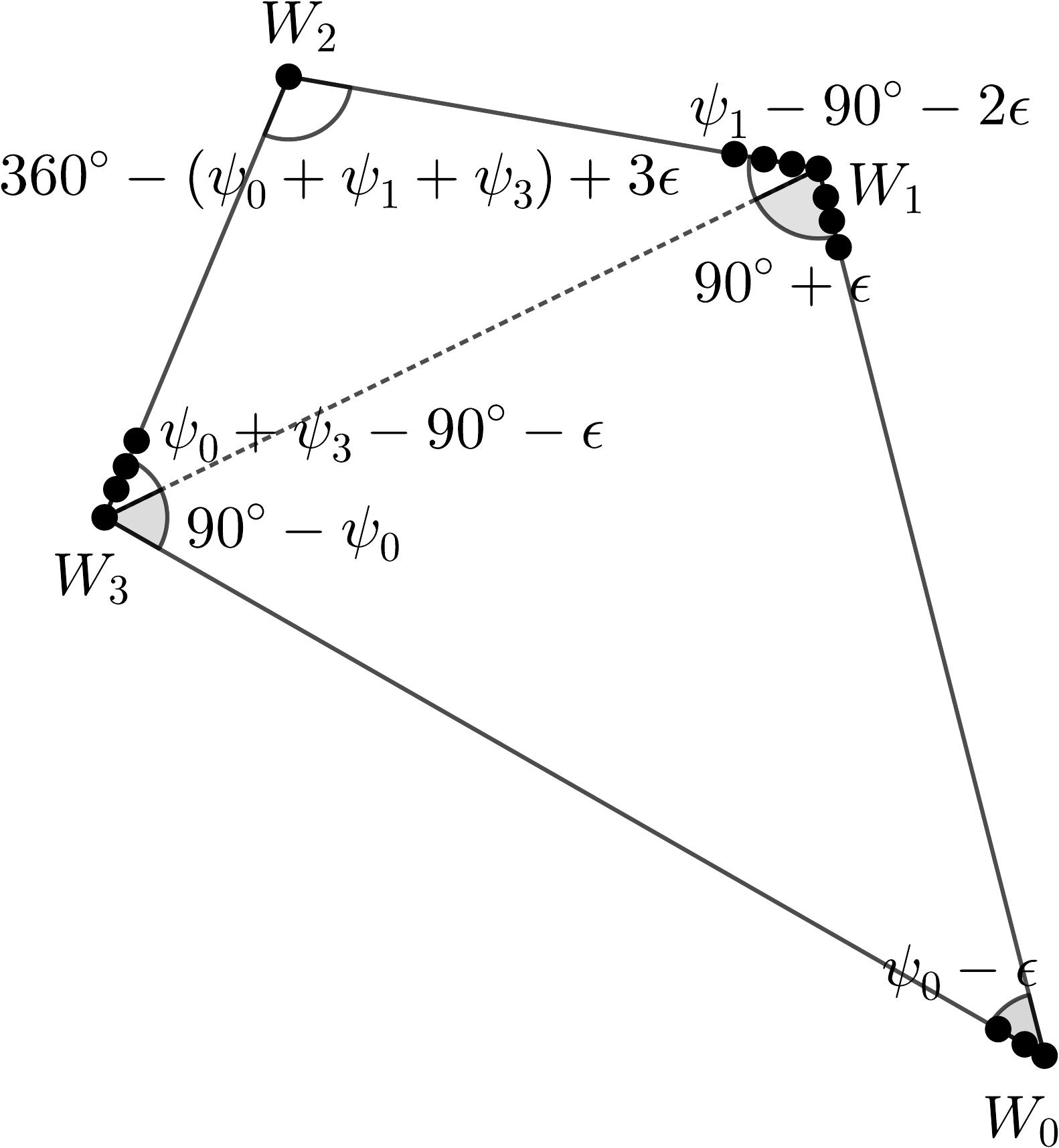} 
\subcaption{$0^\circ < \psi_0 < 45^\circ$} 
\label{fig:quad_type1}
\end{minipage}
\begin{minipage}[t]{0.33\linewidth}
\centering
\includegraphics[scale=0.2, bb = 0 0 785 788, clip]{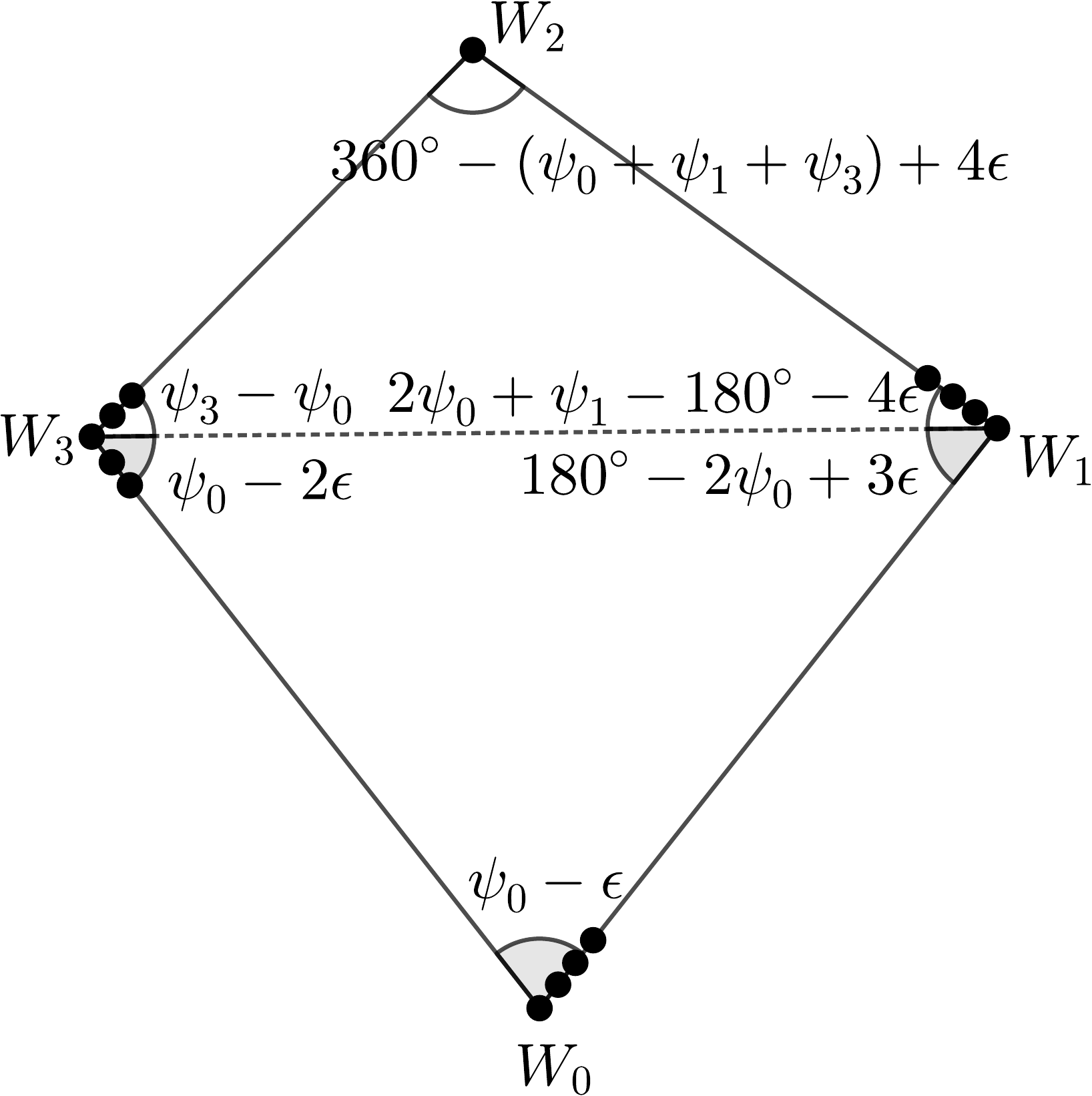} 
\subcaption{$45^\circ \leq \psi_0 \leq 60^\circ$} 
\label{fig:quad_type2}
\end{minipage}
\begin{minipage}[t]{0.33\linewidth}
\centering
\includegraphics[scale=0.2, bb =  0 0 441 483, clip]{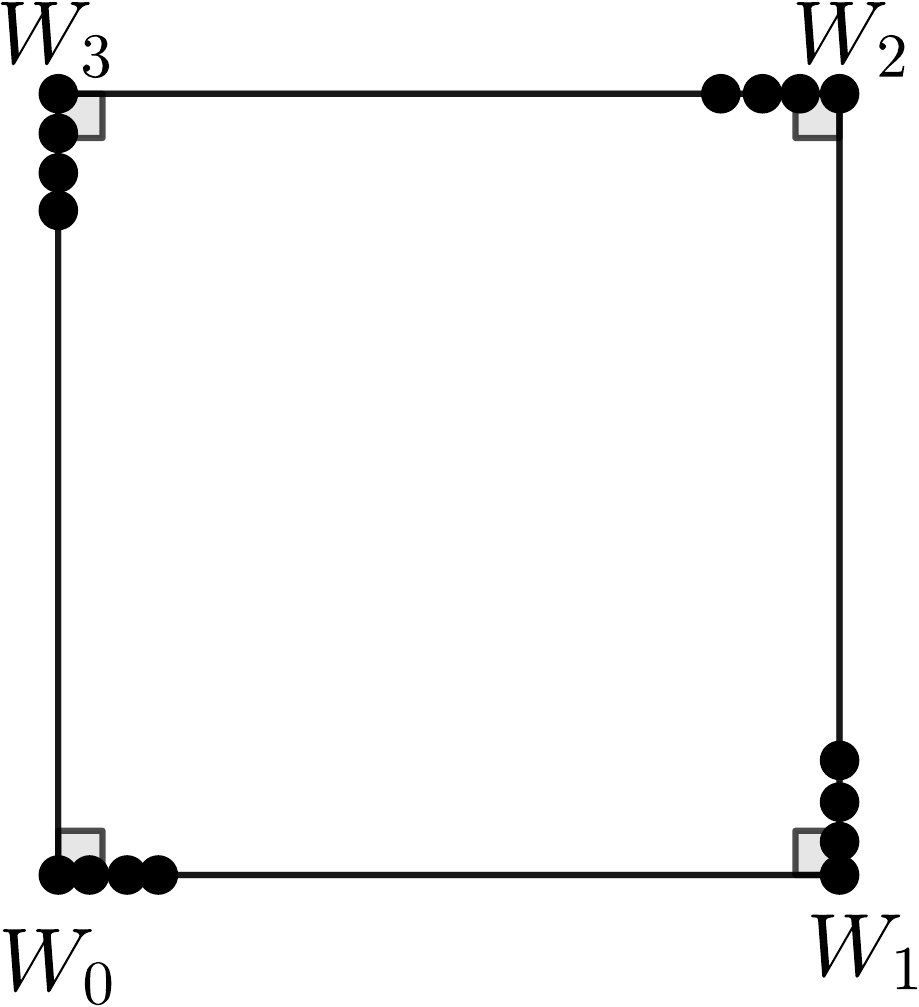} 
\subcaption{$90^\circ < \psi_0 \leq 120^\circ$}
\label{fig:quad_type3} 
\end{minipage}
    \caption{Greedy drawings of  ${\cal C}$ in Case~2} 
\label{fig:quadtilateral}
\end{figure}

\item[(Case 3)] There are at most three angles $\varphi_i$ satisfying $\varphi_i \leq 120^\circ$, and ${\cal T}$ satisfies the condition~(b).

Let $T$ be a $Y$-transformed tree of ${\cal T}$.
Let $v_a,v_b,v_c$ be the connection vertices  and  $w$ the center vertex.
Let $\Gamma_T$ be a greedy drawing of $T$.
By Lemma~\ref{lem:open_closed}, at most one of $T_i$, for $i \in \{ a,b,c \}$, satisfies $|\angle{T_i}| \leq 0$ in $\Gamma_T$.
Let us first assume that there is a subtree $T_i$ with $|\angle{T_i}| \leq 0$, and
suppose without loss of generality that $T_c$ is such a subtree.
Slightly perturbating if necessary, we assume $|\angle{T_c}| < 0$.
Applying the shrinking lemma (Lemma~\ref{lem:shrinking}), we transform the drawing into a drawing in which $T_a$ and $T_b$ are infinitesimally small.
Let $p$ be a vertex of $\mathrm{polygon}(T_{c})$ that belongs to the same connected component of $\mathrm{polygon}(T_{c}) \setminus T_c$ as $T \setminus T_c$.
Let $l_1$ and $l_2$ be the supporting lines of the two edges of $\mathrm{polygon}(T_{c})$ incident to $p$.
We move $w,v_a,v_b$ into places in $\mathrm{polygon}(T_c)$ that are sufficiently close to $p$, and translate $T_a$ and $T_b$ accordingly.
See Figure~\ref{fig:redrawing}.
Then, each of $\angle{T_a}$ and $\angle{T_b}$ contains the original angle, and thus contains $T_{c}$.
We also note that the transformation keeps planarity.
Let $A,B,C$ be the geometric points corresponding to $v_a,v_b,v_c$ respectively.
Then, the triangle $ABC$ formed by $A,B,C$ has a property that $|AB|$ is sufficiently small. 
We construct a greedy drawing of ${\cal T}$ by placing the remaining vertices of ${\cal C}$ on the edges of the triangle $ABC$, and placing the remaining subtrees.
Without loss of generality, we suppose that $\angle{B} > \angle{A}$.
Let us consider the following two cases separately.

\item[(Case 3a)]  $\angle{B} < 90^\circ$.

In this case, $\angle{B}$ is sufficiently close to $90^\circ$.
Let $Q$ be the point on the line segment $AB$ with $|CQ| = |CB|$, and $R$ the point on the line segment $BC$ with $|AR| = |AB|$.
Note that $|BR|$ is sufficiently small, compared to $|AB|$.
We place each vertex $v$ in ${\cal V}_0$ on the line segment $AQ$ sufficiently close to $A$,
 each vertex $v$ in  ${\cal V}_1$ on the line segment $CR$ sufficiently close to $R$,
and each vertex $v$ in ${\cal V}_2$ on the line segment $CA$ sufficiently close to $C$. See Figure~\ref{fig:triangle1}.
Then, it is easy to check that the constructed drawing of ${\cal C}$ is greedy.
Replacing the line segments by convex polygonal chains and placing the remaining subtrees as in Cases~1 and 2, we obtain a greedy drawing of ${\cal T}$.

\item[(Case 3b)]  $\angle{B} \geq 90^\circ$.

We place each vertex in ${\cal V}_0$ on the line segment $AB$ sufficiently close to $A$, 
each vertex in ${\cal V}_1$ on the line segment $BC$ sufficiently close to $B$, and
each vertex in ${\cal V}_2$ on the line segment $CA$ sufficiently close to $C$.
Then, the constructed drawing of ${\cal C}$ is clearly greedy.
Replacing the line segments by convex polygonal chains and placing the remaining subtrees as in Cases~1 and 2, we obtain a greedy drawing of ${\cal T}$.
\end{description}
Therefore, in either case, ${\cal T}$ has a greedy drawing.
Finally, we assume that $|\angle{T_i}| > 0$ for all $i = a,b,c$.
This implies that $\psi_a, \psi_b, \psi_c > 0$.
Without loss of generality, we assume that $\psi_a$ is the largest angle among these three angles.
Then, we have $\psi_a > 90^\circ$ by the result  in Table~\ref{table:open_angle}.
We consider a triangle with vertices $Z_a,Z_b,Z_c$ satisfying the following two conditions:
\begin{itemize}
\item $\angle{Z_i} < \psi_i$ for $i=a,b,c$,
\item $\angle{Z_a} = \psi_a -\epsilon$,
\end{itemize}
for sufficiently small $\epsilon > 0$.
Since $\psi_a + \psi_b + \psi_c > 180^\circ$, such a triangle clearly exists.
Then, one can construct a greedy drawing of ${\cal T}$ similarly to Case~3b.
Therefore, we can always construct a greedy drawing of ${\cal T}$ in Case~3.
\end{proof}
\begin{figure}[h]
\centering
\includegraphics[scale=0.5, bb =0 0 752 278, clip]{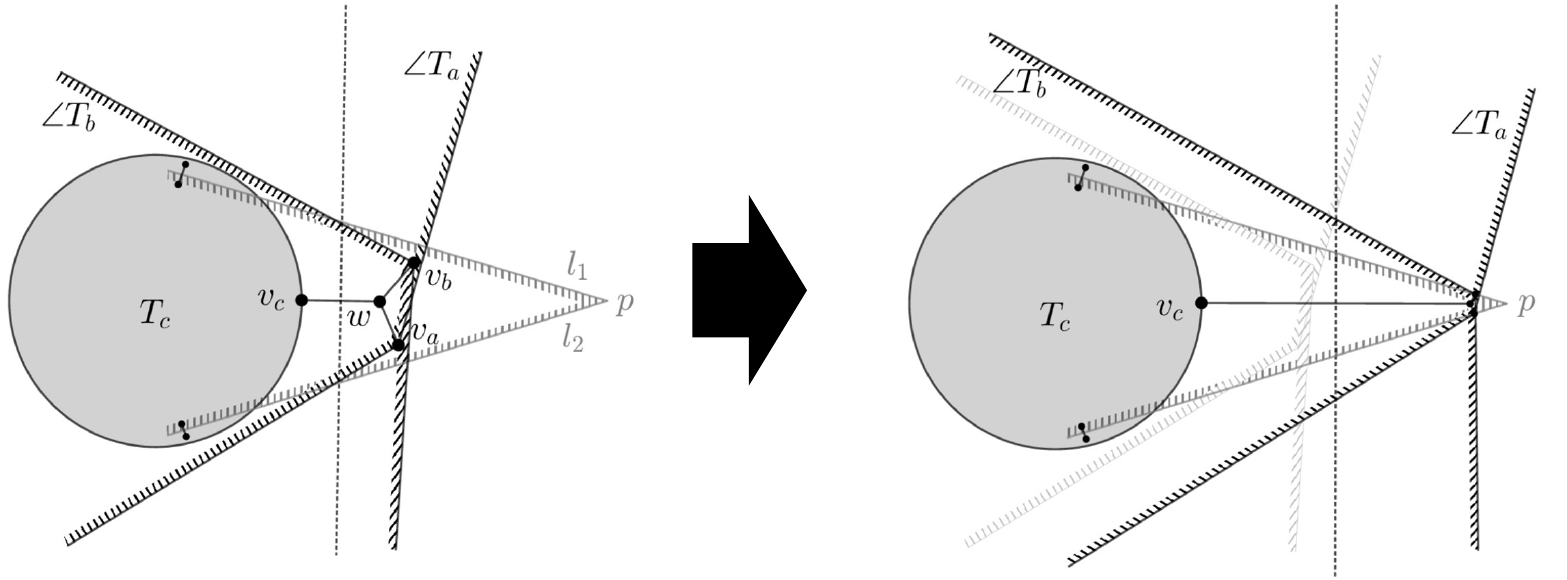} 
    \caption{Transforming a greedy drawing of $T$} 
\label{fig:redrawing}
\end{figure}

\begin{figure}[h]
\begin{minipage}[t]{0.5\linewidth}
\centering
\includegraphics[scale=0.18, bb= 0 0 1022 425, clip]{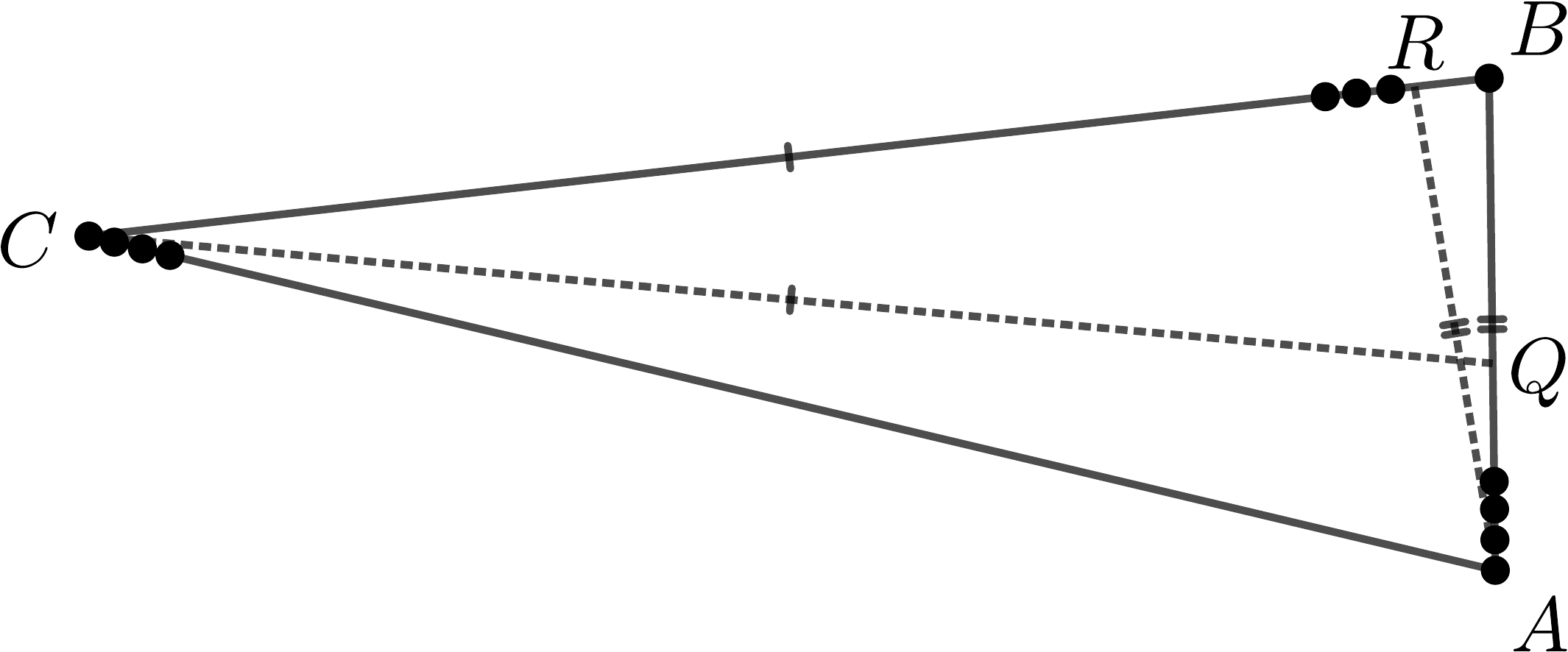} 
\subcaption{$\angle{B} < 90^\circ$} 
\label{fig:triangle1}
\end{minipage}
\begin{minipage}[t]{0.5\linewidth}
\centering
\includegraphics[scale=0.18, bb =  0 0 1011 325, clip]{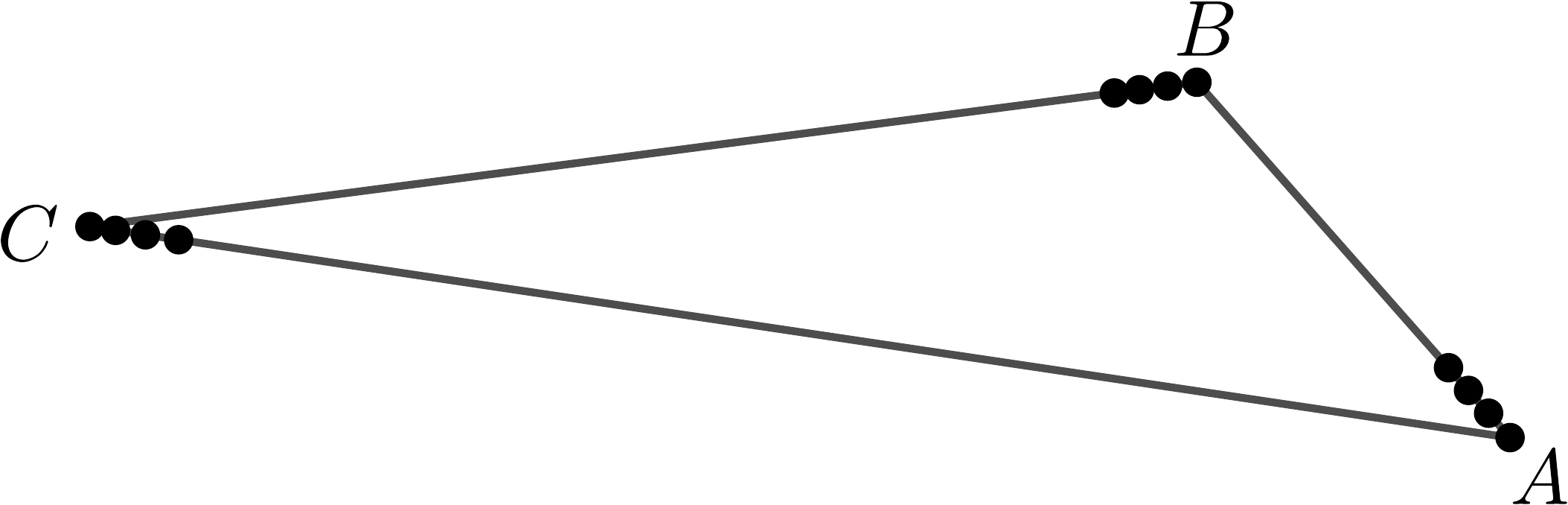} 
\subcaption{$\angle{B} \geq 90^\circ$} 
\label{fig:triangle2}
\end{minipage}
    \caption{Greedy drawings of ${\cal C}$ (Case 3)} 
\end{figure}

\newpage

\section{Concluding remarks}
In this paper, we have presented a complete combinatorial characterization of greedy-drawable trees (Proposition~\ref{prop:deg4} and Theorem~\ref{thm:main}).
Our characterization immediately leads to a linear-time recognition algorithm for greedy-drawable trees.
That is, one can determine the greedy-drawability of a tree with maximum degree $5$ by checking whether the suprema of opening angles of the subtrees around a degree-$5$ vertex  
are in the ranges of angles listed in Tables~\ref{table:degree5_case1} and \ref{table:degree5_case3}.
Since the supremum of opening angles of a tree can be computed in linear time by the algorithm \emph{getOpenAngle} in~\cite{NP17}, this condition can likewise be verified in linear time.
Combining this  algorithm with the linear-time recognition algorithm for greedy-drawable trees with maximum degree $\leq 4$ by N\"{o}llenburg and Prutkin~\cite{NP17}, we obtain
a linear-time recognition algorithm for greedy-drawable trees in the general case.
\begin{cor}
\label{cor:tree}
Greedy drawability of a tree $T=(V,E)$ can be checked in $O(|V|)$ time.
\end{cor}
As a subsequent step, we have investigated a characterization of greedy-drawable pseudo-trees.
Although we did not present an explicit description of greedy-drawable pseudo-trees, such a description can easily be obtained by determining the possible angle types of subtrees using Table~\ref{table:open_angle}.
Similarly to the case of trees, we obtain the following corollary:
\begin{cor}
\label{cor:pseudo_tree}
Greedy drawability of a pseudo-tree ${\cal T}=(V,E)$ can be checked  in $O(|V|)$ time.
\end{cor}

\section*{Acknowledgement}
This work was supported by JSPS KAKENHI Grant Number JP19K20210.

\end{document}